\documentclass[11pt, oneside]{amsart} 

\usepackage[pdftex]{graphicx}
\usepackage[rgb]{xcolor}
\usepackage{geometry}

\usepackage{amsfonts}
\usepackage{amsmath}
\usepackage{amssymb}
\usepackage{amsthm}
\usepackage{algorithm2e}
\usepackage{mathtools}
\usepackage{tikz}
\usepackage{tikz-3dplot}
\usetikzlibrary{patterns}
\usepackage{appendix}
\usepackage{float}

\usepackage{paralist}
\usepackage[colorlinks,cite color=blue,pagebackref=true,pdftex]{hyperref}

\usepackage[T1]{fontenc}
\usepackage{caption}
\usepackage{subcaption}
\usepackage{lscape} %

\makeatletter
\ifcase \@ptsize \relax%
  \newcommand{\miniscule}{\@setfontsize\miniscule{4}{5}}%
\or%
  \newcommand{\miniscule}{\@setfontsize\miniscule{5}{6}}%
\or%
  \newcommand{\miniscule}{\@setfontsize\miniscule{5}{6}}%
\fi
\makeatother

\newcommand\Defn[1]{\textbf{\color{black}#1}}
\newcommand\Def[1]{\Defn{#1}}

\renewcommand\emptyset{\varnothing}
\newcommand\Z{\mathbb{Z}}

\newcommand\Q{\mathbb{Q}}
\newcommand\R{\mathbb{R}}

\newcommand\Rpp{\R_{>0}}
\newcommand\Ri{\R^\times}
\newcommand\C{\mathbb{C}}

\newcommand\inner[1]{\langle {#1} \rangle}
\newcommand\innerprod{\inner{\cdot, \cdot}}
\newcommand\defeq{\coloneqq}
\newcommand\eqdef{\eqqcolon}
\newcommand\0{0}

\newcommand\Fan{\mathcal{N}}%
\newcommand\GL{\mathrm{GL}}%

\newcommand\TypeSpc{\mathcal{T}}%
\newcommand\TypeCone{\TypeSpc_+}%

\newcommand\InSpc{\mathcal{I}}%
\newcommand\InCone{\InSpc_+}%
\newcommand\ZInSpc{\mathcal{Z}}%
\newcommand\ZInCone{\ZInSpc_+}%

\newcommand\Arr{\mathcal{A}}%
\newcommand\Centers{M}%

\newcommand\D{\mathcal{D}}%

\newcommand\PL{\mathrm{PL}}%
\newcommand{\ReflG}{W}

\definecolor{darkred}{rgb}{0.8,0,0}
\definecolor{darkblue}{rgb}{0,0,0.8}

\DeclareMathOperator{\sgn}{sgn}

\DeclareMathOperator{\interior}{int}
\DeclareMathOperator{\aff}{aff}
\DeclareMathOperator{\lin}{lin}
\DeclareMathOperator{\lineal}{lineal}
\DeclareMathOperator{\conv}{conv}
\DeclareMathOperator{\cone}{cone}

\DeclareMathOperator{\rk}{rk}

\newcommand{\ess}{\operatorname{ess}}%

\newtheorem{thm}{Theorem}[section]
\newtheorem{cor}[thm]{Corollary}
\newtheorem{lem}[thm]{Lemma}
\newtheorem{prop}[thm]{Proposition}
\newtheorem{conj}[thm]{Conjecture}
\newtheorem{quest}[thm]{Question}

\theoremstyle{definition}

\newtheorem{example}[thm]{Example}
\newtheorem{rem}[thm]{Remark}

\title[Inscribable fans II]{Inscribable Fans II: Inscribed zonotopes, simplicial arrangements, and reflection groups}

\author{Sebastian Manecke} 
\author{Raman Sanyal}

\address{Institut f\"ur Mathematik, Goethe-Universit\"at Frankfurt, Germany} 
\email{manecke@math.uni-frankfurt.de}
\email{sanyal@math.uni-frankfurt.de}



\date{\today}

\binoppenalty=\maxdimen
\relpenalty=\maxdimen
\begin{document}

\begin{abstract}
    An arrangement of hyperplanes is \emph{strongly inscribable} if it has an
    inscribed (or ideal hyperbolic) zonotope. We characterize inscribed
    zonotopes and prove that the family of strongly inscribable arrangements
    is closed under restriction and localization. Moreover, we show that
    (strongly) inscribable arrangements are simplicial. We conjecture that
    only reflection arrangements and their restrictions are strongly
    inscribable and we verify our conjecture in rank-$3$ using the
    conjecturally complete list of irreducible simplicial rank-$3$
    arrangements.
\end{abstract}

\maketitle

\newcommand{\flats}{\mathcal{L}}%
\newcommand{\faces}{\mathcal{F}}%

\section{Introduction}\label{sec:intro}

A convex polytope $P \subset \R^d$ is \Def{inscribed} if its vertices lie on a
common sphere. A polytope $P$ is \Def{inscribable} if there is an inscribed
polytope $P'$ that is combinatorially equivalent to $P$. The question which
combinatorial types of $3$-polytopes are inscribable was raised by
Steiner~\cite{Steiner} and settled by Steinitz~\cite{Steinitz} and
Rivin~\cite{Rivin}. In stark contrast, our understanding of the inscribability
problem in dimensions four and up is rather exiguous~\cite{PadrolZiegler,
Doolittle}.  In~\cite{InFan1}, we replaced \emph{combinatorial equivalence}
with the discrete-geometric condition of \emph{normal equivalence}. A polytope
$P$ is \Def{normally inscribable} (or \Def{strongly isomorphic}) if there is
an inscribed polytope $P'$ normally equivalent to $P$. We showed that the
collection $\InCone(P)$ of (translation-classes of) inscribed polytopes
normally equivalent to $P$ has the structure of an open polyhedral cone with
respect to Minkowski addition.  This brought to light a number of remarkable
structural and algorithmic properties; cf.~Section~\ref{sec:insc_cones}.
Polytopes $P$ and $P'$ are normally equivalent if and only if they have the
same normal fan $\Fan$.  A fan $\Fan$ is \Def{inscribable} if there is an
inscribed polytope $P$ with normal fan $\Fan(P) = \Fan$. The simplest normal
fans are induced by arrangements of linear hyperplanes. Due to strong
connections to convex geometry and algebra, the geometry and combinatorics of
arrangements of hyperplanes is a very active area of research.  In this paper
we study \Def{inscribable arrangements}, that is, hyperplane arrangements
whose associated fan is inscribable. 

Let $\Arr = \{H_1,\dots,H_n\}$ be an arrangement of linear hyperplanes in
$\R^d$.  The decomposition $\R^d \setminus \bigcup_i H_i$ into open polyhedral
cones, called regions, induces a fan that we will also denote by $\Arr$.
Bolker~\cite{Bolker} calls a polytope $P$ with normal fan $\Arr$ a \Def{belt
polytope}. The most prominent belt polytopes are \Def{zonotopes}: If $H_i =
\{x \in \R^d : \inner{z_i,x} = 0 \}$ for $i=1,\dots,n$, then for any $\lambda
\in \R^n_{> 0}$ the Minkowski sum of segments
\[
    Z_\lambda \ \defeq \ \lambda_1 [-z_1,z_1] + \lambda_2 [-z_2,z_2] + \cdots
    + \lambda_n [-z_n,z_n]  
\]
is a polytope with $\Fan(Z) = \Arr$. Zonotopes are distinguished among belt
polytopes by many favorable geometric and combinatorial properties.  We show
that inscribed zonotopes also stand out among inscribed polytopes. For an
inscribed polytope $P$, we write $c(P) \in \aff(P)$ for the \Def{center} of
the inscribing sphere relative to its affine hull. For a segment $e \subset
\R^d$, we let $\Centers_e$ be the hyperplane of points equidistant to the
endpoints of $e$ and we write $\pi_e : \R^d \to \Centers_e$ for the orthogonal
projection onto $\Centers_e$.
\begin{thm} \label{thm:inscribed_zono}
    For an inscribed polytope $P \subset \R^d$ the following are equivalent:
    \begin{enumerate}[\rm (i)]
        \item $P$ is an inscribed zonotope;
        \item The projection $\pi_e(P)$ is inscribed with center $c(P)$ for
            every edge $e$ of $P$;
        \item The section $P \cap \Centers_e$ is inscribed with center $c(P)$
            for every edge $e$ of $P$.
    \end{enumerate}
\end{thm}

It is quite unusual that projections of inscribed polytopes are inscribed.

We call an arrangement $\Arr$ \Def{strongly inscribable} if $\Arr$ has an
inscribed zonotope (as opposed to an inscribed belt polytope).
Theorem~\ref{thm:inscribed_zono} implies that strongly inscribable
arrangements constitute a structurally interesting class of hyperplane
arrangements.

\begin{thm}\label{thm:restr_insc}
    Let $\Arr$ be a strongly inscribable arrangement and $L$ a flat of $\Arr$.
    Then the restriction $\Arr^L$ and the localization $\Arr_L$ are strongly
    inscribable arrangements.
\end{thm}

A finite group $\ReflG \subset \GL(\R^d)$ is a (finite) \Def{reflection group}
if
it is generated by reflections in linear hyperplanes. The collection of
reflecting hyperplanes is the \Def{reflection arrangement} $\Arr(\ReflG)$ of
$\ReflG$. We show that reflection arrangements are
paragons of strongly inscribable arrangements.

\begin{prop}
    Reflection arrangements and their restrictions are strongly inscribable.
\end{prop}

Note that while localizations of reflection arrangements are reflection
arrangements, this does not hold for restrictions; cf.~\cite[Example
6.83]{OrlikTerao}.

In~\cite{InFan1}, we showed that verifying whether a fan $\Fan$ is inscribable
can be reduced to a linear programming feasibility problem. However, the
feasibility problem requires complete knowledge of the fan.  Quite remarkably,
we give a simple procedure to test if a hyperplane arrangement is strongly
inscribable: the linear programming feasibility problem depends on the
$2$-dimensional flats spanned by the vector configuration $z_1,\dots,z_n$; see
Theorem~\ref{thm:compute_by_2flats} and Remark~\ref{rem:algorithm}.

A hyperplane arrangement $\Arr$ is \Def{simplicial} if every region of $\Arr$
is a simplicial cone. The class of simplicial arrangements is closed with
respect to restrictions and localizations. In particular reflection
arrangements yield prime examples of simplicial arrangements.

\begin{thm}\label{thm:insc_arr_are_simplicial}
    If $\Arr$ is inscribable, then $\Arr$ is simplicial. Equivalently,
    inscribed belt polytopes are simple.
\end{thm}

Simplicial arrangements are fascinating but rare. A collection of two infinite
families and $90$ sporadic examples of simplicial arrangements of rank $3$ up
to projective transformation was described by
Gr\"unbaum~\cite{Grunbaum_Simplicial}.
Cuntz~\cite{Cuntz_27lines,Cuntz_Greedy} contributed five further examples. It
is conjectured that the Gr\"unbaum--Cuntz catalog is complete up to
combinatorial isomorphism; see also Section~\ref{sec:projectively_unique}.
Using techniques from computational algebra, we show the following.

\begin{thm}
    The only strongly inscribable arrangement in the Gr\"unbaum--Cuntz catalog
    are restrictions of reflection arrangements.
\end{thm}

Assuming that the Gr\"unbaum--Cuntz catalog contains all essential
simplicial rank-$3$ arrangements up to combinatorial isomorphism, we
show the following.

\begin{thm}
    If the Gr\"unbaum--Cuntz catalog is complete, then for each $d \ge 3$,
    there exist only finitely many irreducible strongly inscribable
    arrangements of rank $d$ up to combinatorial isomorphism.
\end{thm}

To show these claims, we computationally re-establish that all
simplicial arrangements in the Gr\"unbaum--Cuntz catalog are projectively unique,
that is, every combinatorial isomorphism stems from a linearly
isomorphism. We also show that under the assumption that the
catalog is complete every simplicial arrangement of
rank $d \ge 3$ is projectively unique. This observation might be of
independent interest.

Further computations in higher rank fuel our main conjecture.
\begin{conj}\label{conj:main}
    Every strongly inscribable arrangement of rank $d \ge 3$ is
    combinatorially isomorphic to the restriction of a reflection arrangement.
\end{conj}

Inscribed polytopes correspond to \emph{ideal hyperbolic polytopes}, that is,
polytopes in hyperbolic space with all vertices at infinity.
Conjecture~\ref{conj:main} gives a conjectural classification of ideal
hyperbolic zonotopes.

The paper is organized as follows: 
In Section~\ref{sec:insc_cones}, we recap results on inscribable fans and
hyperplane arrangements. In particular, we recall that the collection of
translation-classes of inscribed polytopes with normal fan $\Fan$ is a
relatively open polyhedral subcone of the type cone of $\Fan$ and we discuss
inscribed \emph{virtual} polytopes.
We also show
that reflection arrangements are strongly inscribable.  In
Section~\ref{sec:dim2}, we give geometric and algebraic characterizations of
(strongly) inscribable line arrangements. As in the case of reflection groups,
the rank-$2$ situation serves as building blocks for higher ranks. In
Section~\ref{sec:rest_simp}, we investigate operations on (strongly) inscribed
arrangements. We show that products, localizations, and restrictions retain
(strong) inscribability. In Section~\ref{sec:repr} we derive a polyhedral
representation of the cone of inscribed zonotopes of a fixed arrangement
$\Arr$. In Section~\ref{sec:simplicial} we show that inscribable arrangements
are simplicial and we discuss projective uniqueness of simplicial
arrangements. We extend a local characterization of zonotopes due to Bolker to
inscribed zonotopes (Theorem~\ref{thm:bolker_inscr}).  In
Section~\ref{sec:inner_product}, we adopt a broader perspective and
investigate arrangements inscribed in a general quadric. This allows us to
treat all arrangements from the Gr\"unbaum--Cuntz catalog.  We first show that
there are only finitely many strongly inscribed arrangements in the two
infinite families (Section~\ref{sec:infinite_families}). The remaining
finitely many arrangements are treated using techniques from computational
algebra in Section~\ref{sec:sporadic}. In Section~\ref{sec:ABD}, we complete
the classification of all quadrics for which restrictions of reflection
arrangements are inscribable. We close in Section~\ref{sec:close} with two
remarks and pretty pictures.

\subsection*{Acknowledgements} 
We thank Michael Cuntz, Thilo R\"orig, Christian Stump, and Martin Winter for
insightful discussions. Research that led to this paper was supported by the
DFG-Collaborative Research Center, TRR 109 ``Discretization in Geometry and
Dynamics'' and we also thank our colleagues of project A3 for their support.
Many of our findings were inspired by experiments and computations conducted
with SAGE~\cite{SAGE} and GeoGebra~\cite{geogebra}.

\renewcommand{\baselinestretch}{0.7}\normalsize
\tableofcontents
\renewcommand{\baselinestretch}{1.0}\normalsize

\section{Inscribed cones of hyperplane arrangements}\label{sec:insc_cones}
In this section, we recall background and central results on inscribable fans
(following~\cite{InFan1}) and hyperplane arrangements
(following~\cite{OrlikTerao}). Throughout our ambient space is $\R^d$
equipped with a fixed inner product $\innerprod$.

 A complete \Def{fan} is a collection $\Fan$ of polyhedral cones in $\R^d$
 such that 
\begin{enumerate}[\rm (F1)]
    \item if $C \in \Fan$ and $F \subseteq C$ is face, then $F \in \Fan$;
    \item if $C, C' \in \Fan$, then $C \cap C'$ is a face of both;
    \item $\bigcup_{C \in \Fan} C = \R^d$.
\end{enumerate}
The inclusion-maximal cones of $\Fan$ are all of dimension $d$ and are called
the \Defn{regions} of $\Fan$ and we will usually identify $\Fan$ with its set
of regions.

Let $P \subset \R^d$ be a convex polytope with vertex set $V(P)$.  For $c \in
\R^d$, we write
\[
    P^c \ \defeq \ \{ x \in P : \inner{c,x} \ge \inner{c,y} \text{ for all } y
    \in P \} 
\]
for the non-empty face that maximizes the linear function $x \mapsto
\inner{c,x}$.  The set of all non-empty faces of $P$ is denoted by
$\faces(P)$. This is a graded lattice with respect by inclusion. The
\Def{normal cone} of a face $F \in \faces(P)$ is 
\[
    N_F P \ \defeq \ \{ c \in \R^d : F \subseteq P^c \}
\]
and the \Def{normal fan} of $P$ is $\Fan(P) \defeq \{ N_FP : F \in \faces(P)
\}$. The regions of $\Fan(P)$ correspond to normal cones of vertices
\[
    N_v P \ = \ \{ c \in \R^d : \inner{c,v} \ge \inner{c,u} \text{ for all } u
    \in V(P) \} \, .
\]
We call a fan \Def{polytopal}, if it is the normal fan of a polytope. Two
polytopes $P,P'$ are called \Def{normally equivalent} if $\Fan(P) = \Fan(P')$.
In~\cite{InFan1} we studied the question when for a polytope $P$ there is some
inscribed $P'$ normally equivalent to $P$. Since this depends only on the
normal fan, we call a polytopal fan $\Fan$ \Def{inscribable} if there is an
inscribed polytope $P$ with $\Fan(P) = \Fan$.

Let $T_d \cong \R^d$ be the group of translations. We defined the
\Def{inscribed cone} of $\Fan$ as
\[
    \InCone(\Fan) \defeq \{ P \subset \R^d \text{ inscribed polytope} :
    \Fan(P) = \Fan \} \,/\, T_d\,.
\]
The name is justified by the following fundamental result.
\begin{thm}[{\cite[Theorem~1.1]{InFan1}}]\label{thm:inspc}
    Let $\Fan$ be a polytopal fan in $\R^d$. Then $\InCone(\Fan)$ is closed
    with respect to Minkowski addition and has the structure of an open
    polyhedral cone of dimension $\le d$.
\end{thm}

It turns out that there is a simple embedding of $\InCone(\Fan)$ into $\R^d$.
For an inscribed polytope $P$, let $c(P) \in \aff(P)$ be the
\Def{center} of the inscribing sphere relative to the affine hull of $P$.
\begin{thm}[{\cite[Corollary~2.8]{InFan1}}]\label{thm:inspc_is_cone}
    Let $\Fan$ be an inscribable fan and let $R_0 \in \Fan$ be a
    region. The map $v_{R_0} : \InCone(\Fan) \to \interior R_0$ with
    $v_{R_0}(P) \defeq v$, where 
    \[
        \{v\} \ = \ V(P - c(P)) \cap \interior R_0\,,
    \]
    is linear and injective. The image of $v_{R_0}$ is called the \Defn{based
    inscribed cone} $\InCone(\Fan, R_0)$ of $\Fan$.
\end{thm}

In this paper, we focus on the class of fans induced by linear
hyperplane arrangements. A linear \Def{hyperplane} is of the form
$H = z^\perp = \{ x : \inner{z,x} = 0 \}$ for some
$z \in \R^d \setminus \{\0\}$. An \Def{arrangement} of linear
hyperplanes is a collection $\Arr = \{ H_1,\dots, H_n \}$ where
$H_i = z_i^\perp$ for $i=1,\dots,n$. For a generic $c \in \R^d$, let
$\sigma = (\sigma_1,\dots,\sigma_n) \in \{-1,+1\}^n$ with
$\sigma_i = \sgn( \inner{z_i,c})$. Then $c$ is contained in the
interior of the cone
\[
    R_\sigma \ \defeq \ \{ x : \sigma_i\inner{z_i,x} \ge 0 \text{ for all }
    i=1,\dots,n \}.
\]
The collection of such cones $R_\sigma$ defines a fan $\Fan(\Arr)$ induced
by $\Arr$. It is the closure of connected components of $\R^d \setminus
\bigcup \Arr$. Since $\Arr$ is uniquely determined by $\Fan(\Arr)$, we do
not distinguish between $\Arr$ and its fan.  The \Defn{lineality space} of
an arrangement $\Arr$ is $\lineal(\Arr) \defeq \bigcap_{H \in \Arr} H$ and
we call $\Arr$ \Defn{essential}, if $\lineal(\Arr) = \{\0\}$. The
\Defn{rank} of $\Arr$ is $\rk(\Arr) \defeq d - \dim(\lineal(\Arr))$.  The
\Defn{essentialization} $\ess(\Arr) \defeq \{H \cap \lineal(\Arr)^\perp : H
\in \Arr\}$ is a hyperplane arrangement in $\lineal(\Arr)^\perp$ and, by passing
to the essentialization, we often will assume that $\Arr$ is essential.

Every fan induced by a hyperplane arrangement is polytopal. The
\Def{zonotope} associated to $z_1,\dots,z_n$ is the polytope
\[
    Z \ = \ [-z_1, z_1] + [-z_2, z_2] + \dots + [-z_n, z_n]
\]
and it is straightforward to verify that $R_\sigma$ is the normal cone of the
vertex $\sum_i \sigma_i z_i$ of $Z$.

A polytope $P$ with $\Fan(P) = \Arr$ is called a \Def{belt
polytope}~\cite{Bolker}. The name derives from the following fact.  For a set
$S \subset \R^d$ denote by $\aff(S)$ its affine hull and by $\aff_0(S) \defeq
\aff(S) - S$ the linear space parallel to it.  Two faces $F, F'$ of a belt
polytope $P$ are normally equivalent if and only if $\aff_0(F) = \aff_0(F')$.
The collections of faces that are normally equivalent are the \Def{belts} of
$P$. Moreover, if $L = \aff_0(F)^\perp$, then the faces normally equivalent to
$F$ are in bijection to the regions of the \Def{restriction}
\[
    \Arr_L \ \defeq \ \{ H \cap L : H \in \Arr, L \not\subseteq H\} \, ,
\]
which is a hyperplane arrangement in $L$. If we denote the \Def{localization}
of $\Arr$ at $L$ by 
\[
    \Arr^L \ \defeq \ \{ H \in \Arr : L \subseteq H \}
\]
then one checks that $F$ is a belt polytope with respect to $\Arr^L$. Such a
subspace $L$, which is an intersection of hyperplanes in $\Arr$, is called a
\Def{flat}.  The \Def{lattice of flats} $\flats(\Arr)$ is the collection of
flats of $\Arr$ partially ordered by reverse inclusion.  This is a (graded)
lattice with minimal element $\R^d$ and maximal element $\lineal(\Arr)$. We
denote by $\flats_k(\Arr)$ the subset of $k$-dimensional flats (or $k$-flats,
for short) and by $\faces_L(P) \subseteq \faces(P)$ the collection of faces in
the belt determined to $L$.

The upcoming local characterization of belt polytopes and zonotopes is
essentially due to Bolker~\cite[Thm.~3.3]{Bolker}.
\begin{thm}\label{thm:charac_zono}
    A polytope $P$ is a belt polytope if and only if for every $2$-dimensional
    face $F \subseteq P$ the following holds: $F$ has an even number of
    vertices and combinatorially antipodal edges are parallel.
    A polytope $P$ is a zonotope if and only if all $2$-dimensional faces are
    centrally-symmetric.
\end{thm}

We give a local characterization of inscribed belt polytopes and zonotopes
akin to the above in Theorem~\ref{thm:bolker_inscr}.

An arrangement $\Arr$ is \Def{inscribable} if $\InCone(\Arr) \defeq
\InCone(\Fan(\Arr)) \neq \emptyset$, that is, if there is an inscribed belt
polytope $P$ with $\Fan(P) = \Arr$. If there is an inscribed zonotope $Z$
with $\Fan(Z) = \Arr$, then $\Arr$ is \Def{strongly} inscribable.  Since
Minkowski sums of zonotopes are zonotopes, this prompts the definition of
the \Def{strongly
inscribed cone} of $\Arr$ as the subcone of the inscribed cone
\[
    \ZInCone(\Arr) \defeq \{ Z \in \InCone(\Arr) : \text{$Z$ zonotope}\} \
    \subseteq \ \InCone(\Arr) \,.
\]

A polytopal fan $\Fan = \Fan(P)$ is \Def{even} if every $2$-face of $P$ has an
even number of vertices or, equivalently, if the link of every codimension-$2$
cone of $\Fan$ is an even cycle.  Theorem~\ref{thm:charac_zono} yields that
every arrangement is even and together with Theorem~4.13 in~\cite{InFan1} we
get:

\begin{cor}\label{cor:belt_full}
    Let $\Arr$ be an arrangement of hyperplanes in
    $\R^d$. If $\Arr$ is inscribable, then $\dim \InCone(\Arr) = d$.
\end{cor}

\subsection{Reflection arrangements}

An important class of inscribable arrangements comes from reflection
arrangements. For $\alpha \in \R^d \setminus \{\0\}$, let $s_\alpha : \R^d
\to \R^d$ be the reflection in $\alpha^\perp$, that is, the orthogonal
transformation that fixes $\alpha^\perp$ pointwise and that satisfies
$s_\alpha(\alpha) = -\alpha$. A \Def{finite reflection group} $\ReflG$ is a
finite subgroup of $O(\R^d)$ that is generated by reflections. The
associated \Def{reflection arrangement} is 
\[
    \Arr_\ReflG \ = \  \{ \alpha^\perp : s_\alpha \in W \} \, .
\]
For $q \in \R^d$, the \Def{$\ReflG$-permutahedron} is the polytope
$P_\ReflG(q) \defeq \conv(\ReflG\cdot q)$. It follows for example from
Theorem~1.12 in~\cite{Humphreys} that whenever $q$ is not contained in a
reflection hyperplane, then $\Fan(P_\ReflG(q)) = \Arr_\ReflG$. Since $W$ acts
by orthogonal transformations, $P_\ReflG(q)$ is inscribed which proves the following:
\begin{prop}\label{prop:refl_insc}
    Every reflection arrangement $\Arr_W$ is inscribed and all
    $P \in \InCone(\Arr_W)$ are $W$-permutahedra.
\end{prop}

In particular, by Corollary~\ref{cor:belt_full}, $\InCone(\Arr_W)$ is
full-dimensional. The fans induced by reflection arrangements play a
distinguished role. A fan $\Fan$ is called \Def{full} if the linear map
$v_{R_0} : \InCone(\Fan) \to \interior R_0$ given in
Theorem~\ref{thm:inspc_is_cone} is an isomorphism for all $R_0$.

\begin{thm}\label{thm:full_fans}
    Let $\Fan$ be a fan. Then $\Fan$ is full if and only if $\Fan =
    \Fan(\Arr_W)$ for some reflection arrangement $\Arr_W$.
\end{thm}

\begin{proof}
    Let $\Fan$ be a full and inscribable fan. By Theorem~4.13
    in~\cite{InFan1} $\Fan$ is even. Let
    $C \in \Fan$ be a cone of codimension-$2$. It follows from
    Corollary~4.16 of~\cite{InFan1} that $\Fan_C$ is also full. In
    fact, $v_{R_0} : \InCone(\Fan_C) \to \interior R_0$ is an
    isomorphism for every cone $R_0 \in \Fan_C$. But this means that
    $\Fan_C$ is a $2$-dimensional complete fan with all regions
    isometric to each other. Since the number of regions is even, this
    is means that $\Fan_C = \Fan(\Arr')$ for some line arrangement
    $\Arr'$.  Via Theorem~\ref{thm:charac_zono} this implies that
    $\Fan = \Fan(\Arr)$ for some hyperplane arrangement. Lemma~2.5
    in~\cite{EhrenborgKlivansReading} states that if the restriction
    of $\Arr$ to any flat of codimension $2$ is a reflection
    arrangement, then $\Arr$ is a reflection arrangement. This
    finishes the proof.
\end{proof}

To see that reflection arrangements are \emph{strongly} inscribed, recall
that a \Def{root system} is a non-empty, finite collection $\Phi \subset \R^d
\setminus \{\0\}$ such that for all $\alpha \in \Phi$
\begin{enumerate}[\rm (R1)]
    \item $\Phi \cap \R\alpha = \{-\alpha, \alpha\}$ and
    \item $s_\alpha(\Phi) = \Phi$.
\end{enumerate}
The group generated by the reflections in the hyperplanes $\{ \alpha^\perp :
\alpha \in \Phi \}$ is a finite reflection group and, conversely, every
reflection group has a root system. The \Def{Coxeter zonotope} associated to
$\Phi$ 
\[
    Z_\Phi \ \defeq \ \sum_{\alpha \in \Phi} [-\alpha, \alpha]
\]
has normal fan $\Arr_W$ and $w \cdot Z_\Phi  =  Z_\Phi$
every $w \in \ReflG$. In fact $\ReflG$ acts transitively on the regions of
$\Arr_W$
\cite[Sect.~1.15]{Humphreys} and thus $\ReflG$ acts transitively on the
vertices of $Z_\Phi$. This implies that all vertices lie on a common sphere
and proves 
\begin{prop}
    $\Arr_W$ is strongly inscribed for all reflection groups $W$.
\end{prop}

The converse is also true:
\begin{prop}
    If $Z$ is an inscribed zonotope with $\Fan(Z) = \Arr_\ReflG$, then $Z$
    is a translate of $Z_\Phi$ for some root system of $\ReflG$.  In
    particular $\dim \ZInCone(\Arr_W)$ is equal to the number of orbits of
    $\Phi$ under $W$.
\end{prop}
\begin{proof}
    We may assume that $Z = \sum_{i = 1}^n [-z_i,z_i]$. If we require the
    vectors $z_i$ to be pairwise linearly independent, then the Minkowski
    sum decomposition is unique up to relabelling. Since $Z \in
    \InCone(\Arr)$, $Z$ is an $W$-permutahedra by
    Proposition~\ref{prop:refl_insc} and hence $wZ = \sum_i [-wz_i,wz_i] =
    Z$ for all $w \in \ReflG$. This shows that $\Phi = \{ \pm z_1,\dots, \pm
    z_n \}$ is a root system for $\ReflG$. For the second claim, we note
    that root systems of $\ReflG$ are unique up to scaling, where (R2)
    requires that roots in the same orbit have the same length.
\end{proof}

Finite reflection groups have been classified \cite[Section 2.7]{Humphreys}:
there are four infinite families of irreducible reflection groups $A_n$,
$B_n$, $D_n, I_2(k)$, where the subscript denotes rank. In addition there
are six sporadic examples $F_4$, $E_6$, $E_7$, $E_8$, $H_3$ and $H_4$. The
following table gives the dimensions of the strong inscribed cones:

\begin{center}
    \begin{tabular}{c|c|c|c|c|c|c|c|c|c}
        $A_n$&$B_n$&$D_n$&$F_4$&$E_6$&$E_7$&$E_8$&$H_3$&$H_4$&$I_2(k)$\\
      \hline
      $1$  &$2$  &$1$  &$2$ &$1$  &$1$  &$1$  &$1$  &$1$ &$k \mod 2$
    \end{tabular}
\end{center}

We come back to these examples in Sections~\ref{sec:sporadic}
to~\ref{sec:A_n}.

\subsection{Inscribed virtual Zonotopes}
In~\cite{InFan1}, we observed that questions of inscribability
naturally extend beyond polytopes to virtual polytopes. In this paper
we will explore inscribed virtual zonotopes and belt polytopes, so we
give an executive summary of the theory tailored to our needs.

\newcommand\hTypeCone{\widehat{\TypeSpc}_+}%
The collection $\hTypeCone$ of polytopes $P$ with fixed normal fan $\Fan$
forms a convex cone with respect to Minkowski sums. The support function
of $P$ is the function $h_P : \R^d \to \R$ with
\[
    h_P(c) \ = \ \max \{ \inner{c,v} : v \in V(P) \} \, .
\]
This is a strictly-convex piecewise-linear function supported on $\Fan$.
Since $h_P$ uniquely determines $P$ and $h_{P+Q} = h_P + h_Q$, we get a
faithful representation of $\hTypeCone$ as the (open) polyhedral cone of
strictly-convex piecewise-linear functions supported on $\Fan$. The formal
Minkowski difference of $P$ and $Q$ is the PL function $h_P - h_Q$.
Conversely, every PL function is the difference of two support functions and
hence $\PL(\Fan) = \hTypeCone + (-\hTypeCone)$. Convex functions in
$\PL(\Fan)$ correspond to weak Minkowski summands of $P$. Non-convex functions
$\ell \in \PL(\Fan)$ are called \Def{virtual polytopes} and denoted by $P-Q$,
whenever $\ell = h_P - h_Q$. The lineality space of $\hTypeCone$ is given by
translations. The \Def{type cone} of $\Fan$ is $\TypeCone(\Fan) =
\hTypeCone(\Fan) / T_d$, where $T_d$ is the group of translations.
Consequently, the \Def{type space} $\TypeSpc(\Fan) = \TypeCone(\Fan) +
(-\TypeCone(\Fan))$ is isomorphic to $\PL(\Fan) / (\R^d)^*$.
If $P - Q \in \TypeSpc(\Fan)$ and $c \in \R^d \setminus \{0\}$, then we say
that $P^c - Q^c$ is a \Defn{face} of $P - Q$.

If $\Arr = \{ H_i = z_i^\perp : i=1,\dots,n\}$ is an arrangement of
hyperplanes, then up to translation every zonotope of $\Arr$ is of the form
$Z_\lambda = \lambda_1 [-z_1,z_1]   + \cdots + \lambda_n [-z_n,z_n]$ for a
unique $\lambda \in \Rpp^n$. Thus, the subcone of $\TypeCone(\Arr)$ spanned by
zonotopes is isomorphic to $\Rpp^n$. If $Z_{\lambda^+}, Z_{\lambda^-}$ are two
zonotopes with $\lambda^+,\lambda^- \in \Rpp^n$, then 
\[
    h_{Z_{\lambda^+}}(c) - h_{ Z_{\lambda^-}}(c) \ = \ \sum_{i=1}^n \lambda_i
    |\inner{z_i,c}| \, .
\]
where $\lambda = \lambda^+ - \lambda^- \in \R^n$. Every \Def{virtual zonotope}
$Z_\lambda = Z_{\lambda^+} - Z_{\lambda^-}$ is, up to translation, uniquely
represented by $\lambda \in \R^n$.

For a PL function $\ell \in \PL(\Fan)$ and a region $R \in \Fan$,
there is $v_R \in \R^d$ such that
$\ell(x) = \inner{v_R,x}$ for all $x \in R$. The collection
$V(\ell) \defeq \{ v_R : R \in \Fan \text{ region}\}$ is the
\Defn{vertex set} of $\ell$. We call $\ell$ \Def{inscribed} if
$V(\ell)$ is contained in a sphere and we define $c(\ell)$ to be the
\Defn{center} of this sphere (if $V(\ell)$ is contained in an affine
subspace of $\R^d$ the center is contained in this subspace). If
$\ell = h_P$, then $V(\ell) = V(P)$ and $P$ is inscribed precisely
when $h_P$ is. The inscribed cone $\InCone(\Fan)$ is naturally a
relatively open subcone of $\TypeCone(\Fan)$.

If a polytopal fan $\Fan$ is inscribable, we call
$\InSpc(\Fan) \defeq \InCone(\Fan) + (-\InCone(\Fan)) \subseteq
\TypeSpc(\Fan)$ the \Def{inscribed space} of $\Fan$. The precise
definition of $\InSpc(\Fan)$ for $\Fan$ not inscribable is a bit
subtle, we refer to~\cite[Section 5]{InFan1} for details. Indeed, not
every inscribed PL function $\ell \in \TypeSpc(\Fan)$ is contained in
$\InSpc(\Fan)$. For example, if $\Fan = \Fan(\Arr)$ with
$z_1^\perp, z_2^\perp \in \Arr$, then
$S_i = [-z_i,z_i] \in \TypeSpc(\Fan)$ is inscribed for $i=1,2$ but
$S_i \in \InSpc(\Fan)$ if and only if $z_1 \perp z_2$.  The
following Lemma follows easily from the definition of the inscribed
space in~\cite[Section 5]{InFan1}, but we might also take it as a
definition for our purposes here. For any two regions $R,S \in \Fan$
such that $R \cap S$ is of codimension $1$, let $\alpha_{RS}$ be a
normal vector to the hyperplane $\lin(R \cap S)$.  We have:
\begin{lem}\label{lem:edge_midpoint_orthogonal}
    Let $P-Q \in \TypeSpc(\Fan)$ and let
    $V(P-Q) = \{ v_R : R \in \Fan \text{ region}\}$. Then
    $P-Q \in \InSpc(\Fan)$ if and only if $P-Q$ is inscribed and
    \[
        \inner{\alpha_{RS}, \,c([v_R, v_S]) - c(P-Q)} \ = \ 0
    \]
    for all adjacent regions $R,S \in \Fan$.
\end{lem}

If $\Fan = \Fan(\Arr)$, then $\alpha_{RS} = z_i$ for some $i$ for any two
adjacent regions $R,S$. We will exploit
Lemma~\ref{lem:edge_midpoint_orthogonal} in Section~\ref{sec:repr} to give a
simple representation of $\InCone(\Arr)$ and $\InSpc(\Arr)$.

\begin{example}
    \begin{figure}[ht]
        \centering
        \includegraphics[width=\textwidth]{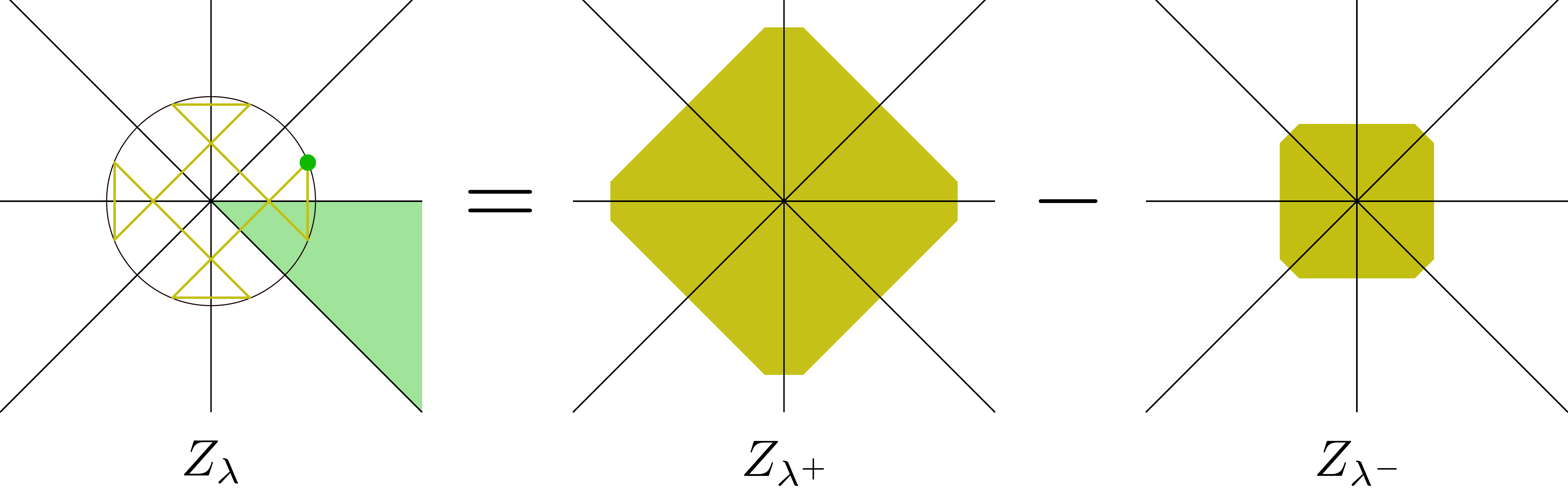}
        \caption{  \label{fig:b2_zonotope} }
    \end{figure}

    Figure~\ref{fig:b2_zonotope} shows an inscribed virtual zonotope
    $Z_\lambda$ presented as the Minkowski difference of two (inscribed)
    non-virtual zonotopes. The vertex of $Z_\lambda$ corresponding to the
    green region $R$ is $v_R = \begin{psmallmatrix}5 \\ 2\end{psmallmatrix}$.
        Thus, the corresponding PL function $\ell(c) \defeq
        h_{Z_{\lambda^+}}(c) - h_{Z_{\lambda^-}}(c)$ satisfies
    $\ell_R(c) = \inner{v_R,c}$ for all $c \in R$.
\end{example}

\section{Strongly inscribable fans in dimension $2$}\label{sec:dim2}
\newcommand{\tbeta}{\tilde\beta}%
Throughout the section, let $\Arr$ be an arrangement of $n \geq 2$ lines in
$\R^2$. Belt polytopes and zonotopes are easily described in the
$2$-dimensional setting: a \Defn{belt polygon} is a polygon with an even
number of edges with opposite edges parallel and a \Defn{zonogon} is a
centrally symmetric polygon, i.e.\ a belt polygon, whose opposites edges are
not only parallel but equal in length.  In the following two subsections, we
give geometric and algebraic characterizations of (strongly) inscribed line
arrangements. These characterizations are indispensable for subsequent
results in higher dimensions. In particular, we show that the notions of
\emph{inscribed} and \emph{strongly inscribed} coincide in the plane.

\begin{thm}\label{thm:2d_strongly_insc}
    A rank-$2$ arrangement $\Arr$ is (virtually) strongly inscribed if and
    only it is (virtually) inscribed.
\end{thm}

If $\Arr$ has an inscribed belt polytope $P$, then $-P$ is also inscribed and
$\Fan(-P) = \Fan(P) = \Arr$. We simply observe that $P + (-P)$ is a
centrally-symmetric polygon and hence inscribed by Theorem~\ref{thm:inspc}.
This shows that inscribed arrangements are also strongly inscribed. The same
argument does not work for inscribed virtual polytopes: For example, if $\Arr$
is the reflection arrangement of type $A_2$, then $P = P_W(1,2,4)$ is an
inscribed hexagon, which is not centrally symmetric. The virtual polytope $P -
(-P)$ given by $\ell = h_P - h_{-P}$ is inscribed by
Lemma~\ref{lem:edge_midpoint_orthogonal} but its symmetrization satisfies
\[
    \ell(c) + \ell(-c) \ = \ h_{P}(c) - h_{-P}(c) + h_{P}(-c) - h_{-P}(-c) \ = \ 0
\]
and thus is the zero function. We proof
Theorem~\ref{thm:2d_strongly_insc} in full generality in
Section~\ref{sec:2d_geometric}.

\subsection{A geometric perspective}\label{sec:2d_geometric}
We order the $2n$ regions of $\Arr$ counterclockwise and denote them by
$R_0,\dots,R_{2n-1}$ as in Figure~\ref{fig:2arr}. For $i = 0,\dots,2n-1$, let
$\beta_i \in (0,\pi)$ be the angle of $R_i$. Note that by central symmetry,
$\beta_i = \beta_{n+i}$ for $i = 0, \dots, n-1$ and thus $\beta_0 + \beta_1 +
\dots + \beta_{n-1} = \pi$. We call $\beta(\Arr) = (\beta_0,\dots,\beta_{n-1},
\beta_0, \dots, \beta_{n-1})$ the \Def{profile} of $\Arr$ and $\tbeta(\Arr) =
(\beta_0, \dots, \beta_{n-1})$ the \Def{reduced profile}. The reduced profile
$\tbeta(\Arr)$ determines $\Arr$ up to rotation. More generally, the profile
$\beta(\Fan)$ of a $2$-dimensional fan $\Fan$ is given by the angles of its
regions ordered counterclockwise~\cite[Section~3]{InFan1}.
Clearly $\beta(\Arr) = \beta(\Fan(\Arr))$.

\begin{figure}[ht]
    \centering
    \includegraphics[width=0.35\textwidth]{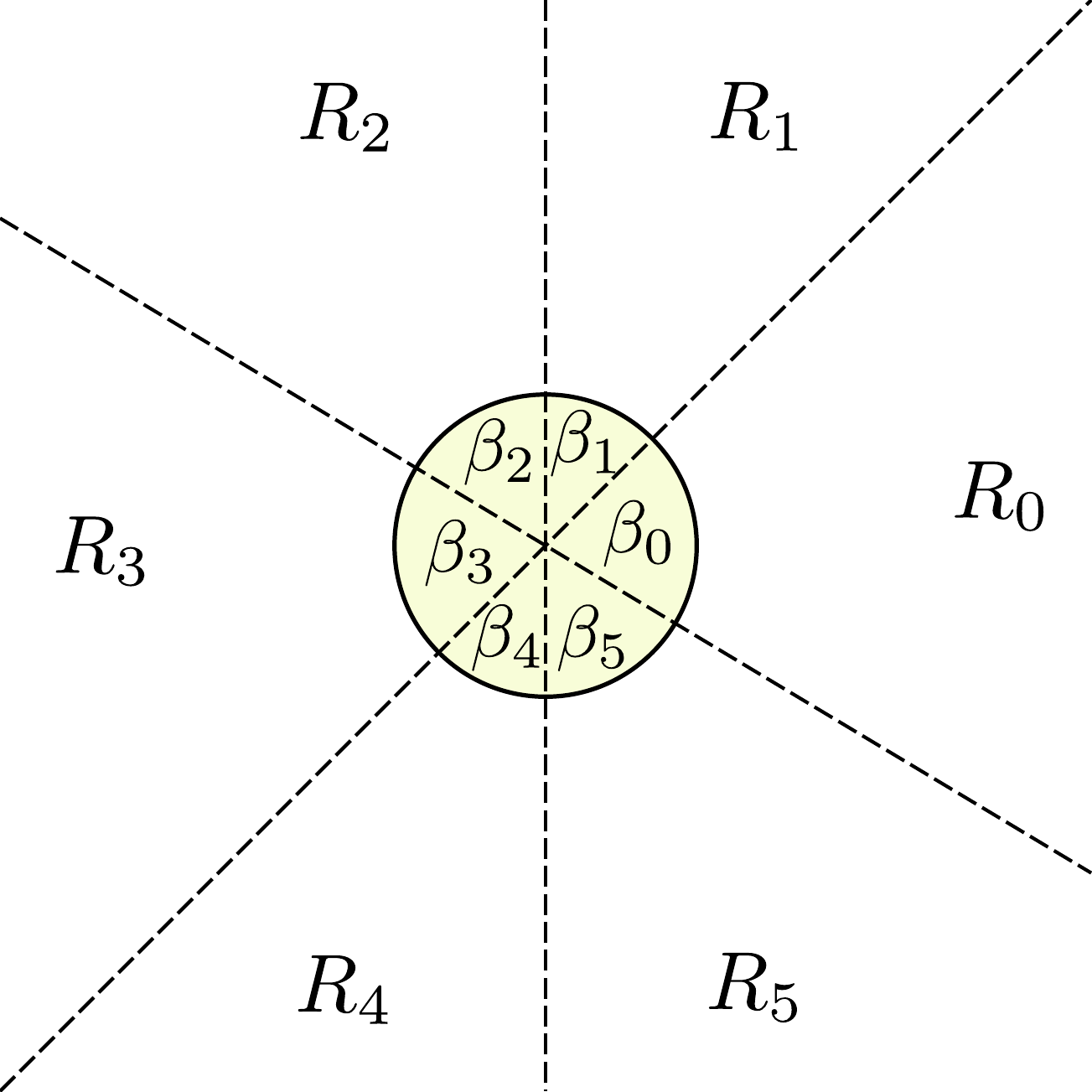}\qquad\qquad
    \includegraphics[width=0.35\textwidth]{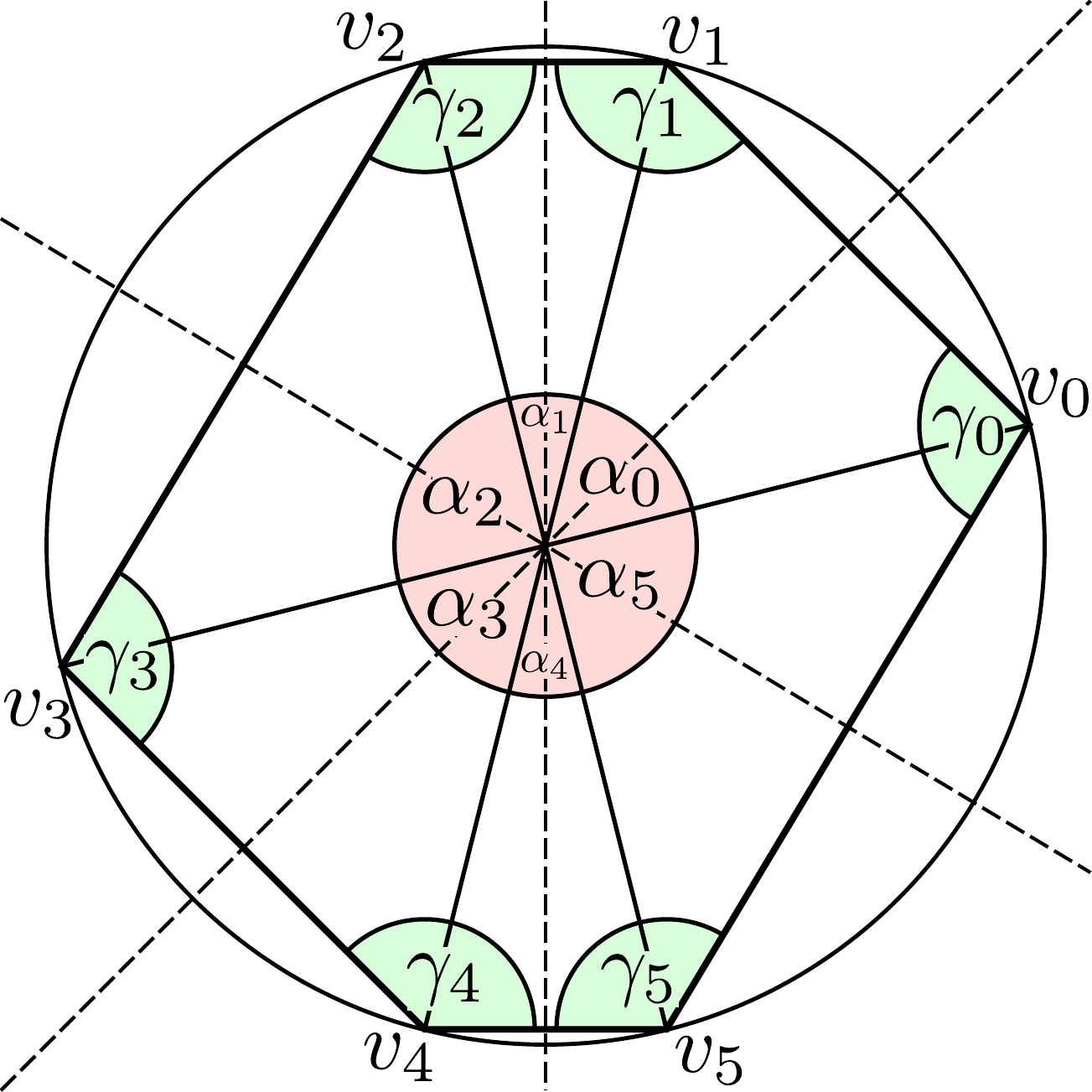}
    \caption{ Left: A $2$-dimensional arrangement $\Arr$ with regions and
      its reduced profile $\tbeta = (\beta_0, \beta_1, \beta_2)$.
      Right: An inscribed zonogon with normal fan $\Arr$, face angles
      $\alpha_0, \alpha_1, \dots, \alpha_5$, and interior angles
      $\gamma_0, \gamma_1, \dots, \gamma_5$.
      \label{fig:2arr} }
\end{figure}

To give a complete geometric picture, let us introduce two additional sets of
angles. Let $P$ be a virtual belt polygon with normal fan $\Arr$ and assume
that $P$ is
inscribed into a circle with center at the origin $\0$. We label the vertices
$v_0$, $v_1$, \dots, $v_{2n-1}$ such that $v_i$ has normal cone $R_i$ for
$i=0,\dots,2n-1$. For convenience, set $v_{-1} \defeq v_{2n-1}$ and $v_{2n}
\defeq v_0$. The \Defn{face angle} $\alpha_i$, $i = 0, \dots, 2n-1$, is the
oriented angle $\measuredangle(v_{i-1}, \0, v_i) \in (-\pi,\pi]$. Furthermore,
we define the \Defn{interior angle} $\gamma_i$ , $i = 0, \dots, 2n-1$, to be
the angle $\measuredangle(v_{i+1}, v_{i}, v_{i-1})$
(cf.~Figure~\ref{fig:2arr}). 
These angles are related to $\beta(\Arr)$ via the
following equations for all $i = 0, \dots, n-1$:
\[
    \tfrac{1}{2}(\alpha_{i} + \alpha_{i+1}) \ \equiv \ \beta_{i} \ \equiv \
    \pi - \gamma_i\,.
\]
Moreover, we have $\gamma_i = \gamma_{i+n}$. If $P$ is a zonotope, then
$\alpha_i = \alpha_{i+n}$ by central symmetry.

\begin{proof}[Proof of Theorem~\ref{thm:2d_strongly_insc}]
    We only need to show that if $\Arr$ has an inscribed virtual belt
    polygon $P$, then is also has an inscribed virtual zonogon.

    Let $\alpha_1, \dots, \alpha_{2n}$ be the face angles of $P$. Let $Z$ be
    the inscribed virtual zonogon with face angles $\alpha'_i = \alpha'_{n+i}
    = \tfrac{1}{2}(\alpha_i + \alpha_{n+i})$, for $i = 1, \dots, n$. 
    We observe that 
    \[
        \tfrac{1}{2}(\alpha'_i + \alpha'_{i+1}) \ = \
        \tfrac{1}{4}(\alpha_i + \alpha_{n+i} + \alpha_{i+1} +
        \alpha_{n+i+1}) \ = \ \tfrac{1}{2}(\beta_i + \beta_{n+i}) \ = \ \beta_i\,,
    \]
    and hence $\tbeta(\Fan(Z)) = \tbeta(\Arr)$. Thus, up to rotation, $Z$ is
    an inscribed zonogon with normal fan $\Arr$.
\end{proof}

\begin{thm}
    $\beta$ is the reduced profile of an inscribed arrangement if and
    only if $2\beta$ is the profile of an inscribed fan.
\end{thm}
\begin{proof}
    Let $\beta = \tbeta(\Arr)$ be the reduced profile of an inscribed
    arrangement $\Arr$. By Theorem~\ref{thm:2d_strongly_insc}, there
    exist an inscribed zonotope $Z \in \ZInCone(\Arr)$. Let
    $\alpha_1, \dots, \alpha_{2n}$ be the face angles of $Z$. These
    are positive real numbers with $\alpha_i = \alpha_{n+i}$, and
    therefore $\alpha_1 + \dots + \alpha_n = \pi$.

    Let $w_1, \dots, w_n$ be $n$ points on the unit sphere such that
    $\measuredangle(w_{i-1}, \0, w_i) = 2\alpha_i$ for
    $i=2, \dots, n$. Then $Q \defeq \conv(w_1, \dots, w_n)$ is an
    inscribed polygon with face angles $\alpha_i' \defeq 2\alpha_i$,
    $1 \leq i \leq n$. It is not hard to see that for
    $\beta' \defeq \beta(\Fan(Q))$:
    \[
        \beta'_i \ = \ \tfrac{1}{2}(\alpha'_i + \alpha'_{i+1}) \ = \
        \alpha_i + \alpha_{i+1} \ = \ 2\beta_i\,.
    \]
    Starting with $Q$ and reversing the process completes the proof.
\end{proof}

We conclude the following from~\cite[Theorem~3.5]{InFan1}:
\begin{prop}\label{prop:2d_insc}
    Let $\Arr$ be a $2$-dimensional line arrangement with $n \geq 2$
    lines, and reduced profile
    $(\beta_0, \beta_1, \dots,\beta_{n-1})$. If $n$ is odd, then
    $\Arr$ is (strongly) inscribable, if and only if for all
    $0 \leq j < n$:
    \[
        \beta_{j+0} - \beta_{j+1} + \cdots - \beta_{j+n-2} +
        \beta_{j+n-1} \ > \ 0\,.
    \]        
    If $n = 2m$ is even, then $\Arr$ is (strongly) inscribable, if and
    only if for all $0 \le h < m$ and $0 \le j < n$
    \begin{align*}
        \beta_0 + \beta_2 + \cdots + \beta_{n-2} \ &= \ \frac{\pi}{2} \,,&
      \sum_{i = 1}^h \beta_{2i+j} + \sum_{i = h+1}^{m-1} \beta_{2i+1+j}
      \ &< \ \frac{\pi}{2}\,.
    \end{align*}
\end{prop}

\begin{example}\label{ex:small_ieqs}
    Let us explicitly state these equations and inequalities for
    $n = 2, 3$ and $4$.

    For $n = 2$, we have $\beta_1 = \beta_2$ and
    $\beta_1 + \beta_2 = \pi$, so
    \begin{align*}
      \beta_1 &= \beta_2 = \tfrac{\pi}{2}\,.
    \end{align*}
    Thus, an inscribed parallelogram is a rectangle.
    
    For $n = 3$, we get $\beta_1 + \beta_2 + \beta_3 = \pi$,
    so $\pi - \beta_3 = \beta_1 + \beta_2 > \beta_3 > 0$, so by symmetry:
    \begin{align*}
      \beta_1 + \beta_2 + \beta_3 &= \pi\,, & 0 &< \beta_1, \beta_2, \beta_3 < \tfrac{\pi}{2}\,.
    \intertext{Finally, for $n = 4$, the inequalities can be reduced to}
      \beta_1 + \beta_3 = \beta_2 + \beta_4 &= \tfrac{\pi}{2}\,, & 0 &< \beta_1, \beta_2, \beta_3, \beta_4\,.
    \end{align*}
    For $n \geq 5$, the inequalities are more involved.
\end{example}

We close this section with a important corollary, which will help us
to show in Theorem~\ref{thm:insc_arr_are_simplicial} that inscribed belt polytopes are
simple:
\begin{cor}\label{cor:belt_polygons_obtuse}
    Let $\Arr$ be a $2$-dimensional inscribed line arrangement with
    $n \geq 2$ lines and reduced profile
    $(\beta_0,\dots,\beta_{n-1})$.  Then
    $\beta_i \leq \frac{\pi}{2} \leq \gamma_i$ for
    $i = 0, \dots, n-1$.
\end{cor}
\begin{proof}
    If $n = 2$, then by Example~\ref{ex:small_ieqs}
    $\beta_0 = \beta_1 = \frac{\pi}{2} = \gamma_0 = \gamma_1$, so any
    inscribed realization is a rectangle. Otherwise:
    \begin{align*}
      \pi \ & = \ \beta_0 + \beta_1 + \dots + \beta_n \ > \ \beta_{i-1} + \beta_i + \beta_{i+1} \\
            & > \ \tfrac{1}{2}\alpha_{i-1} + \alpha_{i} + \alpha_{i+1} + \tfrac{1}{2}\alpha_{i+2} \ > \ \alpha_i + \alpha_{i+1} = 2\beta_i
              \,.\qedhere
    \end{align*}
\end{proof}

\subsection{An algebraic perspective}\label{sec:2d_algebraic}
\newcommand{\pfaff}{\operatorname{pf}}%
\newcommand{\sR}{R}%
While the description of inscribable rank-$2$ arrangements in the previous
subsection in terms of the reduced profile is very simple, it does not
give a description of the strongly inscribed cone.

Let $\Arr$ be an arrangement of $n \geq 2$ lines in $\R^2$. The lines are
given by $z_1^\perp, \dots, z_n^\perp$ for some non-zero vectors $z_1, \dots,
z_n$. We assume $z_1, \dots, z_n, -z_1, \dots, -z_n$ to be cyclically ordered.
The \Defn{skew-Gram matrix} of an ordered collection of vectors $z_1, \dots,
z_n$ is the skew-symmetric matrix $\sR = \sR(z_1, \dots, z_n) \in \R^{n \times
n}$ with $\sR_{ij} = -\sR_{ji} = \inner{z_i, z_j}$ for $i < j$ and $\sR_{ii} =
0$.

\begin{thm}\label{thm:inspc_eq_ker_R}
    Let $\Arr = \{z_1^\perp, \dots, z_n^\perp\}$ as above and let $\sR =
    \sR(z_1, \dots, z_n)$. As vector spaces, $\ZInSpc(\Arr) \cong \ker \sR$.
    As cones, $\ZInCone(\Arr) \cong \ker \sR \cap \Rpp^n$.
\end{thm}

These isomorphisms can be described as follows: For $\lambda \in \ker \sR \cap
\Rpp^n$ the corresponding zonogon is
\[
    Z_{\lambda} \ \defeq \ \sum_{i = 1}^n  \lambda_i [-z_i, z_i]  \in
    \ZInCone(\Arr)\,.
\]
For $\lambda \in \ker \sR$, let $\lambda^+, \lambda^- \in \Rpp^n$ such that
$\lambda = \lambda^+ - \lambda^-$. The virtual zonogon associated to $\lambda$
is then $Z_\lambda \defeq Z_{\lambda^+} - Z_{\lambda^-} \in \ZInSpc(\Arr)$.
\begin{proof}[Proof of Theorem~\ref{thm:inspc_eq_ker_R}]
    Let $Z \in \ZInSpc(\Arr)$ with $c(Z) = 0$, i.e., that
    $Z = Z_\lambda$ for some $\lambda \in \R^n$. The
    vertices of the (virtual) zonogon $Z_\lambda$ are
    \[
        v_i \ = \  \lambda_1 z_1 + \cdots + \lambda_{i-1} z_{i-1} - 
        \lambda_{i} z_{i} - \cdots - \lambda_{n} z_{n}  
    \]
    for $i=1,\dots,n$ together with their reflections
    $-v_1,\dots,-v_n$. Thus, the midpoint of the edge
    $e_i \defeq [v_i,v_{i+1}] = v_i + [0,2\lambda_iz_i]$ is
    \[
        c(e_i) \ = \ v_i + \lambda_iz_i \ = \ \lambda_1 z_1 + \cdots + \lambda_{i-1} z_{i-1} - \lambda_{i+1} z_{i+1} - \cdots - \lambda_{n} z_{n}\,.
    \]
    Lemma~\ref{lem:edge_midpoint_orthogonal} gives $\inner{z_i, c(e_i)} = 0$
    for all $i = 1, \dots, n$ which is equivalent to
    $\sR \lambda = 0$.
\end{proof}
  
Recall that a skew-symmetric matrix of odd order is singular. If $\sR$ is a
skew-symmetric matrix of order $n = 2m$, then $\det \sR = (\pfaff \sR)^2$, where
$\pfaff \sR$ denotes the \Defn{Pfaffian} of $\sR$.

\begin{cor}\label{cor:pfaff}
    Let $\Arr$ be an arrangement of $n$ lines in $\R^2$. If $n$ is
    odd, then $\Arr$ is virtually inscribable. Otherwise, $\Arr$ is
    virtually inscribable if and only if $\pfaff \sR(\Arr) = 0$.
\end{cor}

As a refinement of Corollary~\ref{cor:pfaff}, we have the following:
\begin{cor}
    Let $\Arr$ be an arrangement of $n$ lines in $\R^2$.
    If $n$ is even, $\InSpc(\Arr) = \ZInSpc(\Arr)$. If $n$ is odd,
    then $\dim \ZInSpc(\Arr) = 1$.
\end{cor}
\begin{proof}
    Let $n$ be even. Since $\ZInSpc(\Arr) \subseteq \InSpc(\Arr)$, we
    are left to show that $\dim \ZInSpc(\Arr) = \dim \InSpc(\Arr)$.
    The statement is trivial if $\dim \InSpc(\Arr) = 0$, so assume
    $\dim \InSpc(\Arr) > 0$. The rank of every
    skew-symmetric matrix is even and since
    $\rk\sR(\Arr) \neq 0$, we get
    $\dim \ker \sR = n - \rk(\sR) \geq
    2$. By~\cite[Proposition~3.1]{InFan1}, we see that
    $2 = \dim \InSpc(\Arr) \geq \dim \ZInSpc(\Arr) = \dim \ker \sR$
    and therefore $\dim \ZInSpc(\Arr) = \dim \InSpc(\Arr)$.

    If $n$ is odd, then $\dim \ker \sR = n - \rk(\sR) \geq 1$ is
    odd, but since $\dim \ZInSpc(\Arr) \leq \dim \InSpc(\Arr) = 2$
    and $\ker \sR \cong \ZInSpc(\Arr)$, we have in fact
    $\dim \ZInSpc(\Arr) = 1$.
\end{proof}

\begin{example}\label{ex:small_tRs}
    We will list the immediate consequences of the last results for $n = 2$,
    $3$ and $4$ to compare with Example~\ref{ex:small_ieqs}. For $n=2,3,4$,
    let $\Arr_n = \{z_1^\perp, \dots, z_n^\perp\}$ be an arrangement of $n$
    lines $z_1^\perp,\dots,z_n^\perp$ in $\R^2$. Assume that 
$z_1, \dots, z_n, -z_1, \dots, -z_n$ is cyclically ordered and write 
    $r_{ij} = \inner{z_i, z_j}$ for $i,j = 1,\dots,n$. The skew-Gram matrices
    $\sR_n = \sR(\Arr_n)$ are then:
    \begin{align*}
      \sR_2 &= \begin{pmatrix}0&r_{12}\\-r_{12}&0\end{pmatrix}&
      \sR_3 &= \begin{pmatrix}0&r_{12}&r_{13}\\-r_{12}&0&r_{23}\\-r_{13}&-r_{23}&0\end{pmatrix}&
      \sR_4 &=  \begin{pmatrix}0&r_{12}&r_{13}&r_{14}\\-r_{12}&0&r_{23}&r_{24}\\-r_{13}&-r_{23}&0&r_{34}\\-r_{14}&-r_{24}&-r_{34}&0\end{pmatrix}.
    \end{align*}

    For $n = 2$, we have $\pfaff \sR_2 = r_{12}$, so $\Arr_2$ is
    inscribed if and only if $r_{12} = 0$. The kernel of $\sR_2$ is
    simply $\R^2$. Geometrically, this is clear: If $r_{12} = 0$,
    then $\Arr_2$ consists of two orthogonal lines. All (virtual) polygons
    with this normal fan are rectangles and therefore inscribed. The space of
    such rectangles is parameterized by the two side lengths.

    For $n = 3$, $\pfaff \sR_3 = 0$, since $n$ is odd, so $\Arr_3$ is always
    virtually inscribed. The kernel of $\sR_3$ is $(r_{23}, -r_{13}, r_{12})$
    and therefore $1$-dimensional. Finally, $\Arr_3$ is strongly inscribed, if
    and only if $r_{12}, -r_{13}, r_{23} > 0$. In relation to reduced
    profiles, $r_{13} < 0$ is equivalent to $\beta_3 < \tfrac{\pi}{2}$, while
    $r_{12} > 0$ if and only if $\beta_1 < \tfrac{\pi}{2}$ and $r_{23} > 0$ if
    and only if $\beta_2 < \tfrac{\pi}{2}$. These are precisely the
    inequalities in Example~\ref{ex:small_ieqs} (that $\beta_i > 0$ follows
    from the cyclic ordering).

    For $n = 4$, $\pfaff \sR_4 = r_{12} r_{34} - r_{13} r_{24} + r_{14}
    r_{23}$. Plugging
    \begin{align*}
        \cos \beta_1 &= r_{12},& \cos \beta_2 &= r_{23},& \cos \beta_3 &= r_{34},\\
        \cos(\beta_1 + \beta_3) &= r_{13},& \cos(\beta_2 + \beta_3) &= r_{24},&
        \cos(\beta_1 + \beta_2 + \beta_3) &= r_{14}
    \end{align*}
    into $\pfaff \sR_4$, a quick struggle with trigonometric identities
    reveals $\pfaff \sR_4 = \cos(\beta_1 + \beta_3)$. This is in line with the
    angle conditions of Example~\ref{ex:small_ieqs}.
\end{example}

\section{Operations and strongly inscribed cones}\label{sec:rest_simp} 

We show that the class of inscribed arrangements is closed under the basic
operations of taking products and localizations. While this is clear from the
geometric point of view, we prove that \emph{strong} inscribability is retained
under restrictions as well. Said differently, the orthogonal projection of an
inscribed zonotope onto any of its flats is inscribed. Up to an additional
assumption, this characterizes inscribed zonotopes.

In Section~\ref{sec:repr} we derive a simple representation of the strongly
inscribed cone of an arrangement and give a first extension of Bolker's
characterization to inscribed zonotopes.

\subsection{Products and reducible arrangements}
Let $P_1 \subset \R^d$, $P_2 \subset \R^e$ be inscribed
polytopes. Then $P_1 \times P_2 \subset \R^{d + e}$ is also
inscribed. If $P_1$, $P_2$ are belt polytopes for $\Arr_1$, $\Arr_2$,
respectively, then $P_1 \times P_2$ is a belt polytope for the
arrangement
\[
    \Arr_1 \times \Arr_2 \ \defeq \ \{H \times \R^e : H \in \Arr_1
    \} \cup \{\R^d \times H : H \in \Arr_2 \}\,.
\]
This shows:
\begin{prop}
    If $\Arr_1$, $\Arr_2$ are (strongly) inscribable, then so is
    $\Arr_1 \times \Arr_2$.
\end{prop}

An essential hyperplane arrangement $\Arr$ in $\R^d$ is called \Defn{reducible},
if there exists a partition $\Arr = \Arr_1 \uplus \Arr_2$ into
non-empty subarrangements, such that
$\lineal(\Arr_1) \oplus \lineal(\Arr_2) = \R^d$. Equivalently, $\ess \Arr$
is linearly isomorphic to $\ess \Arr_1 \times \ess \Arr_2$. If $\Arr$
is not reducible, it is called \Defn{irreducible}.

Towards a classification of inscribed arrangements, the following proposition yields that we focus on irreducible arrangements:
\begin{prop}\label{prop:reducible}
    Let $\Arr$ be an inscribable arrangement. If $\Arr$ is reducible
    with $\Arr = \Arr_1 \uplus \Arr_2$ as above, then
    $\lineal \Arr_1 \perp \lineal \Arr_2$.
\end{prop}
\begin{proof}
    Choose $H_1 = z_1^\perp \in \Arr_1$ and
    $H_2 = z_2^\perp \in \Arr_2$. Then $L \defeq H_1 \cap H_2$ is a
    flat of codimension $2$ and no other hyperplane in $\Arr$ contains
    $L$. The localization $\Arr_L$ is inscribed, and thus its
    essentialization is a $2$-dimensional inscribed arrangement with
    $2$ lines. By Example~\ref{ex:small_ieqs}, we see that
    $\beta(\ess(\Arr_L)) = (\frac{\pi}{2}, \frac{\pi}{2})$, or in
    other words, $z_1 \perp z_2$. Varying over $H_1$, $H_2$ shows the
    claim.
\end{proof}

\begin{rem}
    The same argument shows that for two polytopes $P, Q \subset \R^d$
    if $P+Q$ is inscribed and $P + Q$ is combinatorially isomorphic to
    $P \times Q$, then $\aff_0(P) \perp \aff_0(Q)$ and $P$ and $Q$ are
    both inscribed. Note that it is not enough to assume that $P + Q$
    is combinatorially isomorphic to an inscribed polytope, as seen
    for example by irregular inscribed quadrilaterals or more
    generally irregular inscribed cubes.
\end{rem}

\subsection{Restrictions and localizations}
Let $P$ be an inscribed polytope. If  $F \subseteq P$ is a face, then $F$ is
clearly inscribed as well. Recall from Section~\ref{sec:insc_cones} that if
$P$ is a belt polytope with arrangement $\Arr$, then $F$ is a belt polytope
with respect to the localization $\Arr_L$, where \mbox{$L =
\aff_0(F)^\perp$}. This shows that (strongly) inscribed arrangements are
closed under localizations:
\begin{lem}
    Let $L$ be a flat of a (strongly) inscribed arrangement
    $\Arr$. Then $\Arr_L$ is (strongly) inscribed.
\end{lem}

For a flat $L$, let $\pi_L : \R^d \to L$ be the orthogonal projection onto
$L$. The image of a belt polytope $P$ with $\Arr(P) = \Arr$ under $\pi_L$ is a
belt polytope with respect to the restriction $\Arr^L$. The vertices of
$\pi_L(P)$ are precisely the images $\pi_L(F)$, where $F$ ranges over the
faces in the belt corresponding to $L^\perp$.

Projections of inscribed polytopes are rarely inscribed. Inscribed
zonotopes however have the fascinating property, that they remain
inscribed under orthogonal projection onto their flats:
\begin{prop}\label{prop:proj_edges}
    Let $Z$ be an inscribed zonotope and $H \in \Arr = \Arr(Z)$. Then
    $\pi_H(Z)$ is inscribed with center $c(Z)$. Equivalently, $\Arr^H$
    is strongly inscribed.
\end{prop}

Theorem~\ref{thm:restr_insc} follows from Proposition~\ref{prop:proj_edges}
by repeated application.

\begin{proof}
    We may assume that $c(Z) = \0$. Let $E$ be the belt of edges orthogonal to
    $H$. If $e = [u,v] \in E$ is an edge, then all points of $H$ have the same
    Euclidean distance to $u$ and to $v$. In particular, $H$ bisects every edge
    in $E$. Since $Z$ is a zonotope, any two edges in $E$ are
    translates of each other and thus have the same length. Therefore,
    all the triangles $T_e \defeq \conv(e \cup \{\0\})$ with $e \in E$
    are congruent. The segments $[\0, c(e)]$ are the altitudes of
    $T_e$ and therefore $c(e) = \pi_H(e)$ have the same distance to
    the origin for all $e \in E$. The claim now follows from
    $\pi_H(Z) = \conv(c(e) : e \in E)$.
\end{proof}

\begin{rem}
    Note that Proposition~\ref{prop:proj_edges} is not true for belt
    polytopes. If $W$ is the reflection group of type $A_3$ acting on
    $\R^4$ by permuting coordinates, one checks that projections of
    the $W$-permutahedron $P_W(1,2,3,5)$ along edges are not
    inscribed.
\end{rem}

In addition to zonotopes, products of polygons and simplices also
have the property that inscribability is preserved under projection
along edges. Thus, we propose the following question:
\begin{quest}\label{quest:insc_under_proj}
    Which inscribed polytopes remain inscribed when projecting along
    any of their edge directions?
\end{quest}

Unlike products of polygons and simplices, inscribed zonotopes have the
property that the center of the inscribing sphere is preserved under
projection. It turns out that this characterizes inscribed zonotopes.

\begin{proof}[Proof of Theorem~\ref{thm:inscribed_zono}]
    We may assume that $c(P) = 0$ and hence $\Centers_e = e^\perp$.

    (i)$\implies$(ii) follows from Proposition~\ref{prop:proj_edges}. 
    Since $P$ is a belt polytope, it follows that the vertices of the
    projection of $P$ along $e$ are precisely the images of edges $e'$
    parallel to $e$. The proof of Proposition~\ref{prop:proj_edges} now shows
    that the projection and intersection coincide and thus yields
    (i)$\implies$(iii).

    (ii)$\implies$(i) and (iii)$\implies$(i): We will show that
    every $2$-face of $P$ is centrally symmetric. The result then follows from
    Theorem~\ref{thm:charac_zono}. If $P$ satisfies (ii) or (iii)
    respectively, then so does every $2$-face of $P$. Therefore, we can assume
    that $P$ is $2$-dimensional and the  result follows from the next two
    Lemmas.
\end{proof}

\begin{lem}
    Let $P \subseteq \R^2$ be an inscribed polygon with $c(P) =
    \0$. If for every edge $e$ the orthogonal projection $\pi_e(P)$ of
    $P$ onto $e^\perp$ satisfies $c(\pi_e(P)) = 0$, then $P$ is a
    zonogon.
\end{lem}
\begin{proof}
    If $P$ is not a zonogon, then there exists a vertex $v$, such that $-v
    \not\in P$. Let $u, u'$ be the neighbors of $-v$ in $\conv(P \cup
    \{-v\})$. Since $P$ is inscribed, the segment $e = uu'$ is an edge of $P$
    and the affine line spanned by $e$ separates $-v$ from $P$.  Let $\eta$ be
    the outer unit normal vector to $\aff(e)$. The image $\pi_e(P)$ is a
    segment $aa'$, with
    \[
        ||a||  \ = \ \max \{\inner{x, \eta} : x \in P\} \ = \ \inner{u, \eta}\,,\qquad
        ||a'|| \ = \ \max \{\inner{x,-\eta} : x \in P\}\,,\qquad
    \]
    and $||a|| = ||a'||$ by assumption. But then
    \[
        ||a'|| \ \geq \ \inner{v, -\eta} \ = \ \inner{-v, \eta} \ > \
        \inner{u, \eta} \ = \ ||a||\,,
    \]
    a contradiction.
\end{proof}

\begin{lem}
    Let $P \subseteq \R^2$ be an inscribed polygon with $c(P) = 0$. If
    $c(P \cap e^\perp) = \0$ for every edge $e$, then $P$ is a
    zonogon.
\end{lem}
\begin{proof}
    Let $e = [x, y]$ be an edge of $P$.  We first establish the
    following claims:
    \begin{enumerate}[(i)]
    \item \emph{No vertex $v $ of $P$ is contained in $e^\perp$.}
        Otherwise, $P \cap e^\perp = [c(e), v]$. Since $P$ is
        inscribed, $||v|| = ||x|| = ||y||$, but as
        $c(P \cap e^\perp) = \0$, we have $||v|| = ||c(e)|| < ||x||$,
        contradiction.
    \item \emph{If $-e$ is an edge of $P$, then $e \cap f^\perp$
          is empty for all other edges $f$ of $P$.} This follows from
        the observation that if $e \cap f^\perp \neq \emptyset$, then
        also $-e \cap f^\perp \neq \emptyset$, which contradicts
        $f \cap f^\perp = \{c(f)\} \neq \emptyset$, since at most two
        edges of $P$ intersect any line.
    \item \emph{Let $q \defeq c(-e)$. If $\{q\} = e^\perp \cap f$ for
          an edge $f = [u, v]$, we have $||x-y|| \leq ||u-v||$ with
          equality if and only if $f = -e$.}
        By the Chord Theorem from plane geometry, we
        have
        \begin{align*}
            ||u - q|| \cdot ||q - v|| &= ||(-x) - q|| \cdot ||q -
            (-y)|| = \tfrac{1}{4} ||x-y||^2\,.
        \intertext{
        The inequality of arithmetic and geometric means now yields
        }
            ||u - q|| \cdot ||q - v|| &\leq  \big(\tfrac{1}{2}(||u-q|| +
            ||q-v||)\big)^2 = \tfrac{1}{4}(||u-v||)^2\,,
        \end{align*}
        with equality if and only if $||u-q|| = ||q-v||$, i.e. if
        $f = -e$.
    \end{enumerate}
    Now, assume that $e_1$ is an edge of $P$ such that $-e_1$ is not
    an edge of $P$. Because of (i), $e_1^\perp$ intersects another edge
    $e_2 \neq e_1$ of $P$. Iterating this, there exists an edge
    $e_{n+1} \neq e_n$ intersecting $e_n^\perp$ for all $n \geq 1$. By
    (ii), we will never have $e_{n+1} = -e_n$, so by (iii) the lengths of
    the $e_i$'s will strictly increase, a contradiction.
\end{proof}

Theorem~\ref{thm:restr_insc} furnishes further examples of inscribed
arrangements by restrictions of reflection arrangements. These
restrictions are well understood, see \cite[Section 6.5]{OrlikTerao},
forming two infinite families and several sporadic examples (all of
rank $\leq 8$). A Hasse diagram of the restrictions is given in
Figure~\ref{fig:restrictions}. For example, there are $17$ such
restrictions of rank $3$ and Figure~\ref{fig:inscribed_zonotopes}
shows inscribed zonotopes for each one of them (we refer to
Section~\ref{sec:simplicial_arrangements} for the naming scheme).

\begin{figure}
    \captionsetup[subfigure]{labelformat=empty}
    \begin{center}
        \begin{subfigure}{\textwidth/6}
            \includegraphics[width=\textwidth]{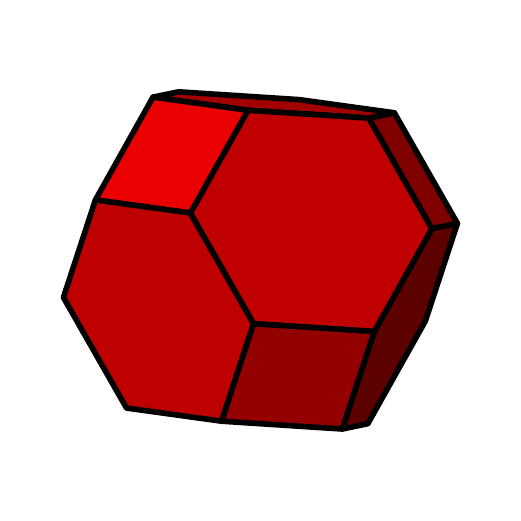}%
            \caption{$\Arr_3(6, 1) = A_3$}
        \end{subfigure}%
        \begin{subfigure}{\textwidth/6}
            \includegraphics[width=\textwidth]{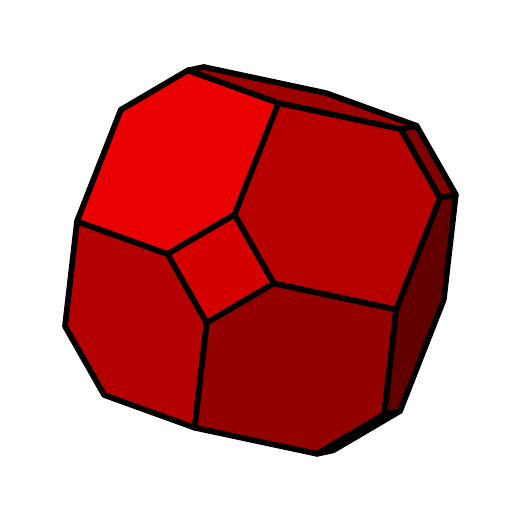}%
            \caption{$\Arr_3(7, 1)$}
        \end{subfigure}%
        \begin{subfigure}{\textwidth/6}
            \includegraphics[width=\textwidth]{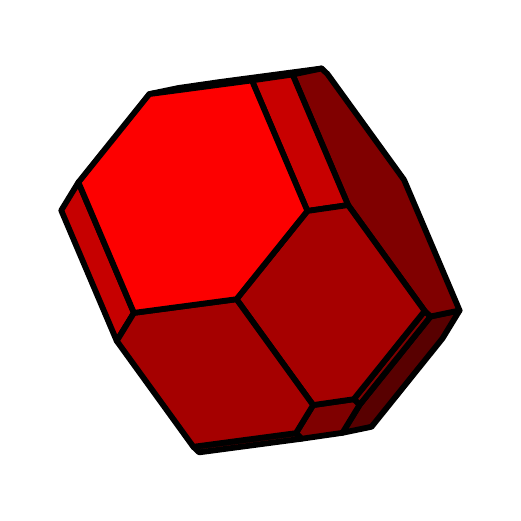}%
            \caption{$\Arr_3(8, 1)$}
        \end{subfigure}%
        \begin{subfigure}{\textwidth/6}
            \includegraphics[width=\textwidth]{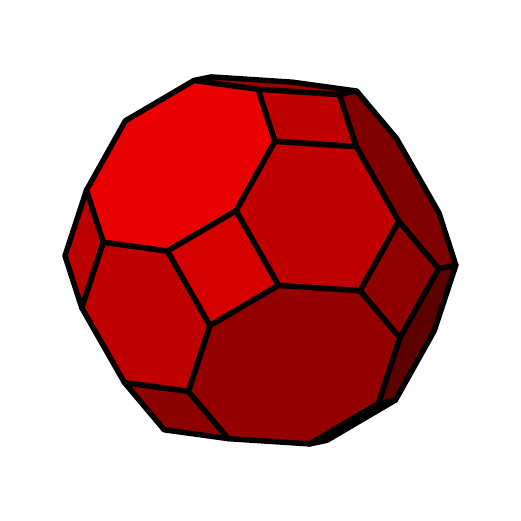}%
            \caption{$\Arr_3(9, 1) = B_3$}
        \end{subfigure}%
        \begin{subfigure}{\textwidth/6}
            \includegraphics[width=\textwidth]{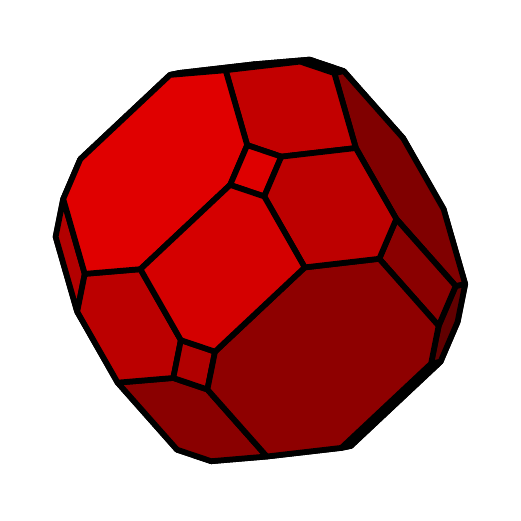}%
            \caption{$\Arr_3(10, 2)$}
        \end{subfigure}%
        \begin{subfigure}{\textwidth/6}
            \includegraphics[width=\textwidth]{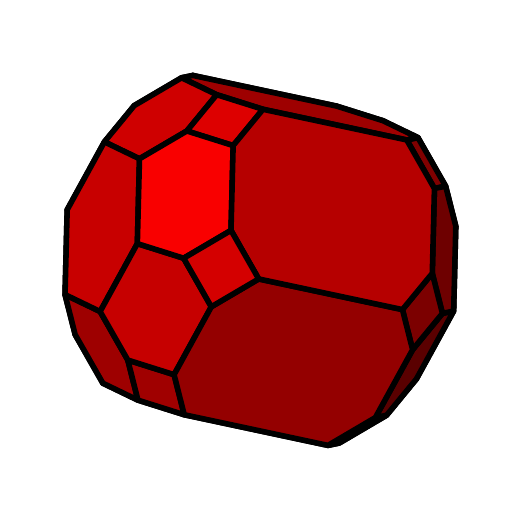}%
            \caption{$\Arr_3(10, 3)$}
        \end{subfigure}

        \begin{subfigure}{\textwidth/6}
            \includegraphics[width=\textwidth]{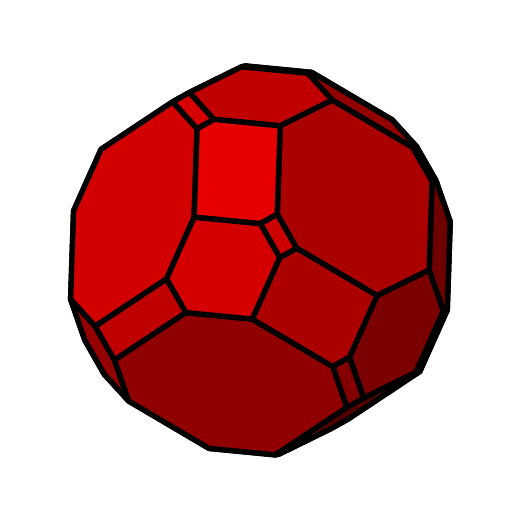}%
            \caption{$\Arr_3(11, 1)$}
        \end{subfigure}%
        \begin{subfigure}{\textwidth/6}
            \includegraphics[width=\textwidth]{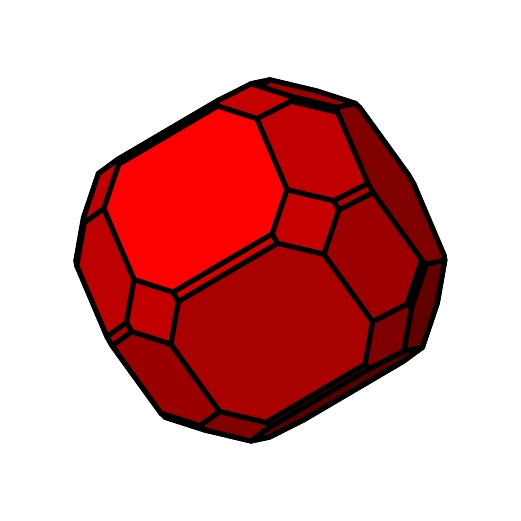}%
            \caption{$\Arr_3(13, 1)$}
        \end{subfigure}%
        \begin{subfigure}{\textwidth/6}
            \includegraphics[width=\textwidth]{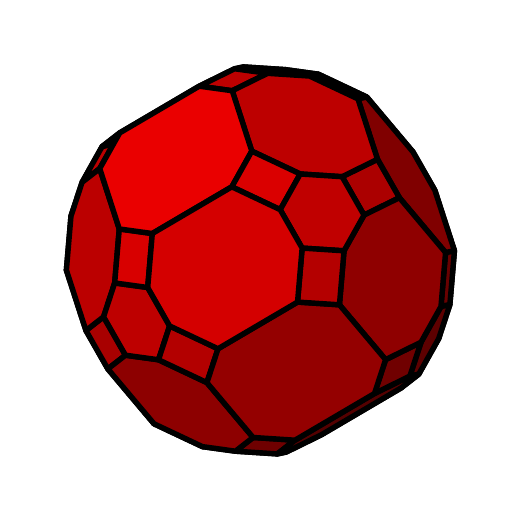}%
            \caption{$\Arr_3(13, 2)$}
        \end{subfigure}%
        \begin{subfigure}{\textwidth/6}
            \includegraphics[width=\textwidth]{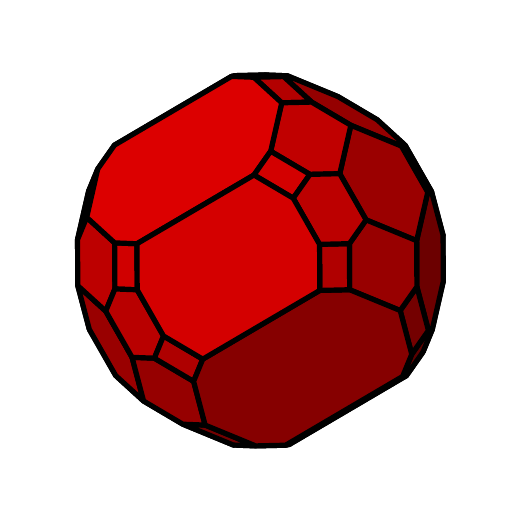}%
            \caption{$\Arr_3(13, 3)$}
        \end{subfigure}%
        \begin{subfigure}{\textwidth/6}
            \includegraphics[width=\textwidth]{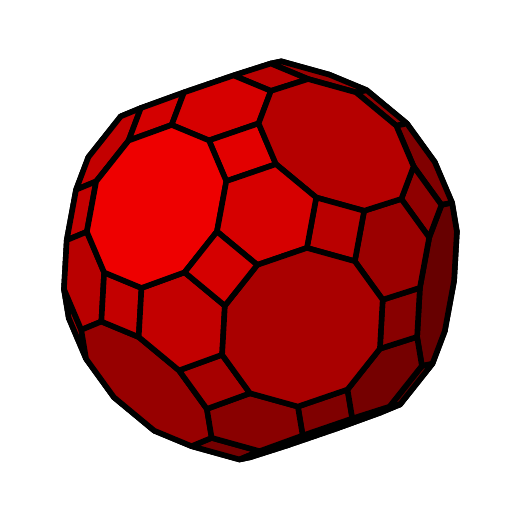}%
            \caption{$\Arr_3(15, 1) = H_3$}
        \end{subfigure}%
        \begin{subfigure}{\textwidth/6}
            \includegraphics[width=\textwidth]{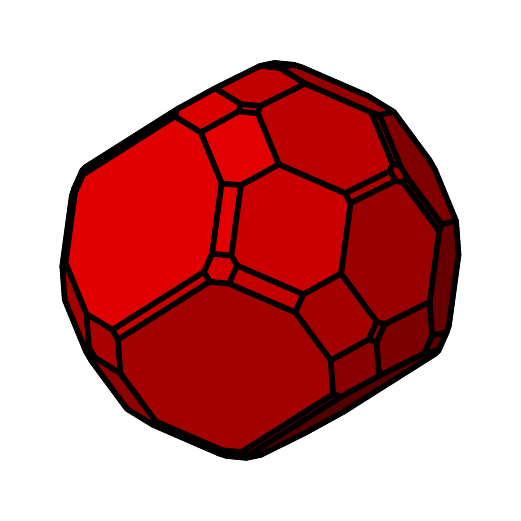}%
            \caption{$\Arr_3(16, 3)$}
        \end{subfigure}%

        \begin{subfigure}{\textwidth/6}
    \includegraphics[width=\textwidth]{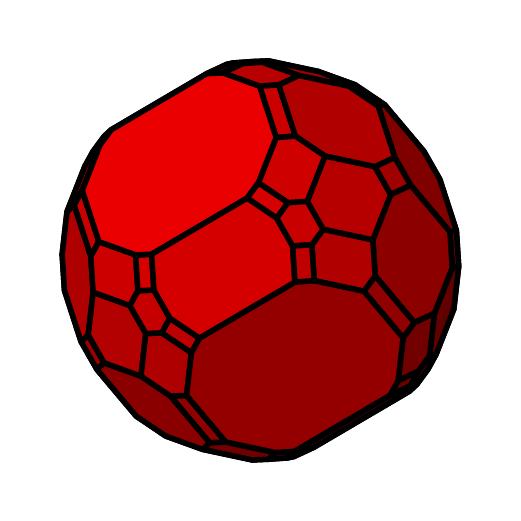}%
            \caption{$\Arr_3(17, 2)$}
        \end{subfigure}%
        \begin{subfigure}{\textwidth/6}
    \includegraphics[width=\textwidth]{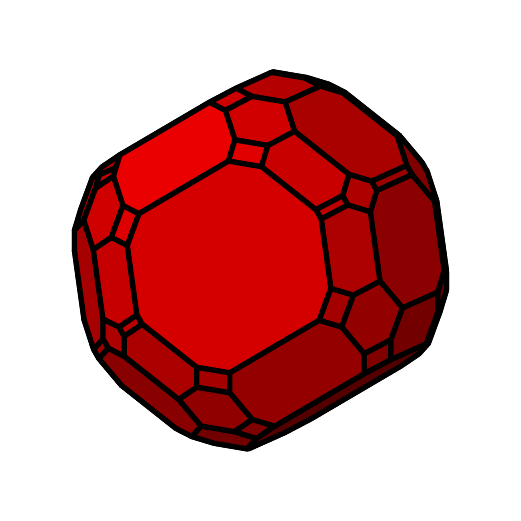}%
            \caption{$\Arr_3(17, 4)$}
        \end{subfigure}%
        \begin{subfigure}{\textwidth/6}
    \includegraphics[width=\textwidth]{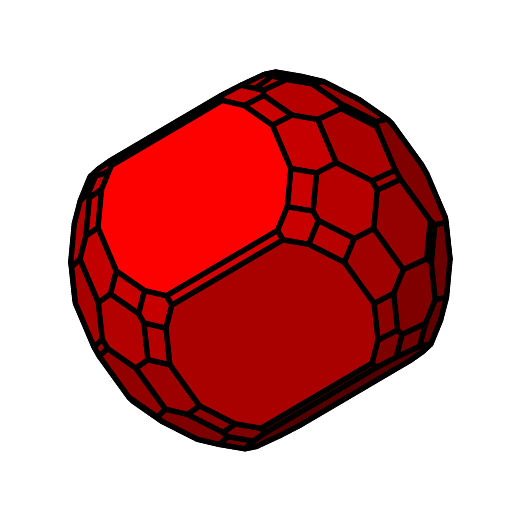}%
            \caption{$\Arr_3(19, 1)$}
        \end{subfigure}%
        \begin{subfigure}{\textwidth/6}
    \includegraphics[width=\textwidth]{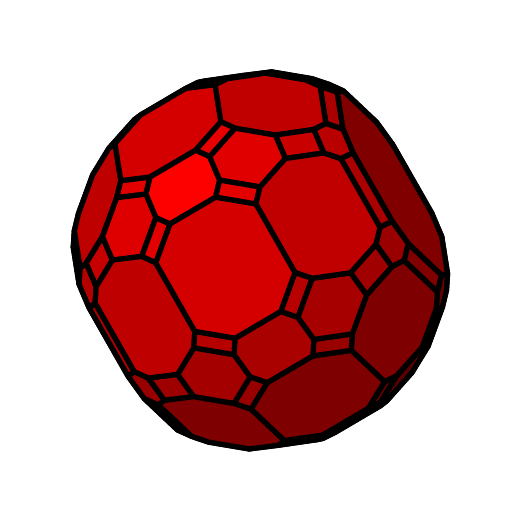}%
            \caption{$\Arr_3(19, 3)$}
        \end{subfigure}%
        \begin{subfigure}{\textwidth/6}
    \includegraphics[width=\textwidth]{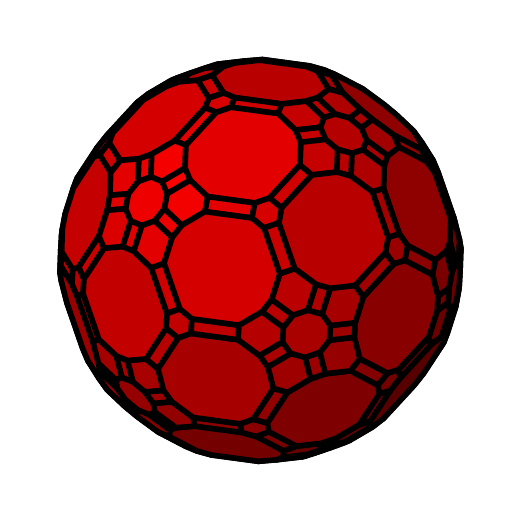}
            \caption{$\Arr_3(31, 1)$}
        \end{subfigure}%

        \caption{An inscribed zonotope for each of the combinatorial
          type of rank~$3$ which can be strongly inscribed.}
        \label{fig:inscribed_zonotopes}
    \end{center}
\end{figure}

\newcommand\PM{\{-1,+1\}}%
\newcommand\PMZ{\{-1,0,+1\}}%
\newcommand\TG{\mathcal{T}}%
\subsection{Representation of the strongly inscribed cone}\label{sec:repr}

In this section we give a simple representation of the strongly inscribed cone
and we prove a characterization of inscribed zonotopes that generalizes that
of Bolker. 

Let $\Arr = \{H_1,\dots,H_n\}$ be an essential arrangement of $n$ distinct
hyperplanes with $H_i = z_i^\perp$ for $i=1,\dots,n$. We assume that
$z_1,\dots,z_n$ are contained in an open halfspace. For every region $R$ of
$\Fan(\Arr)$ there is unique $\sigma = (\sigma_1,\dots,\sigma_n) \in \PM^n$
such that
\[R \ = \ \{ x \in \R^d : \inner{\sigma_iz_i,x} \ge 0 \text{ for } i=1,\dots,n\}\,.\]
Such sign vectors are called \Def{topes}. The tope graph
$\TG(\Arr)$ is the undirected simple graph with nodes given by topes and edges
$\sigma\sigma' \in E(\TG)$ if and only if the corresponding regions are
adjacent. Two topes $\sigma,\sigma'$ represent adjacent regions if and only if
the separation set $S(\sigma,\sigma') \defeq \{ j \in [n] : \sigma_j \neq
\sigma'_j\}$ of the corresponding topes contains a single index $j$. We can
identify edges by sign vectors $\tau \in \PMZ^n$ with a unique zero entry at
$j = j(\tau)$ and such that $\tau \pm e_j$ are topes. 

Let $\lambda = (\lambda_1, \lambda_2,\dots,\lambda_n) \in \R^n$ and let $Z_\lambda$ be the corresponding (virtual) zonotope, where
\[
    Z_\lambda \ \defeq \ \lambda_1 [-z_1,z_1] + \lambda_2 [-z_2,z_2] + \cdots +
    \lambda_n [-z_n,z_n]\,.
\]

Every vertex of $Z_\lambda$ is of the form
$v_\sigma \defeq \sigma_1\lambda_1z_1 + \cdots +\sigma_n\lambda_n z_n$
for a unique tope $\sigma$, and every edge
$e = [v_\sigma, v_{\sigma'}]$, represented by
$\tau = \frac{1}{2}(\sigma + \sigma')$, has center
$c(e) = \sum_i \tau_i \lambda_i z_i$. Let
$T(\Arr) \in \R^{E(\TG) \times n}$ be the matrix with entries
$T_{\tau, i} = \tau_i \inner{z_{j(\tau)},z_i}$ for $\tau \in E(\TG)$
and $i=1,\dots,n$.  Applying Lemma~\ref{lem:edge_midpoint_orthogonal}
to all edges of $Z_\lambda$ yields:
\begin{prop}\label{prop:Zmatrix}
    Let $\Arr = \{ z_1^\perp, \dots, z_n^\perp\}$ be a hyperplane arrangement
    in $\R^d$ and let $T = T(\Arr)$ as above. Then 
    \begin{align*}
        \ZInSpc(\Arr) \ &\cong \ \{\lambda \in \R^n : T\lambda = 0 \}
         \quad \text{and} \\
        \ZInCone(\Arr) \ &\cong \ \{\lambda \in \Rpp^n : T\lambda = 0 \} \, .
    \end{align*}
\end{prop}
Note that if $\dim \Arr = 2$, then up to a permutation of rows
$T(\Arr)$ and $\sR$ as defined in Section~\ref{sec:2d_algebraic} are
identical.

We further analyze the kernel of $T$. For fixed $j=1,\dots,n$, let $T^j$ the
matrix with rows given by those $\tau \in E(\TG)$ such that $j(\tau) = j$.
Define the diagonal matrix $D^j \in \R^{n \times n}$ with diagonal
$\inner{z_j,z_i}$ for $i=1,\dots,n$. Then $T\lambda = 0$ if and only if
$T^jD^j\lambda = 0$ for all $j=1,\dots,n$. The rows of $T^j$ are strongly
related to the topes of the restriction $\Arr_{H_j}$. Every hyperplane in
$\Arr_{H_j}$ is of the form $L = H_s \cap H_j$ for some $s \neq j$. If $H_s
\cap H_j = H_t \cap H_j$, then the corresponding columns $s$ and $t$ of $T^j$
are either identical or differ by $-1$, which is the case when $z_s$ and $z_t$
induce different orientations on $L$. Choose an orientation for every
hyperplane of $\Arr_{H_j}$ and for $s \neq j$ set $\rho^j_s = 1$ if $H_s \cap
H_j$ induces the chosen orientation and $\rho^j_s = -1$ otherwise.

We will need the following technical fact.
\begin{lem}\label{lem:topes}
    The collection of topes of an arrangement of $n$ distinct hyperplanes has
    full rank $n$.
\end{lem}
\begin{proof}
    By replacing some $z_i$ by $-z_i$ if necessary, we may assume that $\sigma
    = (+1,\dots,+1)$ is a tope of $\Arr$. Pick a shortest path in the tope
    graph of $\Arr_{H_j}$ from $\sigma$ to $-\sigma$. Up to relabelling
    hyperplanes if necessary, the
    sequence of topes visited along the path is $(+1,\dots,+1),
    (+1,\dots,+1,-1)$, $ (+1,\dots,+1,-1,-1), \dots, (-1,\dots,-1)$, which is a
    collection of vectors of full rank $n$.
\end{proof}

Lemma~\ref{lem:topes} implies $T^jD^j \lambda = 0$ if and only if for every $L
\in \Arr_{H_j}$
\begin{equation}\label{eqn:T^j_kernel}
    \sum_{s \neq j,L \subseteq H_s} \rho^j_s \inner{z_j,z_s} \lambda_s \ = \ 0
    \, .
\end{equation}

Now $L \in \Arr_{H_j}$ is a codimension-$2$-flat of $\Arr$. We identify $L$
with the ordered sequence $L = (z_{i_1},\dots,z_{i_k})$ such that $L \subset H_s$
if and only if $s \in \{i_1,\dots,i_k\}$ and such that the vector
configuration $z_{i_1}, z_{i_2}, \dots, z_{i_k}, -z_{i_1}, -z_{i_2}, \dots,
-z_{i_k}$ is cyclically ordered in the $2$-dimensional subspace $L^\perp$. We
refer to this tuple as an \Def{ordered} codimension-$2$ flat.
We may now choose $\rho_{i_s}^{i_j} = -1$ if $i_s < i_j$ and $=+1$ otherwise.
Lastly, let $\sR_L \defeq \sR(z_{i_1}, z_{i_2}, \dots, z_{i_k})$ be the
skew-Gram matrix and define $\lambda_L = (\lambda_{i_1},\dots,\lambda_{i_k})$.
Then combining conditions~\eqref{eqn:T^j_kernel} for all $j=1,\dots,k$ then
yields the following representation.

\begin{thm}\label{thm:compute_by_2flats}
    Let $\Arr = \{ H_i = z_i^\perp : i=1,\dots,n\}$ be an arrangement in
    $\R^d$. Then
    \begin{align*}
      \ZInSpc(\Arr) \ &\cong \ \{\lambda \in \R^n : \sR_L \lambda_L = 0, L
        \in \flats_{d-2}(\Arr)\} \quad \text{and} \\
      \ZInCone(\Arr) \ &\cong \ \{\lambda \in \Rpp^n : \sR_L \lambda_L = 0, L
        \in \flats_{d-2}(\Arr)\}\,.
    \end{align*}
\end{thm}

There is nice interpretation of the conditions of
Theorem~\ref{thm:compute_by_2flats}: If $F \subset Z_\lambda$ is a $2$-face
with $\aff_0(F)^\perp = L$, then $F$ is a translate of $
\lambda_{i_1}[-z_{i_1},z_{i_1}] + \lambda_{i_2}[-z_{i_2},z_{i_2}] + \cdots +
\lambda_{i_k}[-z_{i_k},z_{i_k}]$. The{\nobreak}orem~\ref{thm:inspc_eq_ker_R} yields that
$Z_\lambda$ is inscribed if and only if all $2$-faces of $Z_\lambda$ are
inscribed. Combining this with Theorem~\ref{thm:charac_zono}, this gives an
extension of Bolker's characterization of zonotopes.
\begin{thm}\label{thm:bolker_inscr}
    A (virtual) polytope $P$ is an inscribed (virtual) zonotope if and only if
    all $2$-dimensional faces are inscribed and centrally-symmetric.
\end{thm}

Using description of inscribed cones for general fans
in~\cite[Theorem~5.20]{InFan1} with the additional constraint
  that edges in the same belt have the same length gives a different
  path to Theorem~\ref{thm:compute_by_2flats}.  Since $\Fan(\Arr)$ has
  typically many edges, this approach is much more involved. A
  byproduct to our approach is a simple way to decide if an
  arrangement is virtually inscribable.

\begin{cor}\label{cor:pfaffian}
    Let $\Arr$ be an arrangement in $\R^d$. We have
    $\InSpc(\Arr) \cong \R^d$ if and only if $\pfaff \sR_L = 0$ for all
    $L \in \flats_{d-2}(\Arr)$.
\end{cor}
\begin{proof}
    It follows from Corollary~4.14 in~\cite{InFan1} that $\Arr$ is
    virtually inscribable if and only if the localizations $\Arr_L$
    are inscribable for all $L \in \flats_{d-2}(\Arr)$. Thus, the
    result follows from Theorem~\ref{thm:2d_strongly_insc} together
    with Theorem~\ref{thm:inspc_eq_ker_R}.
\end{proof}

\section{Simplicial arrangements}\label{sec:simplicial}

A hyperplane arrangement is \Def{simplicial}, if all of its regions are
simplicial cones, that is, spanned by linearly independent vectors. It is
well-known~\cite[Sect.~1.15]{Humphreys} that reflection arrangements are
simplicial. Restrictions and localizations preserve simpliciality. Therefore,
all examples of inscribed arrangements given so far are simplicial. This is
true for general inscribed arrangements. We use this to prove the extension of
Bolker's characterization to inscribed belt polytopes and zonotopes
(Theorem~\ref{thm:bolker_inscr}).

Simplicial arrangements are fascinating but rare.  They arise, for example, in
the study of reflection arrangements (see the comments in
Section~\ref{sec:rest_simp}) as well as in the classification of simply
connected Cartan schemes~\cite{CuntzHeckenberger, HeckenbergerWelker} by way
of \emph{crystallographic} arrangements~\cite{Cuntz-crystallographic}.  Via
localization and restriction, this is a rich source of simplicial arrangements
of any rank.  Gr\"unbaum~\cite{Grunbaum} studied simplicial arrangements of
rank $3$ (as line arrangements in the projective plane) and gave a list of two
infinite families and $90$ sporadic examples. Five further arrangements were
discovered by Cuntz~\cite{Cuntz_27lines, Cuntz_Greedy} by computer search. It
is believed that the two infinite families together with the $95$ sporadic
examples give a complete list of the combinatorial types of simplicial
arrangements in rank $3$. Apart from restrictions of reflection arrangements we
are aware of only two further simplicial arrangements~\cite{Geis_RP3}, both in
rank $4$.  

We adapt the naming scheme of Gr\"unbaum~\cite{Grunbaum_Simplicial} and denote
simplicial arrangements by $\Arr_r(n, k)$, where $k$ is the position in the
list of simplicial arrangements in rank $r$ with $n$ hyperplanes. For example,
the two infinite families in rank $3$ will be denoted by $\Arr_3(2n, 1)$ and
$\Arr_3(4m + 1, 1)$. 

We confirm that all known simplicial
arrangements in rank $3$ are projectively unique. Assuming that the
Gr\"unbaum--Cuntz catalog is complete, we show in
Section~\ref{sec:projectively_unique} that simplicial arrangements of all
ranks $\ge3$ are projectively unique and that there are only finitely many
irreducible strongly inscribed arrangements of every rank $\ge 3$.

\subsection{Inscribable arrangements are simplicial}
\label{sec:simplicial_arrangements}

In this section we prove Theorem~\ref{thm:insc_arr_are_simplicial}.  Recall
that a $d$-polytope $P$ is \Def{simple} if every vertex is incident to $d$
edges.  We will need the following result on polyhedral cones.

\begin{lem}\label{lem:obtuse_vectors}
    Let $a_1, \dots, a_d \in \R^d$ be linearly independent vectors
    such that $\inner{a_i, a_j} \leq 0$ for all $1 \leq i < j \leq d$.
    Let $b_1, \dots, b_d$ be such that
    \[
        \cone(b_1, \dots, b_d) \ = \ \{ x \in \R^d : \inner{a_i, x}
        \leq 0 \text{ for $1 \leq i \leq d$}\}\,.
    \]
    Then $\inner{b_i, b_j} \geq 0$ for all $1 \leq i < j \leq d$.
\end{lem}
\begin{proof}
    We prove the claim by induction on $d$. The case $d \le 2$ is
    clear.  We may assume that $\inner{a_i,a_i} = 1$ for all $i$.  Let
    $C$ be the stated simplicial cone.  Consider the facet
    $C' \defeq \{ x \in C : \inner{a_d,x} = 0\} =
    \cone(b_1,\dots,b_{d-1})$. This is a $(d-1)$-dimensional cone in
    the hyperplane $H = a_d^\perp$ given by
    \[
        C' \ = \ \{ x \in H : \inner{\bar a_i, x} \le 0 \text{ for }
        i=1,\dots,d-1 \} \, ,
    \]
    where $\bar a_i = a_i - \inner{a_{d},a_i}a_d$ is the orthogonal
    projection of $a_i$ onto $H$. One checks that
    $\inner{ \bar a_i, \bar a_j } \le 0$ for all $1 \le i < j < d$ and
    thus, by induction, $\inner{b_i,b_j} \ge 0$ for all
    $1 \le i < j < d$.
    Let $\bar b_d = b_d - \inner{a_{d},b_d}a_d$ be the orthogonal projection of
    $b_d$ onto $H$. Note that for all $i < d$ we have $\inner{a_i,\bar b_d} =
    \inner{a_i,b_d} - \inner{a_{d},b_d} \inner{a_{i},a_d} \leq 0$ and hence
    $\bar b \in C'$.  Thus, there are $\mu_1,\dots,\mu_{d-1} \ge 0$ such that
    $b_d = \mu_1 b_1 + \cdots + \mu_{d-1} b_{d-1} +
    \inner{a_{d},b_d}a_d$. For $j < d$, it follows that
    $\inner{b_j,b_d} = \sum_{i = 1}^{d-1} \mu_{i} \inner{b_j, b_i} \ge
    0$, which finishes the proof.
\end{proof}

A \Defn{ridge} $F$ of $P$ is a face of codimension $2$. There are
precisely two facets $G, G'$ of $P$ containing $F$. The \Defn{dihedral
  angle} of $F$ is the angle between $G$ and $G'$. If $b$ and $b'$ are
outer facet normals of $G$ and $G'$ respectively, then the dihedral
angle is $\pi - \measuredangle(b, \0, b')$. In particular the dihedral
angle is non-acute, if $\inner{b, b'} \geq 0$. We call a polytope
\Def{non-acute}, if the dihedral angles at all ridges are non-acute.

\begin{proof}[Proof of Theorem~\ref{thm:insc_arr_are_simplicial}]
    We proceed by induction on the dimension $d$. The induction step
    is divided into two parts. First, assuming that every facet is
    non-acute, we show that $P$ is simple. In the second step we then
    show, that if the inscribed belt polytope $P$ is simple, then it
    is non-acute. The dihedral angles of a polygon coincide with its
    interior angles, so Corollary~\ref{cor:belt_polygons_obtuse}
    proves the base case of the induction.

    For the first step, let $P$ be a $d$-dimensional inscribed belt
    polytope and let $F$ be a $(d-3)$\nobreakdash-face of $P$. The link of $F$ in
    the face lattice is a cycle. Let $G$ be a facet of $P$ containing
    $F$. Then $F$ is a ridge of $G$ and the dihedral angle of $F$ in
    $G$ is non-acute by the induction hypothesis. Since the
    sum of all dihedral angles has to be strictly smaller than $2\pi$,
    there are precisely three facets in $P$ containing $F$, which is
    equivalent to $P$ being simple.

    Thus far, we can assume that $P$ is a simple and inscribed belt
    polytope. To show that the dihedral angles of $P$ are non-acute,
    let $v \in V(P)$ be a vertex, and let
    $a_1, \dots, a_d \in \R^d$ be the $d$ edge directions emanating
    from $v$. By Corollary~\ref{cor:belt_polygons_obtuse}, we have
    $\inner{a_i, a_j} \leq 0$ for all $1 \leq i < j \leq d$. Let
    $G_i$, $1 \leq i \leq d$, be the facets of $P$ containing $F$ and
    let $b_i$ be their outer normal vectors. Then the dihedral angle of
    $G_i \cap G_j$ is non-acute if and only if
    $\inner{b_i, b_j} \leq 0$, which is the conclusion of
    Lemma~\ref{lem:obtuse_vectors}.
\end{proof}

\begin{rem}
    The proof of Theorem~\ref{thm:insc_arr_are_simplicial} is
    considerably simpler for \emph{strongly} inscribed arrangements. Let $Z$
    be an inscribed zonotope and $F$ a ridge of $Z$. By
    Theorem~\ref{thm:restr_insc}, the orthogonal projection of $Z$
    along $F$ is an (non-acute) inscribed zonogon, showing that the
    dihedral angle of $F$ is non-acute. This removes the need for
    Lemma~\ref{lem:obtuse_vectors}.
\end{rem}

The proof does not go through for virtually inscribed zonotopes.

\begin{quest}
    Is there a non-simplicial arrangement $\Arr$ that is (strongly) virtually
    inscribable?
\end{quest}

The simplicity of inscribed zonotopes is the key to the following
refinement of Theorem~\ref{thm:charac_zono}, which extends
Theorem~\ref{thm:bolker_inscr} to belt polytopes.
\begin{thm}\label{thm:bolker_inscr_belt}
    A polytope $P$ is an inscribed belt polytope if and only if all
    $2$-dimensional faces are inscribed belt polygons.
    A polytope $P$ is an inscribed zonotope if and only if all $2$-dimensional
    faces are inscribed and centrally-symmetric.
\end{thm}

We recall that for simple polytopes inscribability is governed by
faces of fixed dimension.
\begin{cor}[{\cite[Corollary~4.10]{InFan1}}]\label{cor:simple_inscribed}
    Let $P \subseteq \R^d$ be a simple $d$-polytope and $2 \leq k \leq d$.
    Then $P$ is inscribed if and only if all its $k$-faces are inscribed.
\end{cor}
\begin{proof}[Proof of Theorem~\ref{thm:bolker_inscr_belt}]
    If all $2$-faces are belt polygons (centrally-symmetric), then $P$ is a
    belt polytope (zonotope) by Theorem~\ref{thm:charac_zono}. Furthermore, by
    Corollary~\ref{cor:belt_polygons_obtuse}, the interior angles of every
    $2$-face are $\geq \frac{\pi}{2}$. Using the same argument as in the proof
    of Theorem~\ref{thm:insc_arr_are_simplicial} we see that $P$ is simple,
    thus by Corollary~\ref{cor:simple_inscribed} it is inscribed.
\end{proof}

\newcommand{\RSpc}{\mathfrak{R}}%
\subsection{Projective uniqueness of simplicial
arrangements}\label{sec:projectively_unique}
We call arrangements $\Arr$ and $\Arr'$ \Def{(combinatorially) isomorphic}
($\Arr \cong \Arr'$) if $\Fan(\Arr) \cong \Fan(\Arr')$. In that sense the
following has been conjectured by Gr\"unbaum and Cuntz.

\begin{conj}[Gr\"unbaum~\cite{Grunbaum_Simplicial},
    Cuntz~\cite{Cuntz_Greedy}]\label{conj:GC}
    Every simplicial arrangement of rank $3$ is isomorphic to an arrangement
    in the Gr\"unbaum--Cuntz catalog.
\end{conj}

For $g \in \GL(\R^d)$, let $g\Arr \defeq \{g H : H \in \Arr \}$. Clearly $g
\Arr \cong \Arr$ for all $g \in \GL(\R^d)$. We define
\[ 
    \RSpc(\Arr) \defeq \{ \Arr' \text{ arrangement in $\R^d$} : \Arr \cong
    \Arr' \} \, / \, \GL(\R^d)  \, .
\]
An arrangement $\Arr$ is \Defn{projectively unique}, if $|\RSpc(\Arr)| = 1$,
that is, if for any arrangement $\Arr' \cong \Arr$ there is $g \in \GL(\R^d)$
such that $\Arr' = g\Arr$.

It was shown in~\cite{Cuntz_MinimalFields} that $|\RSpc(\Arr)|$ is finite for
all examples of simplicial arrangements of rank $3$. In fact, using methods
in~\cite{Cuntz_MinimalFields}, one computationally confirms the following for
the sporadic examples.  For the two infinite families this is explicitly
stated in~\cite{Cuntz_MinimalFields}.

\begin{prop}\label{prop:proj_unique}
    All arrangements in the Gr\"unbaum--Cuntz catalog are projectively unique.
\end{prop}

Conditional on the Gr\"unbaum--Cuntz conjecture, this implies projective
uniqueness in all dimensions $\ge 3$.

\begin{cor}\label{cor:all_proj_unique}
    If Conjecture~\ref{conj:GC} holds, then every simplicial arrangement of
    rank $\ge 3$ is projectively unique.
\end{cor}

The proof of Corollary~\ref{cor:all_proj_unique} depends on the following
observation that we did not find in the literature and that might be of
independent interest.

\begin{lem}\label{lem:moduli_size}
    Let $\Arr$ be an arrangement of rank $d \ge 3$. Then
    $|\RSpc(\Arr)| \leq \prod_{H \in \Arr} |\RSpc(\Arr^H)|$.
\end{lem}
\begin{proof}
    We may assume that $\Arr$ is an essential arrangement in $\R^d$ and
    $|\RSpc(\Arr^H)| < \infty$ for all $H \in \Arr$. Let $H_1,\dots,H_d \in
    \Arr$ be hyperplanes with $H_1 \cap \cdots \cap H_d = \{\0\}$. Up to
    linear transformation, we may assume that $H_i = \{ x \in \R^d: x_i = 0
    \}$ for $i=1,\dots,d$. This gives a canonical realization of $\Arr$ within
    $\RSpc(\Arr)$ and, in fact, canonical realizations of $\Arr^{H_i} \in
    \RSpc(\Arr^{H_i})$ for $i=1,\dots,d$. Since $d \ge 3$, we have     \[
        H \ = \ 
        (H_1 \cap H) + 
        (H_1 \cap H) + 
        \cdots  + 
        (H_d \cap H)  \, .
    \]
    for every $H \in \Arr$ and hence $\Arr$ can be reconstructed from
    $\Arr^{H_i}$, $i=1,\dots,d$.  This gives an embedding $\RSpc(\Arr)
    \hookrightarrow \prod_{i=1}^d \RSpc(\Arr^{H_i})$ and yields the claim.
\end{proof}

Proposition~\ref{prop:infinite_families} in Section~\ref{sec:infinite_families}
implies that only finitely many arrangements from the Gr\"unbaum--Cuntz
catalog are strongly inscribable.

\begin{thm}\label{thm:finite_in_higher_dimensions}
    If Conjecture~\ref{conj:GC} holds, then there are only finitely many
    combinatorial types of irreducible and strongly inscribable arrangements
    in every dimension $\ge 3$.
\end{thm}

Let $m$ be the maximal number of hyperplanes that meet a
codimension-$2$ flat of an irreducible and strongly inscribable
arrangement of rank $3$.  Equivalently, $2m$ is the maximal number of
edges of a $2$-face of an irreducible and inscribed $3$-dimensional
zonotopes. (If Conjecture~\ref{conj:GC} is true, then the results of
Section~\ref{sec:inner_product} show that $m = 12$.) We call a codimension-$2$ flat
\Def{large} if it is contained in more than $m$ hyperplanes.

\begin{lem}\label{lem:large_2_face_reducible}
    Assume that $m$ is finite. Let $\Arr$ be a strongly inscribable
    arrangement of rank~$d$. If $\Arr$ has a large codimension-$2$
    flat, then $\Arr$ is reducible.
\end{lem}
\begin{proof}
    Let $\Arr = \{H_1, \dots, H_n\}$ and w.l.o.g.\ let
    $L = H_1 \cap \cdots \cap H_k$ be a large codimension-$2$ flat
    with $k > m$. For any $i > k$, consider the codimension-$3$ flat
    $M = L \cap H_i$. The localization $\Arr_M$ is a strongly
    inscribed arrangement of rank $3$ with the large codimension-$2$
    flat $L$ and therefore reducible. Hence $H_i \perp L$ for
    $i = k+1, \dots, n$ and thus
    \[
        \Arr \ = \ \ess \{H_1, \dots, H_k\} \times \ess \{H_{k+1}, \dots, H_n\}\,.\qedhere
    \]
\end{proof}

The last ingredient for the proof of
Theorem~\ref{thm:finite_in_higher_dimensions} is the following result in
\cite{CuntzMucksch}:
\begin{lem}[{\cite[Lemma 3.11]{CuntzMucksch}}]
    \label{lem:irreducible_closed_under_restrictions}
    Let $\Arr$ be a simplicial and irreducible hyperplane arrangement.
    Then $\Arr^L$ is irreducible for all $L \in \flats(\Arr)$.
\end{lem}

\begin{proof}[Proof of Theorem~\ref{thm:finite_in_higher_dimensions}]
    Let $N_d$ be the maximal number of hyperplanes of an irreducible
    inscribable arrangement of rank $d$. It suffices to show that $N_d <
    \infty$ for all $d \geq 3$. Assuming the validity of
    Conjecture~\ref{conj:GC}, it follows from 
    Proposition~\ref{prop:infinite_families} that there are only 
    finitely many combinatorial types of irreducible and inscribable
    arrangements of rank $3$ and so $N_3$ is finite. 

    Let $\Arr$ be an irreducible inscribable arrangement of rank $d > 3$ and
    let $H \in \Arr$. The restriction $\Arr^H$ is irreducible and inscribable
    by Lemma~\ref{lem:irreducible_closed_under_restrictions} and
    Theorem~\ref{thm:restr_insc}. Any $H' \in \Arr \setminus \{H\}$ meets $H$
    in a hyperplane of $\Arr^H$. Conversely, any 
    $J \in \Arr^H$ is a codimension-$2$ flat of $\Arr$ and hence is contained
    in at most $m$ hyperplanes by Lemma~\ref{lem:large_2_face_reducible}. We
    compute
    \[
        |\Arr| \ \leq \ (m-1) \cdot |\Arr^H| + 1 \ \leq \ (m-1) N_{d-1} + 1
    \]
    and hence $N_d$ is finite.
\end{proof}

\section{Inner products and algebraic computations}\label{sec:inner_product}

Let $\Arr$ be a (simplicial) arrangement of $n$ hyperplanes in $\R^d$. In this
section we consider the question when $\Arr$ is inscribable up to a change of
coordinates, that is, if there is a $g \in \GL(\R^d)$ such that $g\Arr$ is
inscribable. In light of Corollary~\ref{cor:all_proj_unique}, this allows us
to treat the combinatorial types of all known simplicial arrangements of rank
$3$. For this, we propose a change in perspective.

Throughout $Q$ denotes a real symmetric $d \times d$-matrix and we write
$\inner{x, y}_Q \defeq x^t Q y$ for the associated bilinear form.  A fan
$\Fan$ is \Def{(virtually) $Q$-inscribable} if there is a (virtual) polytope
$P$ with $\Fan(P) = \Fan$ such that $V(P) \subset E_Q \defeq  \{ x \in \R^d :
\inner{x,x}_Q = 1 \}$. If $Q$ is positive definite, then $E_Q$ is an
ellipsoid.  The following lemma is apparent.

\begin{lem}\label{lem:innerprod}
    Let $\Fan$ be a complete fan in $\R^d$ and $g \in \GL(\R^d)$. Then $g\Fan$
    is (strongly or virtually) inscribable if and only if $\Fan$ is (strongly
    or virtually) $Q$-inscribable with respect to the positive definite matrix
    $Q \defeq g^tg$.
\end{lem}

Testing if $\Arr$ is strongly inscribable up to linear transformation leads to
a finite set of bilinear equations in $\lambda$ and $Q$, for which we seek a
solution over $\R^n_{>0} \times \mathrm{PSD}_d$.  We treat the two infinite
families in Section~\ref{sec:infinite_families}. For the sporadic simplicial
arrangement from the Gr\"ubaum--Cuntz catalog we employ techniques from
Gr\"obner basis theory to prove the existence or non-existence of a strongly
inscribed realization. Assuming Conjecture~\ref{conj:GC}, this shows that the
only strongly inscribable rank-$3$ arrangements are restrictions of reflection
arrangements.

Our approach also yields information about the space of $Q$'s for which
$\Arr$ is strongly $Q$\nobreakdash-in\-scribable. We determine these for
all known inscribable arrangements.

\renewcommand{\Re}{\operatorname{Re}}%
\renewcommand{\Im}{\operatorname{Im}}%
\newcommand{\regularpolygon}{\mathfrak{P}}%
\subsection{Infinite families of rank-$3$}\label{sec:infinite_families}

Recall from Section~\ref{sec:simplicial} that the two infinite
families in the Gr\"unbaum--Cuntz catalog are denoted $\Arr_3(2n,1)$
for $n \ge 3$ and $\Arr_3(4m+1,1)$ and $m \ge 2$. To construct
$\Arr_3(2n,1)$ as an arrangement of lines in the projective plane, one
starts with the $n$ lines spanned by the sides of a regular polygon
$\regularpolygon_n$ with $n$ sides and adds the $n$ lines of mirror
symmetry of $\regularpolygon_n$. To obtain $\Arr_3(4m+1,1)$ one adds
the line at infinity to $\Arr_3(4m,1)$.

We can realize $\Arr_3(2n, 1)$ as $\mathcal{R}_n \defeq \{z_0^\perp, \dots,
z_{n-1}^\perp, z_0^{\prime\perp}, \dots, z_{n-1}^{\prime\perp}\}\,,$
where
\[
    z_j \ \defeq \ (\Re(\zeta^j ),
    \Im(\zeta^j), 0), \qquad z'_j \ \defeq \
    (-\Im(\zeta^j), \Re(\zeta^j), 1)\,,
\] 
for $j = 0, \dots, n-1$ and $\zeta \defeq \operatorname{exp}(2\pi i / n)$. We
will say that hyperplanes of the form $z_j^\perp$ are of \Defn{type
\color{darkred}$M$}, corresponding to mirrors of $\regularpolygon_n$, and that those of the
form $z_j^{\prime\perp}$ are of type \Defn{type \color{darkblue}$E$},
corresponding to edges of $\regularpolygon_n$. If $n$ is even, then we distinguish those
mirrors that pass through vertices of $\regularpolygon_n$ and those that connect edge
midpoints and denote them by $M_v$ and $M_e$, respectively.  For $m \ge 2$, we
realize $\Arr_3(4m+1,1)$ as $\mathcal{R}'_m \defeq \mathcal{R}_{2n} \cup \{
(0, 0, 1)^\perp \}$ and denote $(0, 0, 1)^\perp$ by $\infty$. Examples are
shown in Figure~\ref{fig:infinite_families}.

\begin{figure}
\begin{center}
    \begin{tikzpicture}[
        dot/.style={circle,inner sep=0pt},
        eline/.style={shorten >=-5cm,shorten <=-5cm}]
        \begin{scope}[scale=0.4]
            \clip(-4.5,-4.5) rectangle (4.5,4.5);
            \node[dot] (n0) at (  0:0) {};
            \node[dot] (n1) at (  0:1) {};
            \node[dot] (n2) at ( 72:1) {};
            \node[dot] (n3) at (144:1) {};
            \node[dot] (n4) at (216:1) {};
            \node[dot] (n5) at (288:1) {};
            
            \draw[eline, darkblue] (n1) -- (n2);
            \draw[eline, darkblue] (n2) -- (n3);
            \draw[eline, darkblue] (n3) -- (n4);
            \draw[eline, darkblue] (n4) -- (n5);
            \draw[eline, darkblue] (n5) -- (n1);
            \draw[eline, darkred] (n1) -- (n0);
            \draw[eline, darkred] (n2) -- (n0);
            \draw[eline, darkred] (n3) -- (n0);
            \draw[eline, darkred] (n4) -- (n0);
            \draw[eline, darkred] (n5) -- (n0);

        \end{scope}
        \node at (0, -2.5) {$\Arr_3(2n, 1)$, $n=5$};
    \end{tikzpicture}
    \qquad
    \begin{tikzpicture}[
        dot/.style={circle,inner sep=1pt,fill},
        eline/.style={shorten >=-5cm,shorten <=-5cm}]
        \begin{scope}[scale=0.6]
            \clip(-3,-3) rectangle (3,3);
            \node (n0) at (  0:0) {};
            \node (n1) at (  0:1) {};
            \node (n2) at ( 60:1) {};
            \node (n3) at (120:1) {};
            \node (n4) at (180:1) {};
            \node (n5) at (240:1) {};
            \node (n6) at (300:1) {};

            \node (h1) at ( 30:1) {};
            \node (h2) at ( 90:1) {};
            \node (h3) at (150:1) {};
            
            \draw[eline, darkblue] (n1) -- (n2);
            \draw[eline, darkblue] (n2) -- (n3);
            \draw[eline, darkblue] (n3) -- (n4);
            \draw[eline, darkblue] (n4) -- (n5);
            \draw[eline, darkblue] (n5) -- (n6);
            \draw[eline, darkblue] (n6) -- (n1);
            \draw[eline, darkred] (n1) -- (n0);
            \draw[eline, darkred] (n2) -- (n0);
            \draw[eline, darkred] (n3) -- (n0);
            \draw[eline, darkred] (h1) -- (n0);
            \draw[eline, darkred] (h2) -- (n0);
            \draw[eline, darkred] (h3) -- (n0);

        \end{scope}
        \node at (0, -2.5) {$\Arr_3(4m+1, 1)$, $m=3$};
        \node at (2.2, 1.7) {$\infty$};
    \end{tikzpicture}
    \end{center}
    \caption{Two members of the infinite families.}
    \label{fig:infinite_families}
\end{figure}
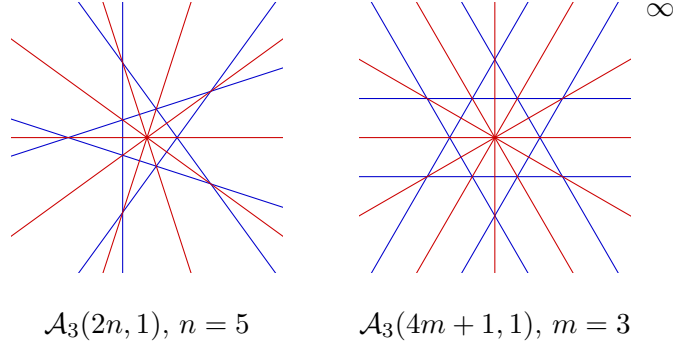

The codimension-$2$ flats (i.e., the $1$-dimensional flats) of $\mathcal{R}_n$
and $\mathcal{R}'_m$ correspond to intersection points (possibly at infinity)
of lines in the projective picture. The possible cases are illustrated in
Figure~\ref{fig:infinite_families}.  We list them according to the types
of hyperplanes in which they are contained. The subscript in $M_v$ and $M_e$ is
to be ignored, if $n$ is odd.
\begin{itemize}[$\bullet$]
\item Intersection of all $n$ lines of type $M$:
    $M_vM_eM_v \dots M_vM_e$;
\item $1$-flats of type $M_eE$;
\item $1$-flats of type $M_vEE$ and $M_eEE$;
\item For $\Arr(4m+1,1)$ there are $m$ $1$-flats of type $M_e\infty$;
\item For $\Arr(4m+1,1)$ there are $m$ $1$-flats of type $M_vE\infty E$.
\end{itemize}

The main result of this section is the following.
\begin{prop}\label{prop:infinite_families}
    The only strongly inscribable simplicial arrangements in the infinite families
    $\Arr_3(2n, 1)$ and $\Arr_3(4m+1, 1)$ are $\Arr_3(6, 1)$, $\Arr_3(8, 1)$,
    $\Arr_3(9, 1)$ and $\Arr_3(13, 1)$.
\end{prop}

\newcommand{\diag}{\operatorname{diag}}%
We first derive conditions on $Q$ for which $\mathcal{R}_n$ and
$\mathcal{R}'_m$ are virtually $Q$-inscribable.

\begin{lem}\label{lem:inf_family_Q}
    Let $\Arr = \mathcal{R}_n$, $n \geq 5$ or $\Arr = \mathcal{R}'_m$, $m \geq
    3$. Then $\Arr$ is virtually $Q$-inscribable if and only if $Q = \diag(a, a,
    t)$, where $a, t \in \Ri$.
\end{lem}

\begin{proof}
    We only consider the arrangement $\Arr = \mathcal{R}_n$ of type
    $\Arr_3(2n, 1)$ for $n \ge 5$. The argument for $\Arr_3(4m+1, 1)$ is
    analogous. Let $Q$ be a symmetric $3\times 3$-matrix.

    The codimension-$2$ flats of type $M_eE$ are precisely $L_j \defeq
    z_j^\perp \cap z_j^{\prime\perp}$ for $j = 0, \dots, n-1$.  Using
    Example~\ref{ex:small_tRs}, we deduce that $\pfaff \sR_{L_j} = \inner{z_j,
    z'_j}_Q = 0$ and thus
    \begin{align*}
      0 \ = \ -\Re(\zeta^j)\Im(\zeta^j) (Q_{11}-Q_{22})
      + (\Re(\zeta^j)^2-\Im(\zeta^j)^2) Q_{12} + \Re(\zeta^j) Q_{13} + \Im(\zeta^j) Q_{23}\,.
    \end{align*}
    Using $\Re(\zeta^j) = \frac{1}{2}(\zeta^j + \zeta^{-j})$ and
    $\Im(\zeta^j) = \frac{i}{2}(\zeta^j - \zeta^{-j})$ we get a linear
    system of equations for $j = 0, 1, 2, 3$:
    \[
        \underbrace{
        \begin{pmatrix}
            0 & 1 & 1 & 0\\
            \frac{1}{4} i (\zeta^{2} - \zeta^{-2}) & \frac{1}{2} (\zeta^{2} + \zeta^{-2}) & \frac{1}{2} (\zeta     + \zeta^{-1}) & -\frac{1}{2} i (\zeta    + \zeta^{-1})  \\
            \frac{1}{4} i (\zeta^{4} - \zeta^{-4}) & \frac{1}{2} (\zeta^{4} + \zeta^{-4}) & \frac{1}{2} (\zeta^{2} + \zeta^{-2}) & -\frac{1}{2} i (\zeta^{2} + \zeta^{-2}) \\
            \frac{1}{4} i (\zeta^{6} - \zeta^{-6}) & \frac{1}{2} (\zeta^{6} + \zeta^{-6}) & \frac{1}{2} (\zeta^{3} + \zeta^{-3}) & -\frac{1}{2} i (\zeta^{3} + \zeta^{-3})
        \end{pmatrix}
      }_{\eqdef X}
      \begin{pmatrix}
          Q_{11}-Q_{22}\\Q_{12}\\Q_{13}\\Q_{23}
      \end{pmatrix}
      \ = \ 0\,.
    \]
    The determinant of $X$ is
    \[
        \det(X) \ = \ -\frac{1}{8\zeta^7} (\zeta-1)^2 (\zeta^2 - 1)
        (\zeta^3-1)^2 (\zeta^4-1)\,.
    \]
    Thus, for $n \geq 5$, $\det X \neq 0$ and we conclude that $Q_{12} =
    Q_{13} = Q_{23} = 0$ and $Q_{11} = Q_{22}$. 

    Conversely, to show that $\Arr$ is virtually $Q$-inscribable for 
    $Q = \diag(a, a, t)$, $a, t \neq 0$, we
    show for each codimension-$2$ flat $L$ that $\Arr_L$ is virtually
    inscribable. For $L$ of type $M_eE$ we have
    $\pfaff \sR_L = 0$ by construction.
    If $L$ is of type $M_vM_eM_v\dots M_vM_e$, then $\Arr_L$ is the
    arrangement of a regular polygon and is therefore (virtually) inscribable.
    For the types $M_vEE$ and $M_eEE$, the localization $\Arr_L$ has an odd
    number of hyperplanes, thus $\det \sR_L = 0$. 
\end{proof}

Let us denote by $I_2(n)$ the symmetry group of the regular polygon
$\regularpolygon_n$ acting on $\R^3 = \R^2 \oplus \R$ by fixing the
last coordinate. By construction, $\mathcal{R}_n$ and $\mathcal{R}'_{n/2}$
are $I_2(n)$-symmetric.

\begin{prop}\label{prop:I2_symmetric}
    If $\Arr_3(2n,1)$ or $\Arr_3(4m+1,1)$ (for $n = 2m$) is
    inscribable, then there exists an inscribable realization that is
    $I_2(n)$-symmetric.
\end{prop}
\begin{proof}
    We again only consider the case when $\Arr$ is an inscribable
    arrangement of type $\Arr_3(2n,1)$. Since $\mathcal{R}_n$ is
    projectively unique by Proposition~\ref{prop:proj_unique}, there
    is some $g \in \GL(\R^3)$ with $\Arr = g\mathcal{R}_n$. Thus
    $\mathcal{R}_n$ is $Q$-inscribable with respect to $Q =
    g^tg$. From Lemma~\ref{lem:inf_family_Q} we conclude that up to an
    orthogonal transformation, we can assume that
    $g = \diag(\sqrt{a},\sqrt{a},\sqrt{t})$ for some $a,t > 0$. Since
    $g$ commutes with the action of $I_2(n)$ on $\R^3$, $\Arr$ is
    $I_2(n)$-symmetric up to orthogonal transformation.
\end{proof}

\begin{proof}[Proof of Proposition~\ref{prop:infinite_families}]
    All of the arrangements $\Arr_3(6, 1)$, $\Arr_3(8, 1)$, $\Arr_3(9, 1)$ and
    $\Arr_3(13, 1)$ are restrictions of reflection arrangements, see
    Figure~\ref{fig:restrictions}, and therefore strongly inscribable by
    Theorem~\ref{thm:restr_insc}. Therefore, we assume that $\Arr$ is a
    strongly inscribable arrangement of type $\Arr_3(2n, 1)$ for $n \geq 5$ or
    of type $\Arr_3(4m+1,1)$ for $m \geq 4$.

    By Proposition~\ref{prop:I2_symmetric} we can assume that $\Arr$ is
    symmetric with respect to $I_2(n)$. Let $Z$ be an inscribed zonotope for
    $\Arr$. Using Corollary~2.9 of~\cite{InFan1} we can assume that $Z$ is
    also symmetric with respect to $I_2(n)$. 

    The type of a hyperplanes of $\Arr$ is invariant under $I_2(n)$
    and hence the edges of $Z$ of the same type $M_v$, $M_e$, $E$, or
    $\infty$ have the same length. As inscribed polygons are uniquely
    determined by their edge lengths, we see that all facets
    corresponding to codimension-$2$ flats of the same type are
    congruent. Hexagons of type $M_vEE$ have two different angles
    $\alpha_1$, $\beta_1$ (between the edges of type $M_v$ and $E$ and
    between edges of type $E$ and $E$), likewise for $M_eEE$. For
    $\Arr_3(2n,1)$ and $n \geq 5$ or $\Arr_3(4m+1,1)$ and $m \geq 4$,
    $Z$ contains the configuration of four hexagons depicted in
    Figure~\ref{fig:four_hexagons}.
    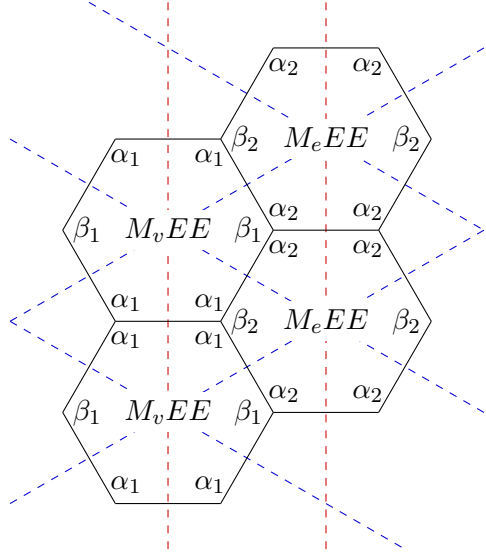
\begin{figure}
    \begin{center}
        \begin{tikzpicture}[yscale=1.732, scale=0.7]
            \draw[dashed, darkred] (-1, 6.5) to (-1, 0.5);
            \draw[dashed, darkred] ( 2, 6.5) to ( 2, 0.5);
            \draw[dashed, darkblue] (-4, 3) to (3.5, 0.5);
            \draw[dashed, darkblue] (-4, 5) to (5, 2);
            \draw[dashed, darkblue] (-2.5, 6.5) to (5, 4);
            \draw[dashed, darkblue] (-4, 1) to (5, 4);
            \draw[dashed, darkblue] (-4, 3) to (5, 6);
                        
            \node[anchor=300] at ( 0, 1) {$\alpha_1$};
            \node[anchor=240] at (-2, 1) {$\alpha_1$};
            \node[anchor=180] at (-3, 2) {$\beta_1$};
            \node[anchor=120] at (-2, 3) {$\alpha_1$};
            \node[anchor= 60] at ( 0, 3) {$\alpha_1$};
            \node[anchor=  0] at ( 1, 2) {$\beta_1$};

            \node[anchor=300] at ( 3, 2) {$\alpha_2$};
            \node[anchor=240] at ( 1, 2) {$\alpha_2$};
            \node[anchor=180] at ( 0, 3) {$\beta_2$};
            \node[anchor=120] at ( 1, 4) {$\alpha_2$};
            \node[anchor= 60] at ( 3, 4) {$\alpha_2$};
            \node[anchor=  0] at ( 4, 3) {$\beta_2$};
            
            \node[anchor=300] at ( 0, 3) {$\alpha_1$};
            \node[anchor=240] at (-2, 3) {$\alpha_1$};
            \node[anchor=180] at (-3, 4) {$\beta_1$};
            \node[anchor=120] at (-2, 5) {$\alpha_1$};
            \node[anchor= 60] at ( 0, 5) {$\alpha_1$};
            \node[anchor=  0] at ( 1, 4) {$\beta_1$};

            \node[anchor=300] at ( 3, 4) {$\alpha_2$};
            \node[anchor=240] at ( 1, 4) {$\alpha_2$};
            \node[anchor=180] at ( 0, 5) {$\beta_2$};
            \node[anchor=120] at ( 1, 6) {$\alpha_2$};
            \node[anchor= 60] at ( 3, 6) {$\alpha_2$};
            \node[anchor=  0] at ( 4, 5) {$\beta_2$};

            \node[preaction={fill, white}, inner sep=3pt] at (-1, 2) {$M_vEE$};
            \node[preaction={fill, white}, inner sep=3pt] at ( 2, 3) {$M_eEE$};
            \node[preaction={fill, white}, inner sep=3pt] at (-1, 4) {$M_vEE$};
            \node[preaction={fill, white}, inner sep=3pt] at ( 2, 5) {$M_eEE$};
            
            \draw ( 0, 1) to ( 1, 2);
            \draw ( 1, 2) to ( 3, 2);
            \draw ( 1, 2) to ( 0, 3);
            \draw ( 0, 3) to ( 1, 4);
            \draw ( 1, 4) to ( 3, 4);
            \draw ( 4, 3) to ( 3, 4);
            \draw ( 3, 2) to ( 4, 3);
            \draw ( 1, 4) to ( 0, 5);
            \draw ( 0, 5) to ( 1, 6);
            \draw ( 1, 6) to ( 3, 6);
            \draw ( 4, 5) to ( 3, 6);
            \draw ( 3, 4) to ( 4, 5);
            \draw (-2, 1) to ( 0, 1);
            \draw (-2, 1) to (-3, 2);
            \draw (-3, 2) to (-2, 3);
            \draw (-2, 3) to ( 0, 3);
            \draw (-2, 3) to (-3, 4);
            \draw (-3, 4) to (-2, 5);
            \draw (-2, 5) to ( 0, 5);

        \end{tikzpicture}
        \caption{An impossible facet configuration in the boundary of
          an inscribed zonotope.\label{fig:four_hexagons}}
    \end{center}
    \end{figure}
    Considering the interior angle sum of each of the hexagons gives
    $4\alpha_1 + 2\beta_1 = 4\pi$ and $4\alpha_2 + 2\beta_2 = 4\pi$,
    but considering the angles at the two vertices in the center gives
    $\beta_1 + 2 \alpha_2 < 2\pi$ and $\beta_2 + 2 \alpha_1 < 2\pi$, a
    contradiction.
\end{proof}
\newcommand{\roots}{\Phi}%

\begin{rem}
    It remains an open question which of the arrangements $\Arr$ in the two
    infinite families are (non-strongly and non-virtually) $Q$-inscribable.
    Based on experiments, we conjecture, that $\Arr$ can always be $Q$-inscribed
    for $Q = \diag(a, a, t)$ with $0 <
    t \ll a$.
\end{rem}

\subsection{Algebraic computations and sporadic simplicial  arrangements}
\label{sec:sporadic}
For an arrangement $\Arr = \{z_i^\perp : i=1,\dots,n\}$, 
$\lambda = (\lambda_1,\dots,\lambda_n) \in (\Ri)^n$, and $Q \in \R^{d \times d}$,
it is a simple matter of linear algebra to check if $Z_\lambda = \sum_i
\lambda_i [-z_i,z_i]$ is $Q$-inscribed: We may use
Lemma~\ref{lem:edge_midpoint_orthogonal} to conclude that $Z_\lambda$ is
$Q$-inscribed if and only if 
\[
    0 \ = \ \inner{z_{j(\tau)}, c(e_\tau)}_Q \ = \ \sum_{i=1}^n \tau_i
    \lambda_i \inner{z_{j(\tau)},z_i}_Q
\]
for all edges $\tau \in E(\TG)$. This gives a linear subspace $\{ Q  :
Z_\lambda \text{ is $Q$-inscribed} \} \subset \R^{d \times d}$. Now to test if
$Z_\lambda$ is inscribed into an ellipsoid, it suffices to test the linear
subspace contains a positive definite matrix, which can be done by
\emph{semidefinite programming}.

Conversely, for a fixed $Q$, Theorem~\ref{thm:compute_by_2flats} may be
adapted to test if there is $\lambda \in \R^n$ such that $Z_\lambda$ is
(virtually) $Q$-inscribed. Let $L = (z_{i_1},\dots,z_{i_k})$ be an ordered
codimension-$2$ flat of $\Arr$ and write $\sR_L(Q) \in \R^{k \times k}$ for
the skew-Gram matrix with respect to $Q$:
\begin{equation}\label{eqn:QskewGram}
    \sR_L(Q)_{ij} = -\sR_L(Q)_{ji} = \inner{z_i,z_j}_Q
\end{equation}
for $1 \le i < j \le k$ and $\sR_L(Q)_{ii} = 0$. Then $\Arr$ is strongly
$Q$-inscribable if and only if there is $\lambda \in \R^n_{>0}$ with
$\sR_L(Q)\lambda_L = 0$ for all ordered codimension-$2$ flats $L$ of $\Arr$.

\begin{rem} \label{rem:algorithm}
    Note that it is quite simple to enumerate all ordered codimension-$2$
    flats: Given $z_1,\dots,z_n \in \R^d$, we can assume that there are $c,w
    \in \R^d$ such that $\inner{c,z_i} = 1$ for $i=1,\dots,n$ and
    $\inner{w,z_i} \neq \inner{w,z_j}$ for $i\neq j$. The codimension-$2$ flat
    of $\Arr$ correspond to $2$-dimensional subspaces spanned by subsets of
    $z_1,\dots,z_n$. The total order induced by $w$ induces a cyclic order on
    each $2$-dimensional subspace. For given $Q$, this give a simple algorithm
    to set up the linear programming feasibility problem of finding $\lambda
    \in \R^n$ with $\lambda > 0$ and $\sR_L(Q) \lambda_L = 0$ for all ordered
    codimension-$2$ flats $L$.
\end{rem}

\newcommand\ideal[1]{\langle {#1} \rangle}%
\newcommand\Jideal{\mathcal{J}_{\Arr}}%
Let $S \defeq \C[Q_{ij}, \lambda_k : 1 \leq i \leq j \leq n, 1 \leq k \leq n]$. The
ideal
\[
    \Jideal \ \defeq \ \ideal{
        \sR_L(Q) \lambda_L  :
    L \in \flats_{d-2}(\Arr)} : (\lambda_1 \cdot \lambda_2 \cdot \ldots \cdot
    \lambda_n \cdot \det Q)^\infty \ \subseteq \ S
\]
contains the Zariski closure of the collection of pairs $(\lambda,Q) \in
(\Ri)^n \times \GL(\R^d)$ such that $Z_\lambda$ is $Q$-inscribed.
The saturation with respect to $\lambda_1 \cdot \lambda_2 \cdot
\ldots \cdot \lambda_n \cdot \det Q$ ensures that the closure takes
place over $(\Ri)^n \times \GL(\R^d)$. Standard facts from
computational algebra now give the following simple criterion for
non-inscribability.

\begin{prop}\label{prop:strongly_groebner_criterion}
    Let $\Arr$ be a simplicial arrangement. If $\Jideal = \ideal{1}$, then
    there is no non-singular $Q$ such that $\Arr$ is strongly virtually
    $Q$-inscribable.
\end{prop}

We applied this criterion to the 95 sporadic simplicial arrangements of the
Gr\"unbaum--Cuntz catalog. Except for the arrangements stemming from
restrictions of reflection arrangements only the arrangements $\Arr_3(10,1)$,
$\Arr_3(12,1)$, $\Arr_3(15,4)$, $\Arr_3(16,4)$, $\Arr_3(17,1)$, and
$\Arr_3(21,2)$ 
were not ruled out. In these $6$ cases the dimension of the ideal
$\Jideal$ is $2$. Since $\Jideal$ is homogeneous in $Q$ and $\lambda$
 it follows that there is a unique $\lambda$ and $Q$
up to scaling such that $Z_\lambda$ is $Q$-inscribed. In the above
cases it turns out that $\lambda$ and $Q$ are independent of each
other and determined by linear equations; see
Example~\ref{ex:comput_A10_1} below.  For $\Arr_3(10,1)$,
$\Arr_3(12,1)$, $\Arr_3(15,4)$, $\Arr_3(17,1)$
the matrix $Q$ is indefinite. For the two arrangements $\Arr_3(16,4)$,
$\Arr_3(21,2)$ the unique $\lambda$ had negative entries. Thus, while these
arrangements
posses a virtual inscribed zonotope, there exists no non-virtual inscribed
zonotope with these hyperplane arrangements as normal fan. 

Using the projective uniqueness of the arrangements in the Gr\"unbaum--Cuntz
catalog, the approach outlined above gives a computational proof of the
following result.

\begin{prop}\label{prop:3d_is_restriction}
    The only combinatorial types of strongly inscribable arrangements in the
    Gr\"unbaum--Cuntz catalog are restrictions of reflection arrangements.

    Moreover, assuming Conjecture~\ref{conj:GC}, an irreducible rank-$3$
    arrangement is
    strongly inscribable if and only if it is combinatorially isomorphic to a
    restriction of a reflection arrangement.
\end{prop}

In addition to restrictions of reflection arrangements, there are $2$
further known simplicial arrangements of rank $4$, denoted $\Arr_4(27, 1)$
and $\Arr_4(28, 1)$, that arise as subarrangements of $H_4$;
see~\cite{Geis_RP3}. For $\Arr_4(27, 1)$, the ideal $\Jideal$ is
trivial, for $\Arr_4(28,1)$ the ideal $\Jideal$ has dimension $2$ but
the matrix $Q$ is indefinite.  Alternatively, restrictions of both
arrangements are simplicial of rank $3$ and contained in the
Gr\"unbaum--Cuntz catalog, see
Figure~\ref{fig:restrictions_sqrt5}. Both arrangements have
non-strongly inscribable restrictions and Theorem~\ref{thm:restr_insc}
together with Proposition~\ref{prop:3d_is_restriction} refute strong
inscribability.
\begin{figure}
\begin{tikzpicture}[xscale=1.7, yscale=1.5]
    \definecolor{color0}{RGB}{255,255,255}
    \definecolor{color1}{RGB}{202,240,248}
    \definecolor{color2}{RGB}{173,232,244}
    \definecolor{color3}{RGB}{144,224,239}
    \definecolor{color4}{RGB}{ 72,202,228}
    \definecolor{color5}{RGB}{  0,180,216}
    \definecolor{color6}{RGB}{  0,150,199}
    \definecolor{color7}{RGB}{  0,119,182}
    \definecolor{color8}{RGB}{  2, 62,138}
    \definecolor{color9}{RGB}{  3,  4, 94}
    \definecolor{color10}{RGB}{  3,  4, 64}
    \definecolor{colorX}{RGB}{130,130,130}

    \tikzset{box/.style={draw, text width=2em, minimum height=1.5em, font=\tiny, align=center}}
    \tikzset{c1/.style={fill=color0}}
    \tikzset{c2/.style={fill=color1}}
    \tikzset{c3/.style={fill=color2}}
    \tikzset{c4/.style={fill=color3}}
    \tikzset{c5/.style={fill=color4}}
    \tikzset{c6/.style={fill=color5}}
    \tikzset{c7/.style={fill=color6}}
    \tikzset{c8/.style={fill=color7}}
    \tikzset{c9/.style={fill=color8, text=white}}
    \tikzset{c10/.style={fill=color9, text=white}}
    \tikzset{c12/.style={fill=color10, text=white}}
    \tikzset{cX/.style={fill=colorX}}

    \newcommand{\xbox}[5]{%
      \def\XargOne{#1}\def\XargTwo{#2}\def\Xnonsense{-1}
      \def\Xcolor{\ifx\XargOne\empty\else{\ifx\XargOne\Xnonsense cX\else c#1\fi}\fi}
      \node[box, \Xcolor] (#3) at (#4) { #5 };%
      \ifx\XargOne\empty\else
      \node[anchor = north east, inner sep = 2pt, \Xcolor, fill=none] (h#3) at (#3.north east) {\miniscule #2};
      \node[anchor = south east, inner sep = 2pt, \Xcolor, fill=none] (l#3) at (#3.south east) {\miniscule #1};
      \fi%
    }

    \node at (-2, 4.7){\bf Rank};
    \node at (-2, 4){\bf 4};
    \node at (-2, 3){\bf 3};
    \node at (-2, 2){\bf 2};
    \node at (-2, 1){\bf 1};
    
    \xbox{}{}{a1_1_1}{2, 1}{ $1, 1$ }
    \xbox{}{}{a2_6_1}{1, 2}{ $6, 1$ }
    \xbox{}{}{a2_5_1}{0, 2}{ $5, 1$ }
    \xbox{}{}{a2_7_1}{2, 2}{ $7, 1$ }
    \xbox{}{}{a2_10_1}{4, 2}{$10, 1$ }
    \xbox{}{}{a2_8_1}{3, 2}{$8, 1$ }
    \xbox{}{}{a2_4_1}{-1, 2}{$4, 1$ }

    \xbox{2}{2}{a3_10_1}{-1, 3}{ $10, 1$ }
    \xbox{2}{2}{a3_12_1}{0, 3}{ $12, 1$ }
    \xbox{-1}{-1}{a3_14_4}{1, 3}{ $14, 4$ }
    \xbox{2}{2}{a3_16_4}{3, 3}{ $16, 4$ }
    \xbox{2}{1}{a3_15_4}{2, 3}{ $15, 4$ }
    \xbox{2}{2}{a3_17_1}{4, 3}{ $17, 1$ }
    \xbox{2}{1}{a3_21_2}{5, 3}{ $21, 2$ }

    \xbox{-1}{1}{a4_27_1}{2, 4}{ $27,1$ }
    \xbox{2}{2}{a4_28_1}{3, 4}{ $28,1$ }
    
    \draw (a3_14_4.north) -- (a4_27_1.south);
    \draw (a3_16_4.north) -- (a4_27_1.south);
    \draw (a3_16_4.north) -- (a4_28_1.south);
    \draw (a3_15_4.north) -- (a4_28_1.south);
    \draw (a3_21_2.north) -- (a4_28_1.south);

    \draw (a2_6_1.north) -- (a3_14_4.south);
    \draw (a2_6_1.north) -- (a3_16_4.south);
    \draw (a2_6_1.north) -- (a3_15_4.south);
    \draw (a2_5_1.north) -- (a3_14_4.south);
    \draw (a2_7_1.north) -- (a3_14_4.south);
    \draw (a2_7_1.north) -- (a3_16_4.south);
    \draw (a2_7_1.north) -- (a3_15_4.south);
    \draw (a2_10_1.north) -- (a3_16_4.south);
    \draw (a2_10_1.north) -- (a3_21_2.south);
    \draw (a2_8_1.north) -- (a3_15_4.south);
    \draw (a2_8_1.north) -- (a3_21_2.south);
    \draw (a2_4_1.north) -- (a3_10_1.south);
    \draw (a2_5_1.north) -- (a3_10_1.south);
    \draw (a2_4_1.north) -- (a3_12_1.south);
    \draw (a2_5_1.north) -- (a3_12_1.south);
    \draw (a2_6_1.north) -- (a3_12_1.south);
    \draw (a2_6_1.north) -- (a3_17_1.south);
    \draw (a2_8_1.north) -- (a3_17_1.south);
    
    \draw (a1_1_1.north) -- (a2_6_1.south);
    \draw (a1_1_1.north) -- (a2_5_1.south);
    \draw (a1_1_1.north) -- (a2_7_1.south);
    \draw (a1_1_1.north) -- (a2_10_1.south);
    \draw (a1_1_1.north) -- (a2_8_1.south);

\end{tikzpicture}
\caption{Hasse diagram of the restrictions of $\Arr_4(27,1)$ and the
  non-inscribable arrangements for which $\Jideal \neq \ideal{1}$. For
  the labeling, we refer to Figure~\ref{fig:restrictions}.}
\label{fig:restrictions_sqrt5}
\end{figure}
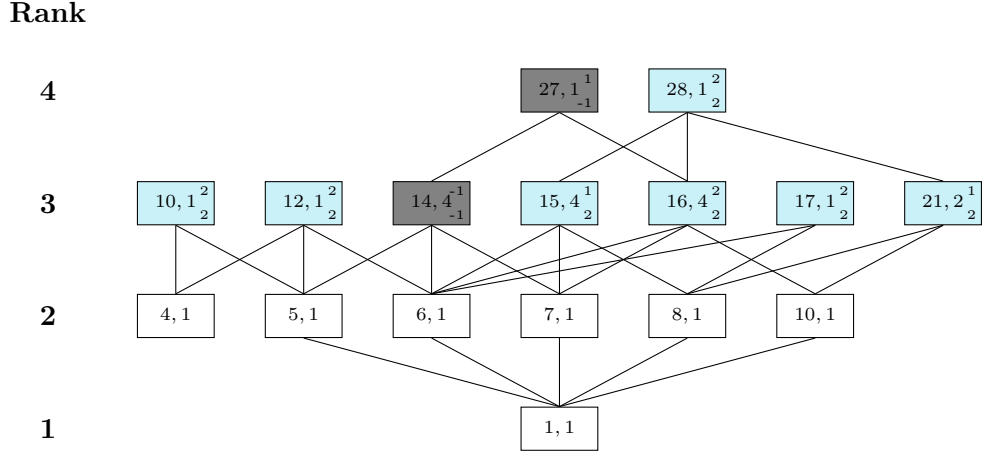

\begin{cor}
    No arrangement combinatorially isomorphic to $\Arr_4(27,1)$ or
    $\Arr_4(28,1)$ is strongly inscribable.
\end{cor}

We implemented this algorithm in \textsc{Sage}~\cite{SAGE} and our code is
part of the arXiv submission. Normal vectors (roots) to all known simplicial
arrangements of rank $3$ are found in~\cite{CuntzEliaLabbe} and as a
\textsc{Sage} database~\cite{CuntzEliaLabbe-SAGE} by the same authors.  Roots
for the restrictions can be found in~\cite{CuntzHeckenberger},
while~\cite{Geis_RP3} contains roots for the rank-$4$ examples $\Arr_4(27, 1)$
and $\Arr_4(28, 1)$.  Example~\ref{ex:comput_A10_1} shows the computation for
$\Arr_3(10,1)$.

\begin{example}\label{ex:comput_A10_1}
    The arrangement $\Arr_3(10,1)$ is an arrangement in $\R^3$ with $10$
    hyperplanes given by normals
    \[
        (z_1,\dots,z_{10}) \ = \ 
        \begin{pmatrix}
            2\tau + 1 &  2\tau + 2 &  1 &  \tau + 1 &  2\tau &  \tau + 1 &  1 &  0 &  1 &  \tau + 1\\
2\tau &  2\tau + 1 &  1 &  \tau + 1 &  2\tau &  \tau + 1 &  1 &  1 &  0 &  \tau\\
\tau &  \tau + 1 &  1 &  \tau &  \tau &  1 &  0 &  0 &  0 &  \tau\\
        \end{pmatrix}
    \]
    where $\tau = \frac{\sqrt{5} + 1}{2}$ is the golden ratio. The
    $16$ ordered codimension-$2$ flats are

    \[
        \begin{aligned}
            &(z_8,  z_3), (z_9,  z_6), (z_{10}, z_5), (z_4,  z_1), (z_7,  z_2),
             (z_5,  z_1, z_9), (z_8,  z_6, z_1), (z_7,  z_1, z_{10}), (z_9,  z_{10},z_3),\\
            &(z_8,  z_5, z_2), (z_4,  z_2, z_9), (z_8,  z_7, z_9), (z_8,  z_4, z_{10}),
             (z_6,  z_2, z_{10}), (z_3,  z_2, z_1), (z_7,  z_6, z_5, z_4, z_3)
        \end{aligned}
    \]
    The ideal $\Jideal$ is generated by the polynomials
    \[
        \begin{aligned}
            &\lambda_9 - \lambda_{10}, \, 
    \lambda_8 - \lambda_{10}, \, 
    \lambda_7 + (\tau + 1) \lambda_{10}, \, 
    \lambda_6 + (\tau + 1) \lambda_{10}, \, 
    \lambda_5 + (\tau + 1) \lambda_{10},\\
            &\lambda_4 + (\tau + 1) \lambda_{10}, \, 
    \lambda_3 + (\tau + 1) \lambda_{10}, \, 
    \lambda_2 - \lambda_{10}, \, 
    \lambda_1 - \lambda_{10},\\
           &10 Q_{23} - (\tau - 3)  Q_{33}, \,
            10 Q_{13} + (\tau + 2)   Q_{33}, \,
            10 Q_{22} + (2 \tau - 1) Q_{33},\\
           &10 Q_{12} + (-\tau - 2)  Q_{33}, \,
            10 Q_{11} + (2 \tau - 1) Q_{33}
        \end{aligned}
    \]

    Setting $\lambda_{10} = s$ and $Q_{33} = t$, this yields
    \[
        \lambda = s( 1, 1, - (\tau + 1), - (\tau + 1), - (\tau + 1), - (\tau + 1), - (\tau + 1), 1, 1, 1) \\
    \]
    and
    \[
        Q \ = \ 
        \frac{t}{10}
        \begin{pmatrix}
            -2 \tau + 1  & \tau + 2   & -\tau -2 \\
            \tau + 2   & -2 \tau + 1  & \tau - 3  \\
            -\tau -2   & \tau - 3   & 10 \\
        \end{pmatrix} \, .
    \]
    It is straightforward to verify that $Q$ is never positive-definite and $\lambda \not\in \Rpp^{10}$.
\end{example}

As an implementation detail, we compute $\Jideal$ by adding the
equations $\sR_L \lambda_L = 0$ step by step, where we order the flats
$L$ by the number of hyperplanes in which they are contained, and
reduce to a Gr\"obner basis. Thus, we start with flats contained in
precisely two hyperplanes, which are especially restrictive, as they
do not involve any of the $\lambda_i$'s (compare
Example~\ref{ex:small_tRs}). It helps tremendously to saturate with
$\lambda_1 \cdot \lambda_2 \cdot \ldots \cdot \lambda_n$ after each
step to cut down the computational cost.  This allowed us to inspect
the ideals not only for all the sporadic examples in dimension $3$ but
for all restrictions of sporadic reflection arrangements of all ranks up to
combinatorial symmetry.  The computation of a Gr\"obner basis for
$\Jideal$ for the largest arrangement $E_8$, took about $20$ minutes on a
contemporary laptop. With a modest level of parallelization,
the remaining examples can be checked. The computation finishes long before
that for $E_8$. This allowed us to compute the dimension
of $\Jideal$ of all sporadic reflection arrangements and their
restrictions, again up to combinatorial isomorphism. The dimensions
are collected in Figure~\ref{fig:restrictions}.

The same ideas yield a criterion for the non-existence of a $Q$-inscribed belt
polytope for a given arrangement $\Arr$.  We appeal to
Corollary~\ref{cor:pfaffian}: $\Arr$ has a (virtual) belt polytope $P$ that is
$Q$-inscribed if and only if $R_L(Q)$ is singular for all ordered
codimension-$2$ flats $L \in \flats_{d-2}(\Arr)$. This gives rise to the ideal
\[
    \Jideal' \ \defeq \ \langle \pfaff R_L(Q) : L \in \flats_{d-2}(\Arr)
    \rangle : (\det Q)^\infty \ \subseteq \ \C[Q_{ij} : 1 \le i \le j \le d] \, .
\]
Similarly to Proposition~\ref{prop:strongly_groebner_criterion}, we get
\begin{prop}\label{prop:groebner_criterion}
    Let $\Arr$ be a simplicial arrangement. If $\Jideal' = \ideal{1}$,
    then there is no non-singular $Q$ such that $\Arr$ is virtually
    $Q$-inscribable.
\end{prop}

Note that the existence of $\lambda \in \R^n$ with $R_L(Q)\lambda_L = 0$ for
all codimension-$2$ flats $L$ is a much stronger condition than $\pfaff R_L(Q)
= 0$. Moreover, the dimension of $\Jideal'$ is hard to interpret.

\begin{rem}
    We computed $\Jideal'$ for the known simplicial arrangements which
    are not part of an infinite family. The computations show that in
    $77$ examples we have $\Jideal' \neq \ideal{1}$. In all of these
    cases, $\Jideal'$ describes a linear subspace. With the help of
    semidefinite programming we could certify that in all but two
    cases, $\Arr_3(14, 3)$ and $\Arr_3(15, 4)$, there is a positive
    definite matrix $Q$, such that $\Arr$ is virtually
    $Q$-inscribable.  We did not check, whether one can find a $Q$
    such that $\Arr$ is (non-virtually) $Q$-inscribable. The
    computations show that there are inscribable arrangements which
    are not strongly inscribable. Moreover, there exist inscribable
    arrangements $\Arr$ such that $\Arr^L$ is not inscribable,
    contrary to Theorem~\ref{thm:restr_insc}.
\end{rem}

\begin{rem}
    If $\Arr$ is a strongly inscribed arrangement and $H \in \Arr$,
    then Theorem~\ref{thm:restr_insc} implies that there is a linear
    projection $\ZInCone(\Arr) \to \ZInCone(\Arr^H)$. This projection is
    in general not onto, as the following example shows: Let
    $\Arr = \Arr_4(22, 1)$ and $H$ in $\Arr$ such that
    $\Arr^H \cong \Arr_3(13,2)$ (compare
    Figure~\ref{fig:restrictions}). 
    Our computations show that $\dim \ZInCone(\Arr) = 1$, but
    $\dim \ZInCone(\Arr^H) = 2$.
\end{rem}

\subsection{Restrictions of arrangements of type $A,B,D$}\label{sec:ABD}

The algorithm of the last section also yields the possible bilinear forms $Q$
for which the known sporadic simplicial arrangements of ranks $\le 8$ are
strongly inscribable.  This included the reflection arrangements
$F_4$, $E_6$, $E_7$, $E_8$, $H_3$ and $H_4$ as well as their restrictions. In
this section, we explicitly determine the bilinear forms $Q$ for which the
reflection arrangements of type $A_n, B_n, D_n$ as well as their restrictions
are inscribable. Restrictions of reflection arrangements of type $A$ are
combinatorially, and hence projectively, of type $A$. For the arrangements of
type $B$ and $D$, the restrictions are well understood and captured by a
family of simplicial arrangements $\D_{n,s}$ for $0 \le s \le n$; see
Figure~\ref{fig:restrictions}.

\newcommand\Ones{\mathbf{J}}%
\subsubsection*{Inscribed arrangements of type $A$}\label{sec:A_n}
The reflection arrangement of type $A_n$ is canonically realized in $\R^{n+1}$
by the $\binom{n+1}{2}$ hyperplanes 
\[
    H_{ij} \ = \ \{ x \in \R^{n+1} : x_i \ = \ x_j \}
\]
for $1 \le i < j \le n+1$. In \cite[Proposition~6.73]{OrlikTerao}, it is shown
that the restriction of $A_n$ to any hyperplane is combinatorially isomorphic
to $A_{n-1}$.  Since $A_3$ is projectively unique, it follows from
Lemma~\ref{lem:moduli_size} that every $A_n$ is projectively unique.  Hence it
suffices to determine the possible inner products for a fixed realization of
$A_n$ for each $n$. The above realization is not essential. Restricting to the
hyperplane $\{x : x_{n+1} = 0 \} \cong \R^n$ gives the realization
\[
    \Arr_n \defeq \{ f_{ij}^\perp : 1 \le i < j \le n+1 \} \, ,
\]
where we set $f_{ij} \defeq e_i - e_j$ for $1 \le i < j \le n$ and $f_{in+1}
\defeq e_i$. Let $\Ones_n \in \R^{n \times n}$ be the matrix of all ones.

\begin{thm}
    For $n \ge 3$, $\Arr_n$ is virtually $Q$-inscribable if and only if 
    \[
        Q \ = \ \diag(a_1, \dots, a_n) + a_{n+1} \Ones_n
    \]
    for some $a_1, \dots, a_n, a_{n+1} \in \R$. Moreover, if $Z_{\lambda} =
    \sum_{i<j} \lambda_{ij} [-f_{ij},f_{ij}]$ is a $Q$-inscribed zonotope,
    then $a_1,\dots,a_n,a_{n+1} > 0$ and $\lambda_{ij} = \frac{t}{a_ia_j}$ for
    $1 \le i < j \le n+1$ and a unique $t \in \Rpp$.
\end{thm}
\begin{proof}
    The arrangement $\Arr_n$ is virtually $Q$-inscribable if and only if
    $\pfaff R_L(Q) = 0$ for every ordered codimension-$2$ flat of $\Arr_n$.
    It follows from general theory (\cite{Humphreys}) or simply by inspection
    that every ordered codimension-$2$ flat is the intersection of $2$ or $3$
    hyperplanes and of the form
    \begin{compactenum}[\rm(1)]
    \item $L = (f_{ij},f_{kl})$ for $\{i,j\} \cap \{k,l\} = \emptyset$, or
    \item $L = (f_{ij},f_{ik},f_{jk})$ for $i < j < k$.
    \end{compactenum}

    For flats of type (2), the matrix $R_L(Q)$ has odd order and hence $\pfaff
    R_L(Q) \equiv 0$. For flats of type (1), we compute
    \[
        R_L(Q) 
        \ = \ 
        \begin{pmatrix}
            0 & f_{ij}^tQ f_{kl}\\
            -f_{ij}^tQ f_{kl} & 0 \\
        \end{pmatrix}
    \]
    For $j = n+1$, we thus get $0 = \pfaff R_L(Q) = Q_{ik} - Q_{il}$ and hence
    all off-diagonal entries of $Q$ are equal to some $a_{n+1} \in \R$.  This
    implies $Q - a_{n+1} \Ones_n = \diag(a_1,\dots,a_n)$. Conversely,
    $f_{ij}^tQ f_{kl} =0$ for all flats of type (1) whenever $Q$ is of this
    form.  This shows the first claim.

    Assume that $Z_{\lambda} = \sum_{i<j} \lambda_{ij} [-f_{ij},f_{ij}]$ is a
    $Q$-inscribed zonotope. In particular $\lambda_{ij} > 0$ for all $i < j$.
    Consider the ordered flat $L = (f_{ij},f_{ik},f_{jk})$ with $k=n+1$. Then
    \[
        R_L(Q) 
        \ = \ 
        \begin{pmatrix}
            0 & a_i & -a_j \\
            -a_i & 0 & a_{n+1} \\
             a_j & -a_{n+1} & 0 \\
        \end{pmatrix}
    \]
    and $R_L(Q)\lambda_L = 0$ if and only if $\lambda_L =
    (\lambda_{ij},\lambda_{in+1},\lambda_{jn+1}) =  s_L (a_{n+1},a_j,a_i)$ for
    some $s_L \in \R$. Since $\lambda > 0$, we conclude that all $a_i$ have
    the same sign. By inspecting all
    flats of this type, we infer that there is a $t \in \R$ with
    $\lambda_{ij} = \frac{t}{a_ia_j}$ for all $1 \le i < j \le n+1$. Since
    $V(Z_\lambda) \subseteq E_Q = \{ x : \inner{x,x}_Q = 1 \}$, the scaling
    parameter $t$ is positive and unique.
\end{proof}

Based on these computations, we conclude:
\begin{cor}
    Let $\Arr = \Arr_n$, $n \geq 3$. Then $\dim \Jideal = n + 2$ and
    $\dim \Jideal' = n+1$.
\end{cor}

\subsubsection*{Inscribed arrangements from types $B$ and $D$}
For $0 \le s \le n$, define the arrangement in $\R^n$
\[
    \D_{n,s} \defeq \{e_1^\perp, \dots, e_s^\perp\} \cup \{(e_i \pm e_j)^\perp
    : 1 \leq i < j \leq n\}\,.
\]
For $s = 0$, $\D_{n,0}$ is the reflection arrangement of type $D_n$. For $s =
n$, $\D_{n,n}$ is the reflection arrangement of type $B_n$. In particular
$\D_{n,s}$ is an $s$-fold restriction of the arrangement of type $D_{n+s}$.
It is shown in Orlik-Terao~\cite[Chapter~6.4]{OrlikTerao} that the
restrictions of $\D_{n,0}$ are combinatorially isomorphic to $\D_{n-1,1}$, while
$B_n$ uniquely restricts to $B_{n-1}$. Moreover, for $0 < s < n$, the
restrictions of $\D_{n, s}$ is combinatorially of type $B_{n-1}$,
$\D_{n-1,s-1}$ (if $s > 1$), $\D_{n-1,s}$, or $\D_{n-1,s+1}$ (if $s < n-1$).
Figure~\ref{fig:restrictions} shows the Hasse diagram of the restrictions of
reflection arrangements. Since the arrangements $\D_{3,1}$, $\D_{3,2}$ and
$\D_{3,3}$ are projectively unique, we conclude by Lemma~\ref{lem:moduli_size}
that the same is true for all $\D_{n,s}$ with $n \ge 3$. We represent a
zonotope for $\D_{n,s}$ by
\[
    Z_\lambda \ = \ 
    \sum_{1\le i<j\le n} \lambda^-_{ij} [e_i-e_j,e_j-e_i] + 
    \sum_{1\le i<j\le n} \lambda^+_{ij} [-e_j-e_i,e_i+e_j] + 
    \sum_{k=1}^s \lambda_k [-e_k,e_k] 
\]
for $\lambda^+_{ij},\lambda^-_{ij},\lambda_k \in \R$ for $1 \le i < j \le n$
and $1 \le k \le s$.

\begin{thm}\label{thm:Q_Dns}
    Let $n \ge 4$ and $0 \le s \le n$ and $Q \in \R^{n \times n}$
    non-singular. Then $\D_{n,s}$ is virtually $Q$-inscribable if and
    only if there are $a_1, \dots, a_s, a \in \Ri$ with $a_i \neq a$ for
    ${i=1,\dots,s}$ such that $Q = \diag(a_1, \dots, a_s,a,\dots,a)$.
    Moreover, $Z_\lambda$ is $Q$-inscribed if and only if
    ${\lambda_{ij}^+ = \lambda_{ij}^- = \frac{t}{a_ia_j}}$ for $i < j$
    and $\lambda_k = \tfrac{t(a - a_k)}{a_k^2a} + \tfrac{t'}{a_k}$ for $k=1,\dots,s$
    for some $t, t' \in \Ri$, where $t' = 0$ if $s < n$.
\end{thm}
\begin{proof}
    For $i<j$ and $k<l$ with $\{i,j\} \cap \{k,l\} = \emptyset$, $L = (e_i +
    \sigma e_j, e_k + \tau e_l)$ is an ordered codimension-$2$ flat for all
    $\sigma,\tau \in \PM$. If $\D_{n,s}$ is virtually $Q$-inscribable, then
    for all $\sigma,\tau \in \PM$
    \[
        0 \ = \ \pfaff \sR_L \ = \ (e_i + \sigma e_j)^t Q (e_k + \tau e_l) \,
        .
    \]
    This implies $Q_{ij} = 0$ for all $i \neq j$ and hence $Q =
    \diag(a_1,\dots,a_n)$ for some $a_1,\dots,a_n \in \R$. For $i < j$ with $i
    > s$, $L = (e_i-e_j,e_i+e_j)$ is an ordered codimension-$2$ flat and
    \[
        0 \ = \ \pfaff \sR_L \ = \ (e_i - e_j)^t Q (e_i + e_j)  \ = \ 
        Q_{ii} + Q_{ij} - Q_{ji} - Q_{jj} \ = \ a_i - a_j 
    \]
    shows that $a_{s+1} = \cdots = a_n = a$. 
    It is a routine calculation to check that $\pfaff R_L(Q) = 0$ for all
    codimension-$2$ flats if and only if $Q$ is of the stated form.

    For the second claim, observe that for the codimension-$2$ flat $L =
    (e_i-e_j, e_i - e_k ,e_j-e_k)$ the condition $R_L(Q)\lambda_L = 0$
    translates into
    \begin{align*}
        0 \ &= \  (e_i - e_j)^t Q \big( \lambda^-_{i k} (e_i - e_k) +
        \lambda^-_{j k} (e_j - e_k) \big)
      \ = \ \lambda^-_{ik} a_i - \lambda^-_{j k} a_j\,.\\
      0 \ &= \ (e_j - e_k)^t Q \big( \lambda^-_{i j} (e_i - e_j) + 
        \lambda^-_{i k} (e_i - e_k) \big)
      \ = \ \lambda^-_{i k} a_k - \lambda^-_{i j} a_j\,.
    \end{align*}
    Set $t \defeq a_1 a_2 \lambda^-_{1 2}$. The second equation
    now implies $\lambda^-_{1  k} = \frac{t}{a_1a_k}$ for all $1 < k \leq n$
    and the first in turn implies $\lambda^-_{j k} = \frac{t}{a_ja_k}$ for all
    $1 \leq j < k \leq n$.

    Likewise, the ordered codimension-$2$ flats $L = (e_i - e_j, e_i + e_k,
    e_j + e_k)$ give the condition $0 = \lambda^+_{i k} a_k - \lambda^-_{i j}
    a_j$ and thus $\lambda^+_{i k} = \frac{t}{a_ia_k}$ for all $1 \leq i < j <
    k \leq n$.  A similar calculation for $L = (e_i + e_j, e_i - e_k, e_j +
    e_k)$ shows that this also holds if $k = i + 1$.

    If $s < n$, then $L = (e_i + e_n, e_i, e_i - e_n)$ implies 
    $0   
    =  a_i\lambda_{i} - (a - a_i)\lambda^+_{i n}$.  Thus, $\lambda_{i} =
    \frac{t(a_i - a)}{a_i^2a}$, which is nonzero if and only if $a_i \neq a$. 

    Finally, if $s = n$, then the ordered codimension-$2$-flat $L = (e_j,  e_i
    + e_j, e_i, e_i - e_j)$ yields 
    $0 = a_i \lambda_{i} - a_j \lambda_{j} + (a_i - a_j) \lambda^+_{i j}$.
    Setting $t' \defeq a_1\lambda_{1} - \tfrac{t(a - a_1)}{a_1a}$ then gives for $i=1$
    \begin{align*}
      a_j \lambda_{j}
      \ &= \ a_1 \Big(\frac{t(a - a_1)}{a_1^2a} + \frac{t'}{a_1}\Big) + (a_1 - a_j)\frac{t}{a_1a_j}\\
      \ &= \ \frac{ta_ja-ta_1a_j}{a_1a_ja} + t' + \frac{ta_1a - ta_ja}{a_1a_ja} \ = \ \frac{t(a-a_j)}{a_ja} + t'\,,
    \end{align*}
    thus $\lambda_{j} = \tfrac{t(a-a_j)}{a_j^2a} + \tfrac{t'}{a_j}$, as required.  It is now easy
    (but tedious) to verify that $Z_\lambda$ is $Q$-inscribed for the stated
    choice of $\lambda$.
\end{proof}

\begin{rem}
    The proof does not cover the case $n = 3$, since no
    codimension-$2$ flats of the form $(e_i + \sigma e_j, e_k + \tau e_l)$
    exist. Instead, we determined $\Jideal$ and $\Jideal'$ with the computer
    as explained in the previous section.
\end{rem}

Again, we directly conclude:
\begin{cor}
    Let $\Arr = \D_{n,s}$, $n \geq 4$, $0 \leq s \leq n$. Then
    $\dim \Jideal = s + 2$ and $\dim \Jideal' = \min(s+1, n)$.
\end{cor}

\section{General quadrics and infinite arrangements}\label{sec:close}

We close with two brief remarks and some appealing pictures.

\subsection{General quadrics}
Our algebraic approach for checking $Q$-inscribability in
Section~\ref{sec:sporadic} showed that for the arrangements $\Arr_3(10,1)$,
$\Arr_3(12,1)$, $\Arr_3(15,4)$, $\Arr_3(16,4)$, $\Arr_3(17,1)$, and
$\Arr_3(21,2)$, the corresponding ideal $\Jideal$ had dimension $2$. Since
$\Jideal$ is homogeneous with respect to $Q$ as well as $\lambda$, this
implies that the bilinear form $Q$  and lengths $\lambda$ are unique if we
require that $\inner{v,v}_Q = 1$ for all vertices $v \in V(Z_\lambda)$.  In
all all but two cases, the bilinear form $Q$ was indefinite and the zonotope
$Z_\lambda$ virtual. The two arrangements $\Arr_3(16,4)$, $\Arr_3(21,2)$ only
have virtual zonotopes that are inscribed into an ellipsoid.
Figure~\ref{fig:hyperboloid} shows all six examples.

\begin{figure}[h]
    \captionsetup[subfigure]{labelformat=empty}
    \begin{center}
        \begin{subfigure}{0.15\textwidth}
            \includegraphics[width=\textwidth]{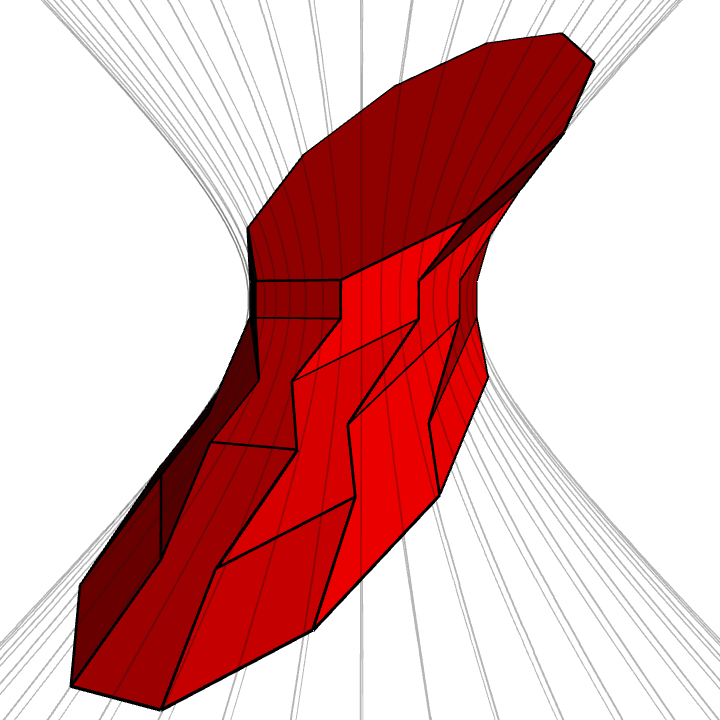}%
            \caption{$\Arr_3(10, 1)$}
        \end{subfigure}
        \begin{subfigure}{0.15\textwidth}
    \includegraphics[width=\textwidth]{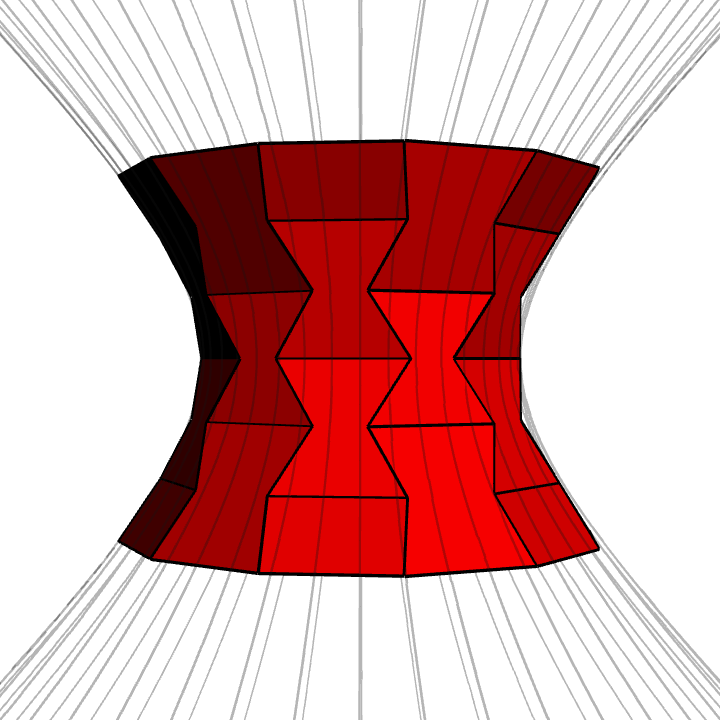}
            \caption{$\Arr_3(12, 1)$}
        \end{subfigure}
        \begin{subfigure}{0.15\textwidth}
    \includegraphics[width=\textwidth]{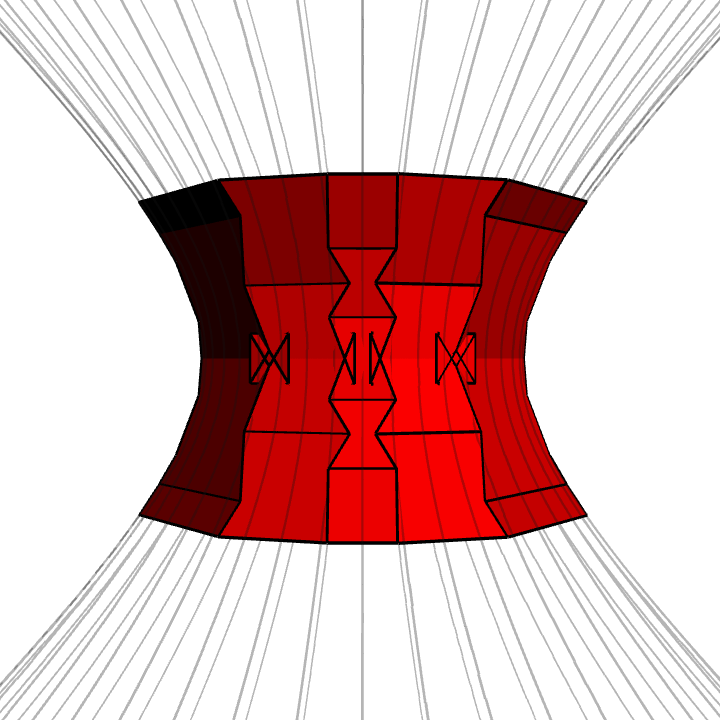}
            \caption{$\Arr_3(15, 4)$}
        \end{subfigure}
        \begin{subfigure}{0.15\textwidth}
    \includegraphics[width=\textwidth]{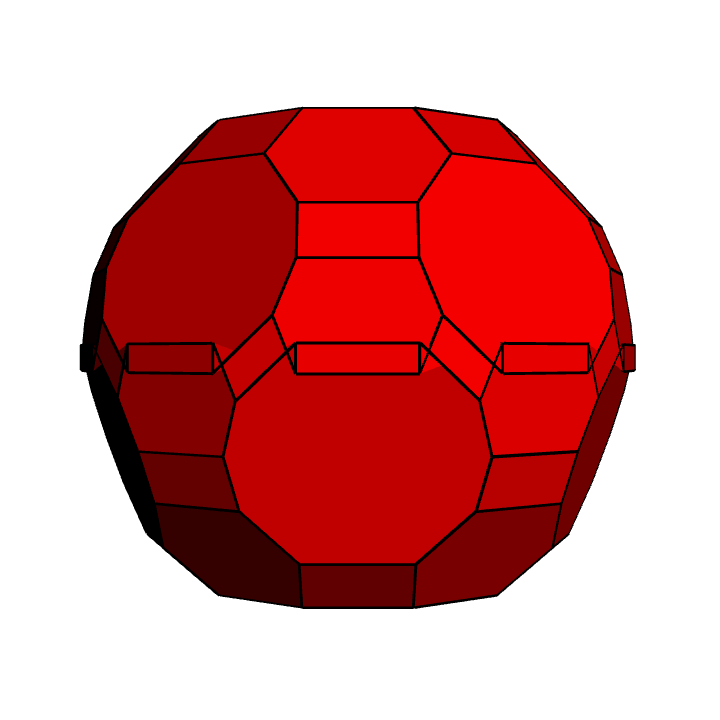}
            \caption{$\Arr_3(16, 4)$}
        \end{subfigure}
        \begin{subfigure}{0.15\textwidth}
    \includegraphics[width=\textwidth]{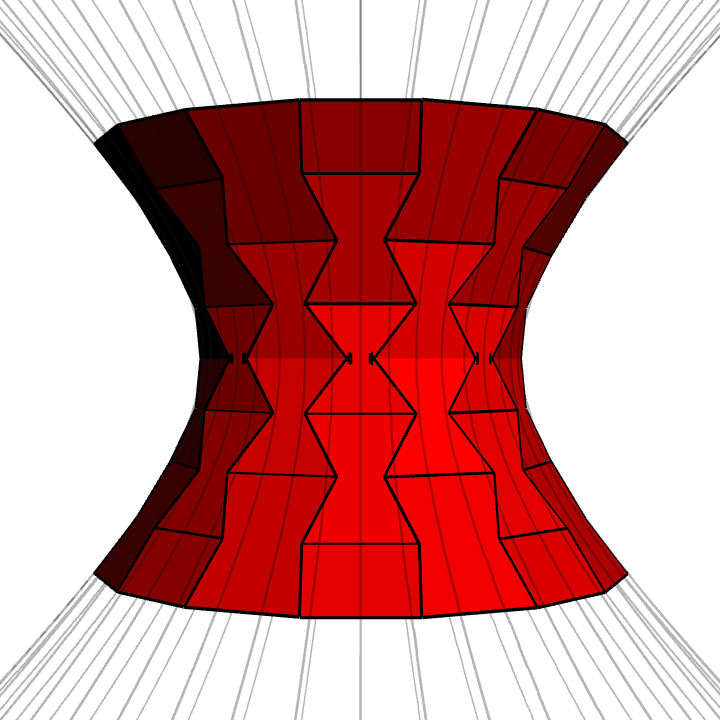}
            \caption{$\Arr_3(17, 1)$}
        \end{subfigure}
        \begin{subfigure}{0.15\textwidth}
    \includegraphics[width=\textwidth]{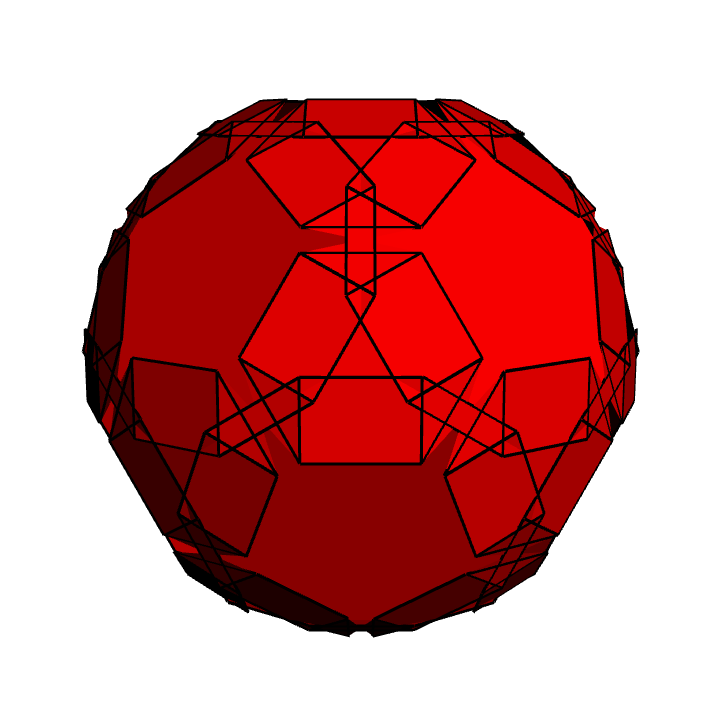}
            \caption{$\Arr_3(21, 2)$}
        \end{subfigure}%
        \caption{Virtual zonotopes for the six combinatorial types not
          excluded from Gr\"unbaums list by our program. They are
          inscribed into some quadric, not necessarily a sphere.
          }
          \label{fig:hyperboloid}
    \end{center}
\end{figure}

It is noteworthy that among all finite reflection arrangements, only the
arrangements of type $A_n$ and $B_n$ have inscribable realizations
into quadrics of different type. Figure~\ref{fig:hyperboloid} fuels the
question for a better geometric understanding of inscribable (virtual)
zonotopes. Also, we do not know if arrangements which are virtually
inscribable into other quadrics need to be simplicial.

\subsection{Infinite arrangements}

An \Def{affine reflection arrangement} $\Arr$ is the infinite collection of
affine hyperplanes 
\[
    H_{\alpha,k} \ = \ \{ x \in \R^d : \inner{\alpha,x} = k \} \, ,
\]
where $k \in \Z$ and $\alpha$ ranges over all roots in a crystallographic root
system $\Phi$. The associated affine Weyl group $W_a$ is the group of affine
transformations generated by the affine reflections $s_{\alpha,k}$ in
$H_{\alpha,k}$; see~\cite[Chapter 4]{Humphreys}. The localization of $\Arr$ at
any $0$-flat is the translate of a finite reflection arrangement. Hence the
orbit of a point $p$ under $W_a$ yields a tessellation of $\R^d$ into
inscribed belt polytopes.

For a suitable choice of $p$, the belt polytopes of the tessellation are
zonotopes.  Theorem~\ref{thm:restr_insc} now yields that the restriction of
$\Arr$ to a flat $L$ is a strongly inscribed infinite arrangement
and the tessellation into inscribed zonotopes is not induced by an affine
reflection arrangement. Figure~\ref{fig:zono_tiling} shows a $2$-dimensional
restriction of $\tilde{E}_8$. 
\begin{figure}[h]
    \begin{center}
         \begin{tikzpicture}
             \clip (0,0) rectangle (0.8\textwidth,0.5\textwidth);
             \begin{scope}[scale=1.1]
                 \input{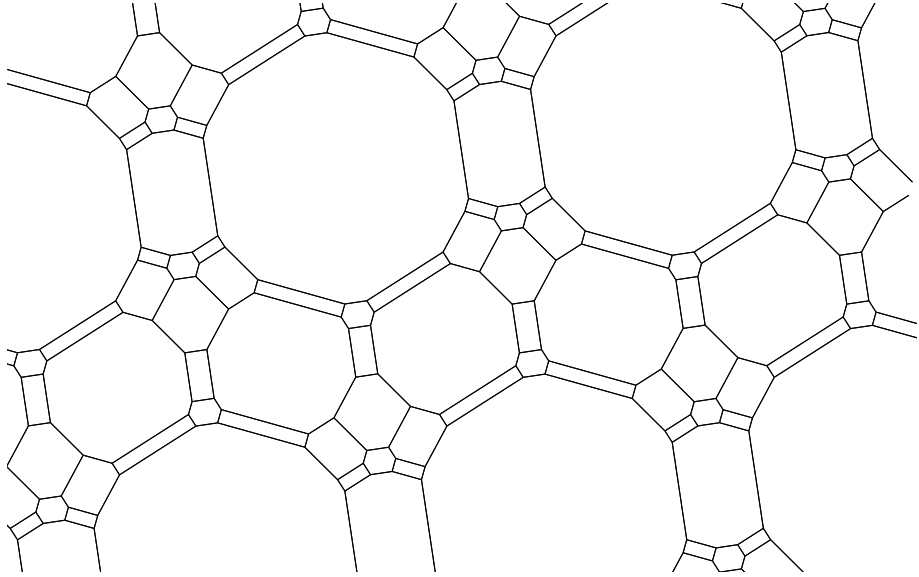}
             \end{scope}
        \end{tikzpicture}
        \caption{A tiling of $\R^2$ by inscribed zonogons.}
        \label{fig:zono_tiling}
    \end{center}
\end{figure}

It would be very interesting to investigate geometric and
combinatorial properties of tilings of space into inscribed zonotopes.

\newgeometry{bottom=2.5cm}
\begin{landscape}
    \begin{figure}
    \begin{center}
        \begin{tikzpicture}[xscale=1.1, yscale=1.5]
            \definecolor{color0}{RGB}{255,255,255}
            \definecolor{color1}{RGB}{220,243,248}
            \definecolor{color2}{RGB}{173,232,244}
            \definecolor{color3}{RGB}{144,224,239}
            \definecolor{color4}{RGB}{ 72,202,228}
            \definecolor{color5}{RGB}{  0,180,216}
            \definecolor{color6}{RGB}{  0,150,199}
            \definecolor{color7}{RGB}{  0,119,182}
            \definecolor{color8}{RGB}{  2, 62,138}
            \definecolor{color9}{RGB}{  3,  4, 94}
            \definecolor{color10}{RGB}{  3,  4, 64}

            \tikzset{box/.style={draw, text width=2em, minimum height=1.5em, font=\tiny, align=center}}
            \tikzset{c0/.style={fill=gray}}
            \tikzset{c1/.style={fill=color0}}
            \tikzset{c2/.style={fill=color1}}
            \tikzset{c3/.style={fill=color2}}
            \tikzset{c4/.style={fill=color3}}
            \tikzset{c5/.style={fill=color4}}
            \tikzset{c6/.style={fill=color5}}
            \tikzset{c7/.style={fill=color6}}
            \tikzset{c8/.style={fill=color7}}
            \tikzset{c9/.style={fill=color8, text=white}}
            \tikzset{c10/.style={fill=color9, text=white}}
            \tikzset{c12/.style={fill=color10, text=white}}

            \newcommand{\xbox}[5]{%
              \def\XargOne{#1}\def\XargTwo{#2}\def\Xnonsense{-1}
              \def\Xcolor{\ifx\XargOne\empty\else{\ifx\XargOne\Xnonsense cX\else c#1\fi}\fi}
              \node[box, \Xcolor] (#3) at (#4) { #5 };%
              \ifx\XargOne\empty\else
              \node[anchor = north east, inner sep = 2pt, \Xcolor, fill=none] (h#3) at (#3.north east) {\miniscule #2};
              \node[anchor = south east, inner sep = 2pt, \Xcolor, fill=none] (l#3) at (#3.south east) {\miniscule #1};
              \fi%
            }
            \node at (-1, 8.7){\bf Rank};
            \node at (-1, 8){\bf 8};
            \node at (-1, 7){\bf 7};
            \node at (-1, 6){\bf 6};
            \node at (-1, 5){\bf 5};
            \node at (-1, 4){\bf 4};
            \node at (-1, 3){\bf 3};
            \node at (-1, 2){\bf 2};
            \node at (-1, 1){\bf 1};

            \xbox{}{}{n43}{ 8, 1}{$A_1$}

            \xbox{}{}{n31}{ 4, 2}{ $A_2$ }
            \xbox{}{}{n30}{ 7, 2}{ $B_2$ }
            \xbox{}{}{n62}{ 8, 2}{ $5, 1$ }
            \xbox{}{}{n42}{11, 2}{ $6, 1$ }
            \xbox{}{}{n68}{14, 2}{ $8, 1$ }
            \xbox{}{}{n74}{15, 2}{ $10, 1$ }
            \xbox{}{}{n75}{16, 2}{ $12, 1$ }

            \xbox{5}{4}{n20}{ 0, 3}{ $A_3$ }
            \xbox{4}{3}{n29}{ 2, 3}{ $\D_{3,1}$ } 
            \xbox{4}{3}{n41}{ 3, 3}{ $\D_{3,2}$ } 
            \xbox{5}{3}{n19}{ 5, 3}{ $B_3$ }
            \xbox{3}{2}{n52}{ 7, 3}{ $10, 2$ } 
            \xbox{2}{1}{n53}{ 6, 3}{ $10, 3$ } 
            \xbox{3}{2}{n59}{ 8, 3}{ $11, 1$ } 
            \xbox{3}{2}{n60}{11, 3}{ $13, 1$ } 
            \xbox{3}{1}{n28}{10, 3}{ $13, 2$ } 
            \xbox{2}{1}{n61}{ 9, 3}{ $13, 3$ } 
            \xbox{2}{1}{n76}{12, 3}{ $H_3$ }   
            \xbox{2}{1}{n63}{13, 3}{ $16, 3$ } 
            \xbox{2}{1}{n65}{15, 3}{ $17, 2$ } 
            \xbox{2}{1}{n64}{14, 3}{ $17, 4$ } 
            \xbox{2}{1}{n66}{16, 3}{ $19, 1$ } 
            \xbox{2}{1}{n67}{17, 3}{ $19, 3$ } 
            \xbox{2}{1}{n73}{18, 3}{ $31, 1$ }
            
            \xbox{6}{5}{n18}{ 0, 4}{ $A_4$ }
            \xbox{2}{1}{n16}{ 1, 4}{ $D_4$ }
            \xbox{3}{2}{n27}{ 2, 4}{ $\D_{4,1}$ }
            \xbox{4}{3}{n40}{ 3, 4}{ $\D_{4,2}$ }
            \xbox{5}{4}{n51}{ 4, 4}{ $\D_{4,3}$ }
            \xbox{6}{4}{n17}{ 5, 4}{ $B_4$ }
            \xbox{3}{2}{n38}{ 7, 4}{ $15, 2$ }
            \xbox{2}{1}{n39}{ 6, 4}{ $17, 1$ }
            \xbox{3}{2}{n48}{ 8, 4}{ $18, 1$ }
            \xbox{2}{1}{n49}{ 9, 4}{ $21, 1$ }
            \xbox{2}{1}{n50}{10, 4}{ $22, 1$ }
            \xbox{3}{1}{n15}{11, 4}{ $F_4$ }
            \xbox{2}{1}{n54}{12, 4}{ $25, 1$ }
            \xbox{2}{1}{n55}{13, 4}{ $28, 2$ }
            \xbox{2}{1}{n56}{14, 4}{ $30, 1$ }
            \xbox{2}{1}{n57}{15, 4}{ $32, 1$ }
            \xbox{2}{1}{n58}{16, 4}{ $32, 2$ }
            \xbox{2}{1}{n72}{18, 4}{ $H_4$ }

            \xbox{7}{6}{n14}{ 0, 5}{ $A_5$ }
            \xbox{7}{5}{n13}{ 5, 5}{ $B_5$ }
            \xbox{2}{1}{n12}{ 1, 5}{ $D_5$ }
            \xbox{3}{2}{n26}{ 2, 5}{ $\D_{5,1}$ }
            \xbox{6}{5}{n47}{ 4, 5}{ $\D_{5,4}$ }
            \xbox{2}{1}{n25}{ 7, 5}{ $25, 1$ }
            \xbox{2}{1}{n35}{ 9, 5}{ $30, 1$ }
            \xbox{2}{1}{n36}{10, 5}{ $33, 1$ }
            \xbox{2}{1}{n44}{12, 5}{ $41, 1$ }
            \xbox{2}{1}{n45}{13, 5}{ $46, 1$ }
            \xbox{2}{1}{n46}{14, 5}{ $49, 1$ }

            \xbox{8}{7}{n11}{ 0, 6}{ $A_6$ }
            \xbox{2}{1}{n9}{ 1, 6}{ $D_6$ }
            \xbox{3}{2}{n24}{ 2, 6}{ $\D_{6,1}$ }
            \xbox{7}{6}{n34}{ 4, 6}{ $\D_{6,5}$ }
            \xbox{8}{6}{n10}{ 5, 6}{ $B_6$ }
            \xbox{2}{1}{n8}{ 7, 6}{ $E_6$ }
            \xbox{2}{1}{n23}{10, 6}{ $46, 1$ }
            \xbox{2}{1}{n32}{12, 6}{ $63, 1$ }
            \xbox{2}{1}{n33}{13, 6}{ $68, 1$ }

            \xbox{9}{8}{n7}{ 0, 7}{ $A_7$ }
            \xbox{2}{1}{n5}{ 1, 7}{ $D_7$ }
            \xbox{3}{2}{n22}{ 2, 7}{ $\D_{7,1}$ }
            \xbox{8}{7}{n69}{ 4, 7}{ $\D_{7,6}$ }
            \xbox{9}{7}{n6}{ 5, 7}{ $B_7$ }
            \xbox{2}{1}{n4}{10, 7}{ $E_7$ }
            \xbox{2}{1}{n21}{13, 7}{ $91, 1$ }

            \xbox{10}{9}{n1}{ 0, 8}{ $A_8$ }
            \xbox{2}{1}{n3}{ 1, 8}{ $D_8$ }
            \xbox{3}{2}{n70}{ 2, 8}{ $\D_{8,1}$ }
            \xbox{9}{8}{n71}{ 4, 8}{ $\D_{8,7}$ }
            \xbox{10}{8}{n2}{ 5, 8}{ $B_8$ }
            \xbox{2}{1}{n0}{13, 8}{ $E_8$ }

            \node at ( 3, 5) { $\dots$ };
            \node at ( 3, 6) { $\dots$ };
            \node at ( 3, 7) { $\dots$ };
            \node at ( 3, 8) { $\dots$ };

            \node at (0, 8.7){$\vdots$};
            \node at (1, 8.7){$\vdots$};
            \node at (2, 8.7){$\vdots$};
            \node at (4, 8.7){$\vdots$};
            \node at (5, 8.7){$\vdots$};

            \draw (n75.north) -- (n73.south);
            \draw (n74.north) -- (n73.south);
            \draw (n73.north) -- (n72.south);
            \draw ( n6.north) -- (n70.south);
            \draw (n22.north) -- (n70.south);
            \draw (n69.north) -- (n71.south);
            \draw ( n6.north) -- (n71.south);
            \draw ( n6.north) -- ( n2.south);
            \draw ( n7.north) -- ( n1.south);
            \draw (n10.north) -- (n22.south);
            \draw (n10.north) -- (n69.south);
            \draw (n10.north) -- ( n6.south);
            \draw (n11.north) -- ( n7.south);
            \draw (n13.north) -- (n24.south);
            \draw (n13.north) -- (n34.south);
            \draw (n13.north) -- (n10.south);
            \draw (n14.north) -- (n11.south);
            \draw (n15.north) -- (n44.south);
            \draw (n15.north) -- (n46.south);
            \draw (n15.north) -- (n36.south);
            \draw (n17.north) -- (n26.south);
            \draw (n17.north) -- (n13.south);
            \draw (n17.north) -- (n47.south);
            \draw (n18.north) -- (n14.south);
            \draw (n19.north) -- (n40.south);
            \draw (n19.north) -- (n48.south);
            \draw (n19.north) -- (n17.south);
            \draw (n19.north) -- (n50.south);
            \draw (n19.north) -- (n51.south);
            \draw (n19.north) -- (n27.south);
            \draw (n20.north) -- (n18.south);
            \draw (n21.north) -- ( n0.south);
            \draw (n22.north) -- ( n3.south);
            \draw (n23.north) -- ( n4.south);
            \draw (n24.north) -- ( n5.south);
            \draw (n24.north) -- (n22.south);
            \draw (n25.north) -- ( n8.south);
            \draw (n26.north) -- (n24.south);
            \draw (n26.north) -- ( n9.south);
            \draw (n27.north) -- (n26.south);
            \draw (n27.north) -- (n12.south);
            \draw (n28.north) -- (n15.south);
            \draw (n28.north) -- (n48.south);
            \draw (n28.north) -- (n49.south);
            \draw (n28.north) -- (n50.south);
            \draw (n28.north) -- (n54.south);
            \draw (n28.north) -- (n55.south);
            \draw (n28.north) -- (n57.south);
            \draw (n28.north) -- (n58.south);
            \draw (n29.north) -- (n40.south);
            \draw (n29.north) -- (n16.south);
            \draw (n29.north) -- (n27.south);
            \draw (n30.north) -- (n64.south);
            \draw (n30.north) -- (n65.south);
            \draw (n30.north) -- (n41.south);
            \draw (n30.north) -- (n60.south);
            \draw (n30.north) -- (n29.south);
            \draw (n30.north) -- (n19.south);
            \draw (n30.north) -- (n52.south);
            \draw (n30.north) -- (n53.south);
            \draw (n30.north) -- (n59.south);
            \draw (n30.north) -- (n28.south);
            \draw (n30.north) -- (n61.south);
            \draw (n31.north) -- (n41.south);
            \draw (n31.north) -- (n29.south);
            \draw (n31.north) -- (n20.south);
            \draw (n31.north) -- (n53.south);
            \draw (n32.north) -- (n21.south);
            \draw (n33.north) -- (n21.south);
            \draw (n34.north) -- (n69.south);
            \draw (n35.north) -- (n23.south);
            \draw (n36.north) -- (n23.south);
            \draw (n38.north) -- (n25.south);
            \draw (n39.north) -- (n25.south);
            \draw (n40.north) -- (n26.south);
            \draw (n41.north) -- (n38.south);
            \draw (n41.north) -- (n39.south);
            \draw (n41.north) -- (n40.south);
            \draw (n41.north) -- (n51.south);
            \draw (n41.north) -- (n27.south);
            \draw (n42.north) -- (n64.south);
            \draw (n42.north) -- (n65.south);
            \draw (n42.north) -- (n66.south);
            \draw (n42.north) -- (n67.south);
            \draw (n42.north) -- (n60.south);
            \draw (n42.north) -- (n52.south);
            \draw (n42.north) -- (n59.south);
            \draw (n42.north) -- (n28.south);
            \draw (n42.north) -- (n61.south);
            \draw (n42.north) -- (n63.south);
            \draw (n43.north) -- (n68.south);
            \draw (n43.north) -- (n42.south);
            \draw (n43.north) -- (n30.south);
            \draw (n43.north) -- (n62.south);
            \draw (n43.north) -- (n31.south);
            \draw (n44.north) -- (n32.south);
            \draw (n44.north) -- (n33.south);
            \draw (n45.north) -- (n32.south);
            \draw (n45.north) -- (n33.south);
            \draw (n46.north) -- (n33.south);
            \draw (n47.north) -- (n34.south);
            \draw (n48.north) -- (n35.south);
            \draw (n48.north) -- (n36.south);
            \draw (n49.north) -- (n35.south);
            \draw (n49.north) -- (n36.south);
            \draw (n50.north) -- (n36.south);
            \draw (n51.north) -- (n47.south);
            \draw (n52.north) -- (n48.south);
            \draw (n52.north) -- (n49.south);
            \draw (n52.north) -- (n38.south);
            \draw (n52.north) -- (n39.south);
            \draw (n53.north) -- (n39.south);
            \draw (n54.north) -- (n44.south);
            \draw (n54.north) -- (n45.south);
            \draw (n55.north) -- (n44.south);
            \draw (n55.north) -- (n45.south);
            \draw (n55.north) -- (n46.south);
            \draw (n56.north) -- (n45.south);
            \draw (n57.north) -- (n45.south);
            \draw (n57.north) -- (n46.south);
            \draw (n58.north) -- (n46.south);
            \draw (n59.north) -- (n48.south);
            \draw (n59.north) -- (n49.south);
            \draw (n59.north) -- (n50.south);
            \draw (n60.north) -- (n56.south);
            \draw (n60.north) -- (n49.south);
            \draw (n60.north) -- (n54.south);
            \draw (n60.north) -- (n55.south);
            \draw (n61.north) -- (n49.south);
            \draw (n61.north) -- (n50.south);
            \draw (n62.north) -- (n59.south);
            \draw (n62.north) -- (n53.south);
            \draw (n62.north) -- (n52.south);
            \draw (n62.north) -- (n61.south);
            \draw (n63.north) -- (n57.south);
            \draw (n63.north) -- (n54.south);
            \draw (n63.north) -- (n55.south);
            \draw (n64.north) -- (n56.south);
            \draw (n64.north) -- (n57.south);
            \draw (n64.north) -- (n55.south);
            \draw (n65.north) -- (n57.south);
            \draw (n65.north) -- (n58.south);
            \draw (n65.north) -- (n55.south);
            \draw (n66.north) -- (n56.south);
            \draw (n66.north) -- (n57.south);
            \draw (n67.north) -- (n57.south);
            \draw (n67.north) -- (n58.south);
            \draw (n68.north) -- (n64.south);
            \draw (n68.north) -- (n65.south);
            \draw (n68.north) -- (n66.south);
            \draw (n68.north) -- (n67.south);
            \draw (n68.north) -- (n63.south);
            \draw (n43.north) -- (n74.south);
            \draw (n43.north) -- (n75.south);
            \draw (n42.north) -- (n76.south);
            
            \draw (15, 6.5) rectangle (18, 8);
            \xbox{1}{1}{ex}{16, 7}{ $n, k$ };
            \node (hexp) at (17.5, 7.2) {\tiny $\dim \Jideal'$};
            \node (lexp) at (17.5, 6.8) {\tiny $\dim \Jideal$};
            \node[anchor=west, text width = 3cm] (nameexp) at (15, 7.7) {\tiny\baselineskip=-4pt Shorthand for $\Arr_r(n,k)$ where $r$ is the rank\par};
            \draw (hexp.west) -- (hex.east);
            \draw (lexp.west) -- (lex.east);

            \node[anchor=south, inner sep=0pt] (brace) at (ex.north) {$\overbrace{}\;$};
            
            
        \end{tikzpicture}
    \end{center}
    \caption{Hasse diagram of the restrictions of reflection
      arrangements and $\dim \Jideal$ and $\dim \Jideal'$ for all
      entries of rank greater or equal to $3$. The intensity of the
      blue shading increases with $\dim \Jideal$. }
    \label{fig:restrictions}
    \end{figure}
\end{landscape}
\restoregeometry

\bibliographystyle{siam}%
\bibliography{bibliography}%

\end{document}